\RequirePackage[l2tabu, orthodox]{nag} 

\documentclass{amsart}

\pdfoutput=1

\usepackage{bbm}
\usepackage[utf8]{inputenc} 
\usepackage[colorlinks, citecolor=red, breaklinks=true]{hyperref} 
\usepackage[stretch=10, shrink=10]{microtype} 
\usepackage{todonotes} 
\usepackage{cases}
\usepackage{amscd, amssymb}  
\usepackage{commath} 
\usepackage{mathtools} 
\usepackage{doi}
\usepackage{xcolor}
\usepackage[initials,msc-links,non-compressed-cites,nobysame]{amsrefs}[2007/10/22]
\usepackage{subcaption}
\usepackage{float}

\usepackage{paralist,enumitem}
\setlist{listparindent=0pt,parsep=3pt}
\setdefaultenum{1.}{}{}{}

\newcommand{\TitleWithUrl}[1]{\IfEmptyBibField{doi}%
  {\IfEmptyBibField{url}{\textit{#1}}%
    {\IfEmptyBibField{eprint}{\href {\BibField{url}}{\textit{#1}}}{\textit{#1}}}%
    }%
  {\href {https://doi.org/\BibField{doi}}{\textit{#1}}}}
\renewcommand{\eprint}[1]{\IfEmptyBibField{url}{\url{#1}}%
  {\href {\BibField{url}}{#1}}}

\BibSpec{article}{%
    +{}  {\PrintAuthors}                {author}
    +{,} { \TitleWithUrl}               {title}
    +{.} { }                            {part}
    +{:} { \textit}                     {subtitle}
    +{,} { \PrintContributions}         {contribution}
    +{.} { \PrintPartials}              {partial}
    +{,} { }                            {journal}
    +{}  { \textbf}                     {volume}
    +{}  { \PrintDatePV}                {date}
    +{,} { \issuetext}                  {number}
    +{,} { \eprintpages}                {pages}
    +{,} { }                            {status}
    +{,} { available at arXiv:\eprint}        {eprint}
    +{}  { \parenthesize}               {language}
    +{}  { \PrintTranslation}           {translation}
    +{;} { \PrintReprint}               {reprint}
    +{.} { }                            {note}
    +{.} {}                             {transition}
    +{}  {\SentenceSpace \PrintReviews} {review}
}

\BibSpec{collection.article}{%
    +{}  {\PrintAuthors}                {author}
    +{,} { \TitleWithUrl}                     {title}
    +{.} { }                            {part}
    +{:} { \textit}                     {subtitle}
    +{,} { \PrintContributions}         {contribution}
    +{,} { \PrintConference}            {conference}
    +{}  {\PrintBook}                   {book}
    +{,} { }                            {booktitle}
    +{,} { \PrintDateB}                 {date}
    +{,} { pp.~}                        {pages}
    +{,} { }                            {status}
    +{,} { available at \eprint}        {eprint}
    +{}  { \parenthesize}               {language}
    +{}  { \PrintTranslation}           {translation}
    +{;} { \PrintReprint}               {reprint}
    +{.} { }                            {note}
    +{.} {}                             {transition}
    +{}  {\SentenceSpace \PrintReviews} {review}
}
\BibSpec{book}{%
    +{}  {\PrintPrimary}                {transition}
    +{,} { \TitleWithUrl}               {title}
    +{.} { }                            {part}
    +{:} { \textit}                     {subtitle}
    +{,} { \PrintEdition}               {edition}
    +{}  { \PrintEditorsB}              {editor}
    +{,} { \PrintTranslatorsC}          {translator}
    +{,} { \PrintContributions}         {contribution}
    +{,} { }                            {series}
    +{,} { \voltext}                    {volume}
    +{,} { }                            {publisher}
    +{,} { }                            {organization}
    +{,} { }                            {address}
    +{,} { \PrintDateB}                 {date}
    +{,} { }                            {status}
    +{}  { \parenthesize}               {language}
    +{}  { \PrintTranslation}           {translation}
    +{;} { \PrintReprint}               {reprint}
    +{.} { }                            {note}
    +{.} {}                             {transition}
    +{}  {\SentenceSpace \PrintReviews} {review}
}

\BibSpec{thesis}{%
    +{}  {\PrintAuthors}                {author}
    +{.} { \TitleWithUrl}               {title}
    +{:} { \textit}                     {subtitle}
    +{,} { \PrintThesisType}            {type}
    +{,} { }                            {organization}
    +{} { \PrintDate}                  {date}
    +{,} { }                            {address}
    +{,} { \eprint}                     {eprint}
    +{,} { }                            {status}
    +{}  { \parenthesize}               {language}
    +{}  { \PrintTranslation}           {translation}
    +{;} { \PrintReprint}               {reprint}
    +{.} { }                            {note}
    +{.} {}                             {transition}
    +{}  {\SentenceSpace \PrintReviews} {review}
}

\newtheorem{theorem}{Theorem}[section]
\newtheorem{lemma}[theorem]{Lemma}
\newtheorem{corollary}[theorem]{Corollary}
\newtheorem{proposition}[theorem]{Proposition}

\theoremstyle{definition}
\newtheorem{definition}[theorem]{Definition}

\theoremstyle{remark}
\newtheorem{remark}[theorem]{Remark}
\newtheorem{example}[theorem]{Example}

\numberwithin{equation}{section}

\newcommand{\cratio}{\mathrm{cr}}
\newcommand{\iden}{\mathrm{id}}
\newcommand{\domain}{I}
\definecolor{darkblue}{rgb}{0,0,0.7}
\newcommand{\ii}{\mathbbm{i}}
\newcommand{\jj}{\mathbbm{j}}
\newcommand{\kk}{\mathbbm{k}}

\newsavebox\mybox
\DeclareMathOperator{\Span}{span}
\let\Im\relax
\DeclareMathOperator{\Im}{Im}
\let\Re\relax
\DeclareMathOperator{\Re}{Re}
\DeclareMathOperator{\im}{im}

\title{Periodic discrete Darboux transforms}

\author{Joseph Cho}
\address[Joseph Cho]{Institute of Discrete Mathematics and Geometry, TU Wien, Wiedner Hauptstrasse 8-10/104, 1040 Wien, Austria}
\email{jcho@geometrie.tuwien.ac.at}

\author{Katrin Leschke}
\address[Katrin Leschke]{Department of Mathematics,
  University of Leicester, University Road, Leicester LE1 7RH, United
  Kingdom}
\email{k.leschke@leicester.ac.uk }

\author{Yuta Ogata}
\address[Yuta Ogata]{Department of Mathematics, Faculty of Science, Kyoto Sangyo University, 
  Motoyama, Kamigamo, Kita-ku, Kyoto-City, 603-8555, Japan}
\email{yogata@cc.kyoto-su.ac.jp}
\begin{document}
\begin{abstract}
  We express Darboux transformations of discrete polarised curves as parallel sections of discrete connections in the quaternionic formalism. This immediately leads to the linearisation of the monodromy of the transformation. We also consider the integrable reduction to the case of discrete bicycle correspondence.  Applying our method to the case of discrete circles, we obtain closed-form discrete parametrisations of all (closed) Darboux transforms and (closed) bicycle correspondences.
\end{abstract}

\maketitle

\section{Introduction}
In computational modeling, discretisation has been central to a number of applications in the form of a polygonal mesh: For example, computer graphics uses triangular meshes to represent $3$--dimensional models; freeform architecture greatly benefits from a systematic analysis of polygonal meshes (see, for example, \cite{pottmann_architectural_2007}).
The primary objective of polygonal meshes is to approximate a given
smooth surface via polygons.

In view of  its manifold applications in computational modelling, recently the field of discrete differential geometry, with the central ethos of \emph{integrable discretisation}, has experienced a surge in interest.
In contrast to classical numerical approaches for mesh generation,  \emph{integrable discretisation} in its nascence  \cite{levi_backlund_1980, quispel_linear_1984} saught to recover the integrable system structure of smooth solitonic theory in its discrete counterparts.
As surface theory became modernised via the solitonic approach, integrable discretisation began to take shape in the form of \emph{discrete surfaces}, with discrete pseudospherical surfaces \cite{bobenko_discrete_1996} and discrete isothermic surfaces \cite{bobenko_discrete_1996-1} being the seminal examples.
These discrete surfaces with integrability approximate the smooth surfaces capably; more importantly, integrable discretisation were quickly found to possess a rich mathematical structure rivaling that of the smooth counterpart, giving birth to the field of \emph{discrete differential geometry} \cite{bobenko_discrete_2008}.
With the growth of the field, discrete differential geometry no longer merely seeks to replicate the mathematical structure of the smooth theory; the field is now quickly becoming a key ingredient in understanding the smooth theory.
For example, a solution to the Björling problem for isothermic surfaces was obtained via discrete isothermic surfaces in \cite{bucking_constructing_2016}; remarkably, discrete differential geometry also was essential to the resolution of the long standing global Bonnet problem in \cite{bobenko_compact_2021}.

Much of the interest in discrete differential geometry was centered around the \emph{local theory}; on the contrary, the global theory of discrete surfaces from the viewpoint of integrable discretisations has received comparatively less interest.
In this work, we seek to focus on the global aspects of discrete differential geometry.

As a starting point, we will investigate periodic Darboux transforms of discrete polarised curves.
Darboux transformations of smooth polarised curves were defined in \cite{burstall_semi-discrete_2016} in the context of interpreting the semi-discrete isothermic surfaces \cite{muller_semi-discrete_2013} in terms of transformation theory.
In fact, it has been investigated that the integrable reductions of such Darboux transformations include the \emph{bicycle correspondences}, a mathematical model of the pair of tire tracks of a bicycle, with various connections to the Hashimoto or the smoke ring flow and the filament equations \cite{hasimoto_soliton_1972, tabachnikov_bicycle_2017, bor_tire_2020} and the modified Korteweg--de Vries equations \cite{cho_infinitesimal_2020}.

The integrable discretisation of the Darboux transformation was obtained in \cite{cho_discrete_2021-1} in the case of plane curves motivated by discrete isothermic surfaces \cite{bobenko_discrete_1996-1}; meanwhile, the monodromy of discrete bicycle correspondences was investigated in \cite{tabachnikov_discrete_2013} (see also \cite{matthaus_discrete_2003}) while that of the discrete Hashimoto flows was examined in \cite{pinkall_new_2007, hoffmann_discrete_2008}.
Building on these results, we will investigate the monodromy of discrete Darboux transformations via the gauge theoretic approach to integrability, where the zero curvature formalism is expressed via the existence of a $1$-parameter family of (flat) connections on the trivial bundle \cite{burstall_conformal_2010, burstall_isothermic_2011}, as the approach has been shown to be amenable to discretisations \cite{burstall_discrete_2014, burstall_discrete_2015, burstall_discrete_2018, burstall_discrete_2020, pember_discrete_2021}.

In Section~\ref{sect:two}, we reinterpret the Darboux transformations of smooth polarised curves of \cite{burstall_semi-discrete_2016} in the quaternionic setting to serve as a motivation for the discrete case.
The Darboux transformations of smooth polarised curves can be expressed via a Riccati-type equation \cite{hertrich-jeromin_remarks_1997, burstall_semi-discrete_2016}; via a suitably defined $1$-parameter family of (flat) connections, we will show in Theorem~\ref{thm:one} that Darboux tranformations can be characterised as the parallel sections of the connection, recovering a quaternionic analogue of the result in \cite{burstall_semi-discrete_2016}.
The interpretation of Darboux transformations via parallel sections is key to linearising the monodromy problem, a process that we explain in Section~\ref{subsect:twotwo}.
Finally, we finish the introductory section by considering the integrable reduction to the bicycle correspondences and the bicycle monodromy in Section~\ref{subsect:twothree}.
In fact, the quaternionic formalism to Darboux transformations allows us to obtain closed-form parametrisations for the transforms of a circle in Examples~\ref{exam:circle} and \ref{exam:smoothB} as the quaternionic approach yields second-order ordinary differential equations with constant coefficients from the linearisation of the Riccati-type equation (see Remark~\ref{rema:non-constant}).
As we will see, this approach will also allow us to obtain the \emph{closed-form discrete parametrisations}, a comparatively rare result in the discrete theory.

The next Section~\ref{sect:three} is devoted to the Darboux transformations of discrete polarised curves via characterised as parallel sections of discrete (flat) connections.
Motivated by the Darboux transformations of smooth curves, we define the discrete connections associated with discrete polarised curves in Definition~\ref{def:discFlat}, and show in Theorem~\ref{thm:discDar} that the parallel sections correspond to discrete Darboux transformations obtained via a discrete Riccati-type equation, coming from the well-known cross-ratios condition of discrete isothermic surfaces \cite{bobenko_discrete_1996-1}.
The discrete connections approach immediately yields the linearisation of the discrete monodromy problem (see Section~\ref{subsect:discMon}).
Then in Theorem~\ref{thm:discB}, we obtain the integrable reduction to the case of discrete bicycle correspondences.
In Examples~\ref{exam:disc} and \ref{exam:disc2}, we test the robustness of our discretisation by considering the case of discrete circles.
Surprisingly, our methods efficiently yield closed-form discrete parametrisations of: the Darboux transformations, the closed Darboux transformations, the bicycle correspondences, and the closed bicycle correspondences of the discrete circle (see Figure~\ref{fig:discB3}).
\begin{figure}
	\centering
	\begin{minipage}{0.25\textwidth}
		\includegraphics[width=\linewidth]{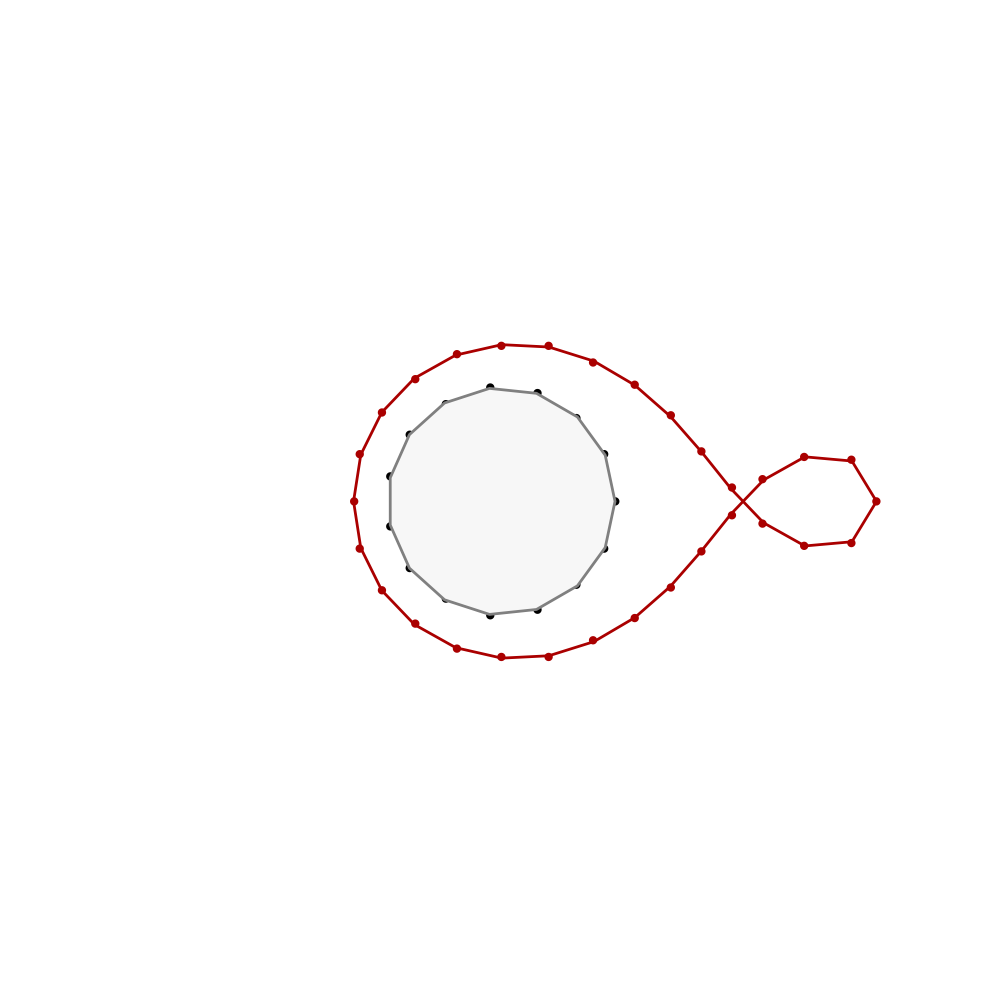}
	\end{minipage}
	\quad
	\begin{minipage}{0.25\textwidth}
		\includegraphics[width=\linewidth]{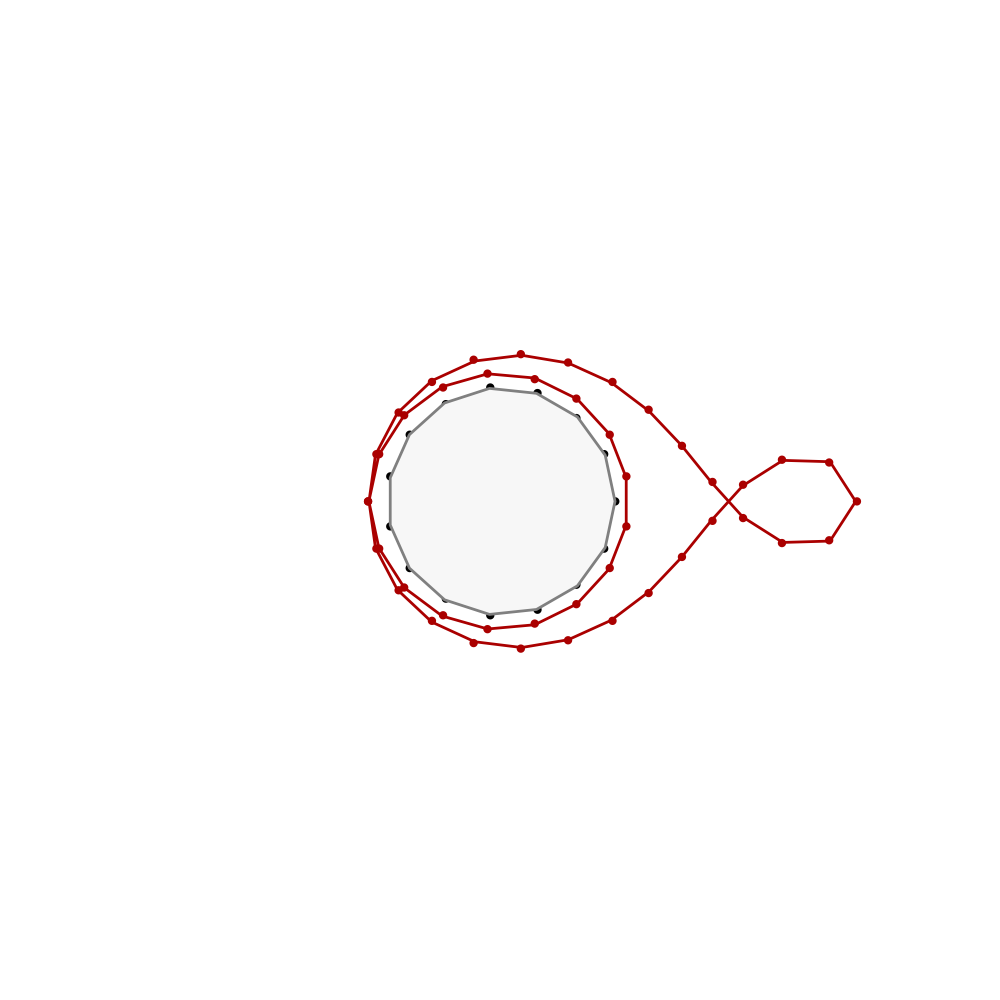}
	\end{minipage}
	\quad
	\begin{minipage}{0.25\textwidth}
		\includegraphics[width=\linewidth]{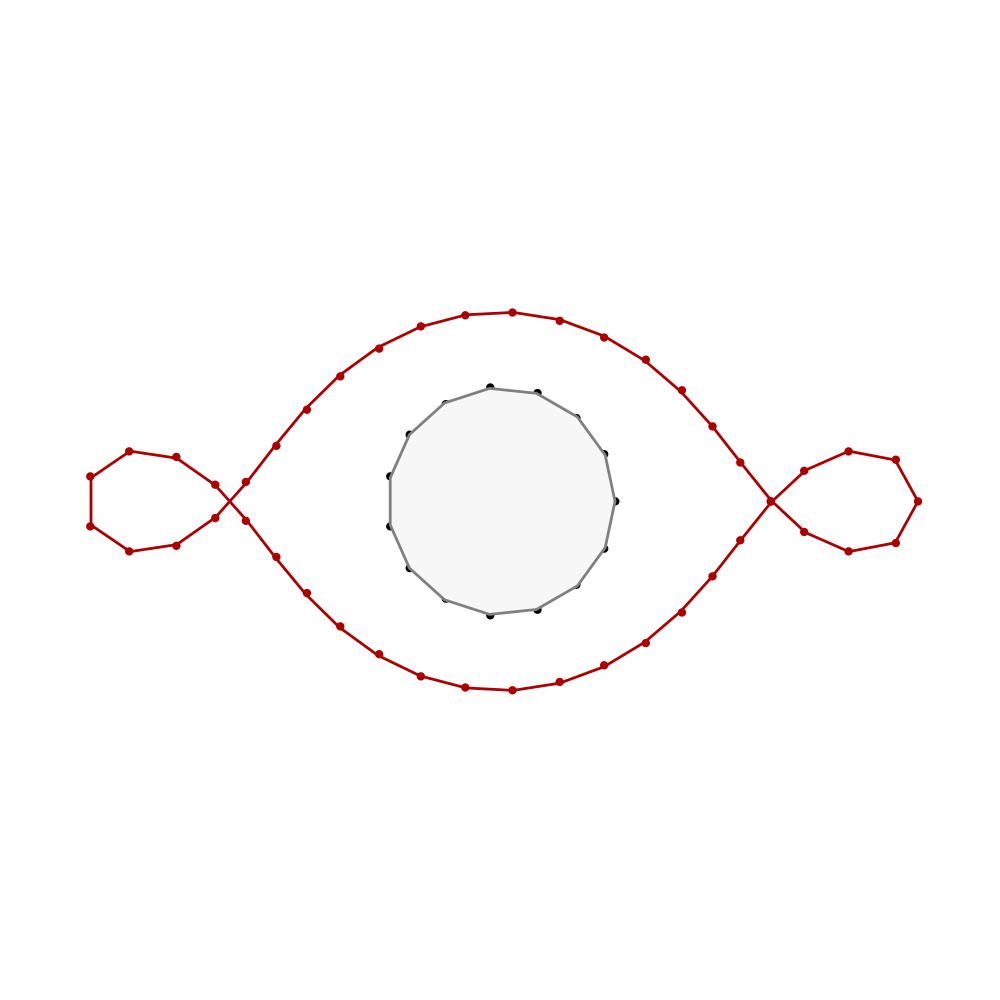}
	\end{minipage}
	\\
	\begin{minipage}{0.25\textwidth}
		\includegraphics[width=\linewidth]{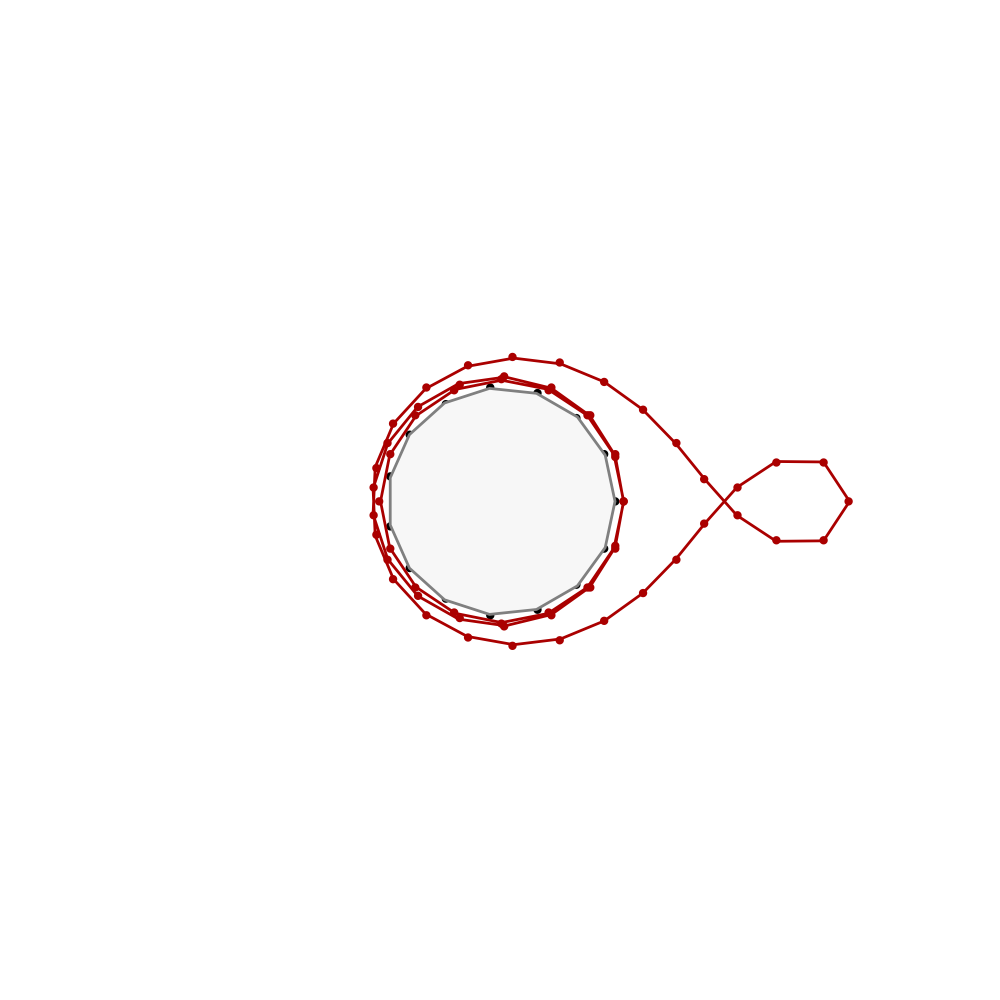}
	\end{minipage}
	\quad
	\begin{minipage}{0.25\textwidth}
		\includegraphics[width=\linewidth]{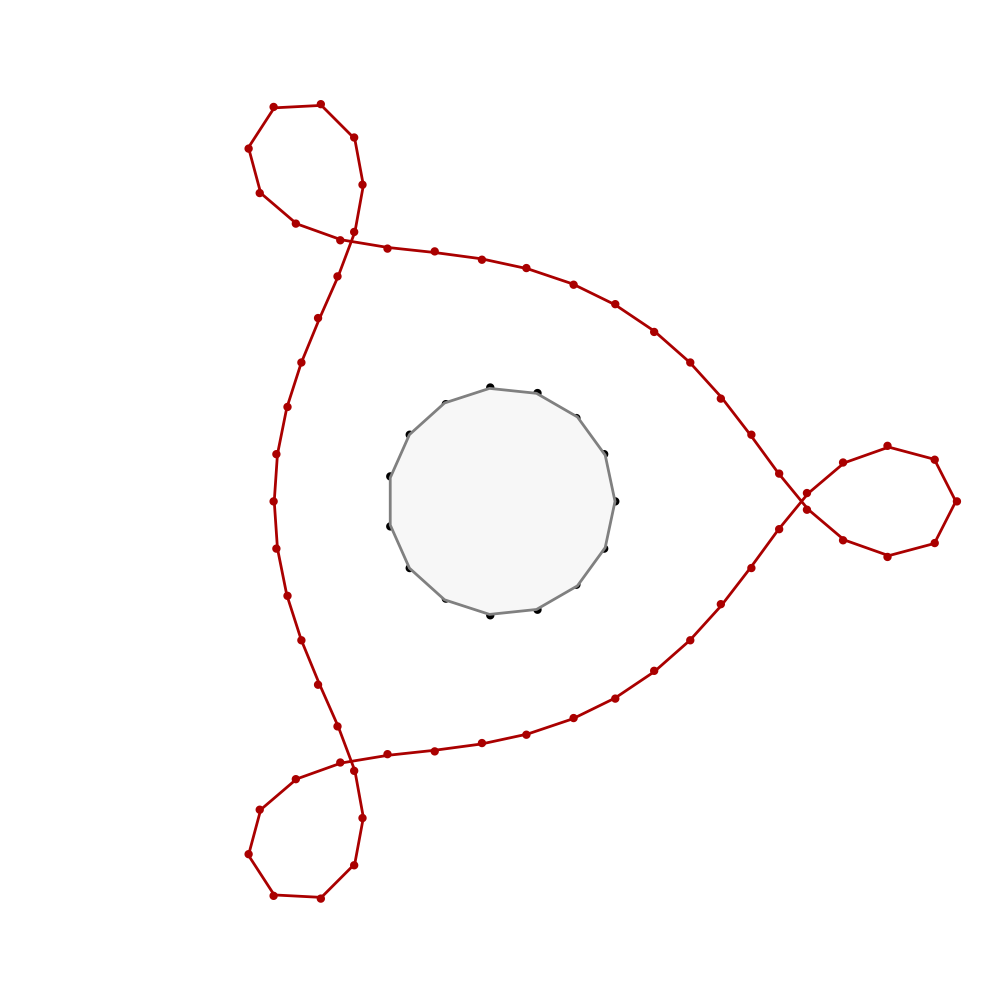}
	\end{minipage}
	\quad
	\begin{minipage}{0.25\textwidth}
		\includegraphics[width=\linewidth]{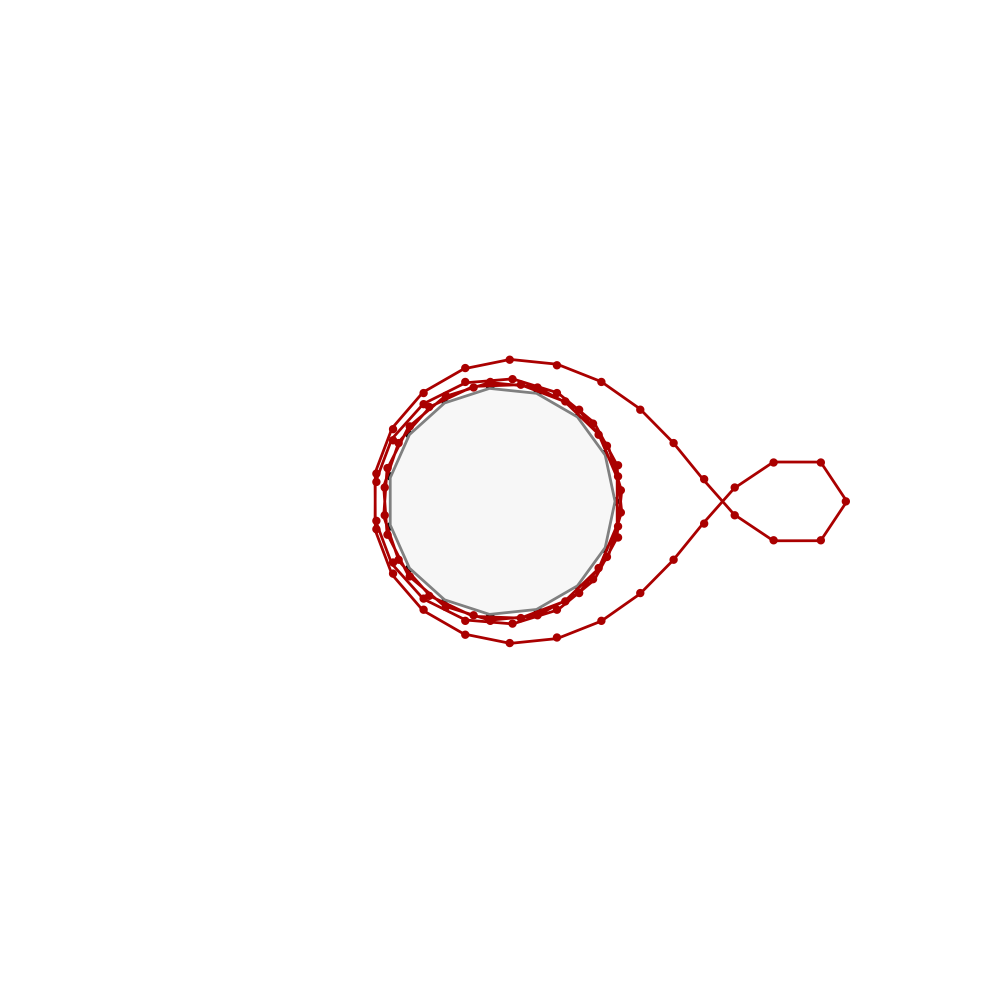}
	\end{minipage}
	\\
	\begin{minipage}{0.25\textwidth}
		\includegraphics[width=\linewidth]{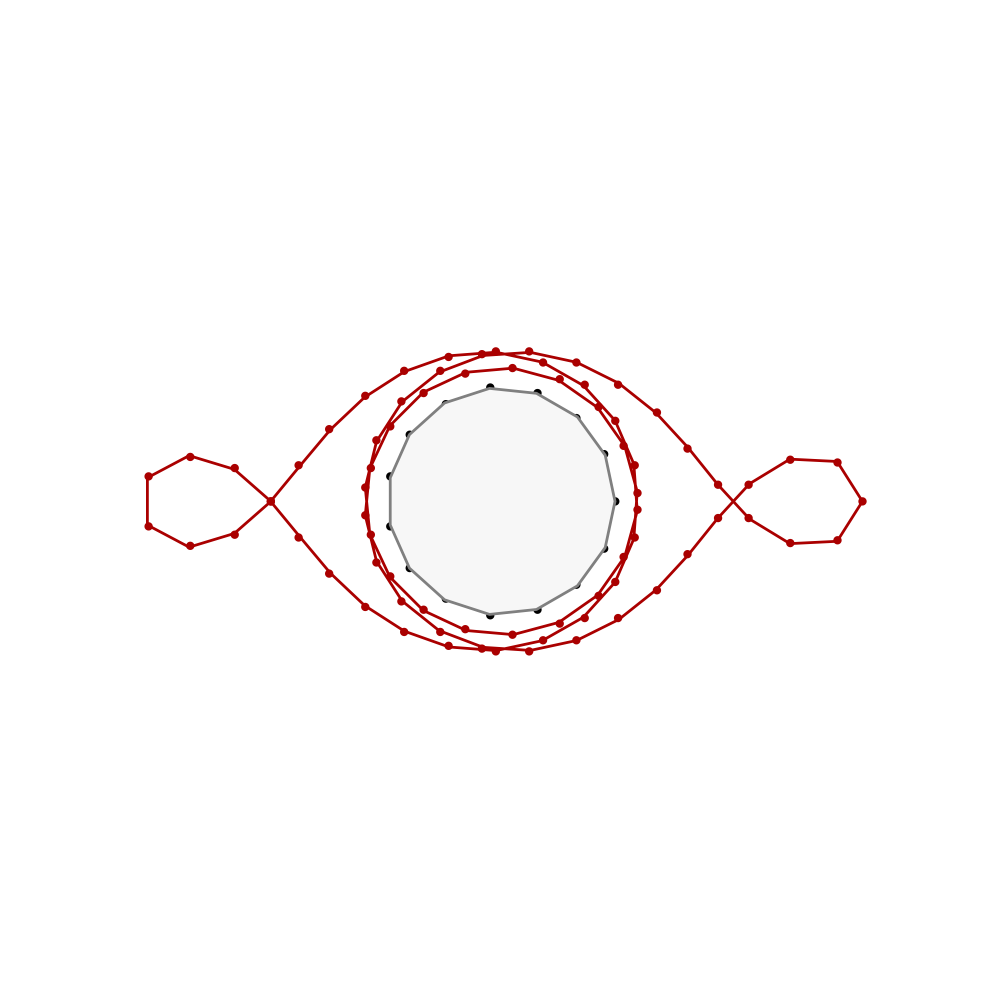}
	\end{minipage}
	\quad
	\begin{minipage}{0.25\textwidth}
		\includegraphics[width=\linewidth]{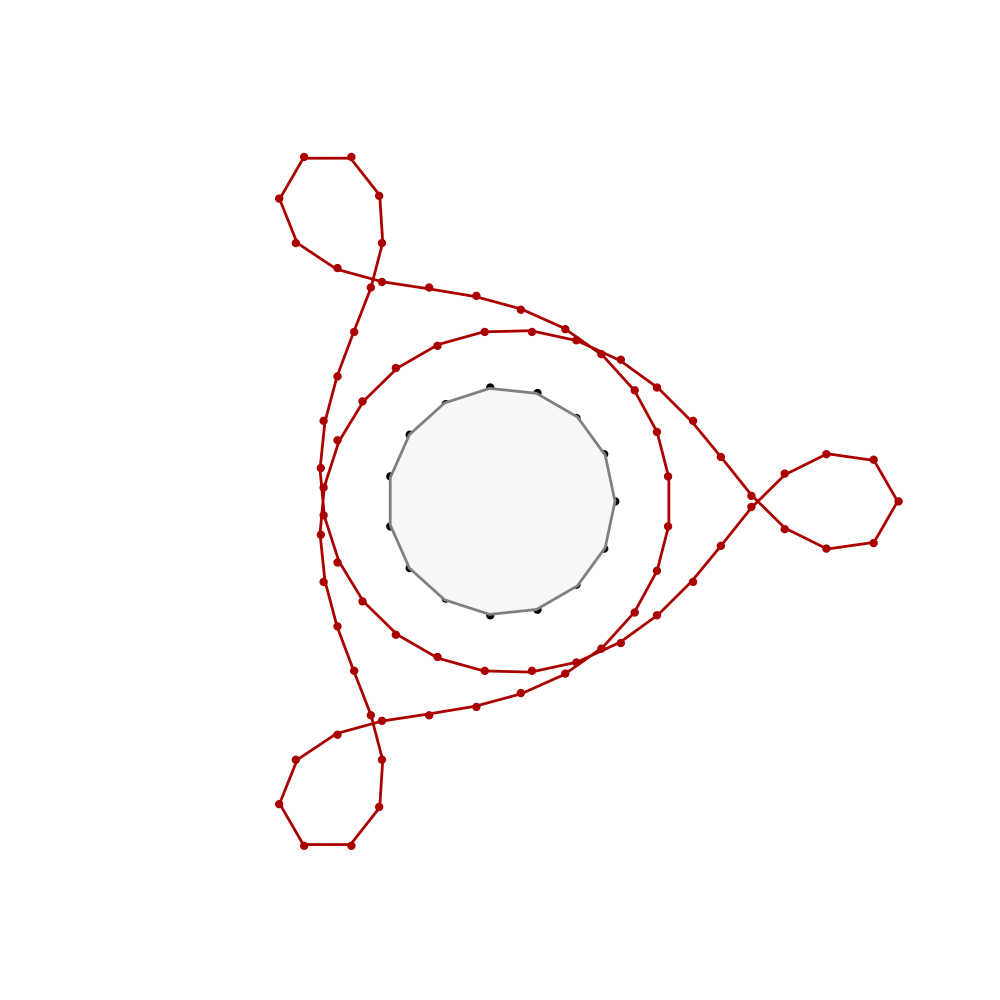}
	\end{minipage}
	\quad
	\begin{minipage}{0.25\textwidth}
		\includegraphics[width=\linewidth]{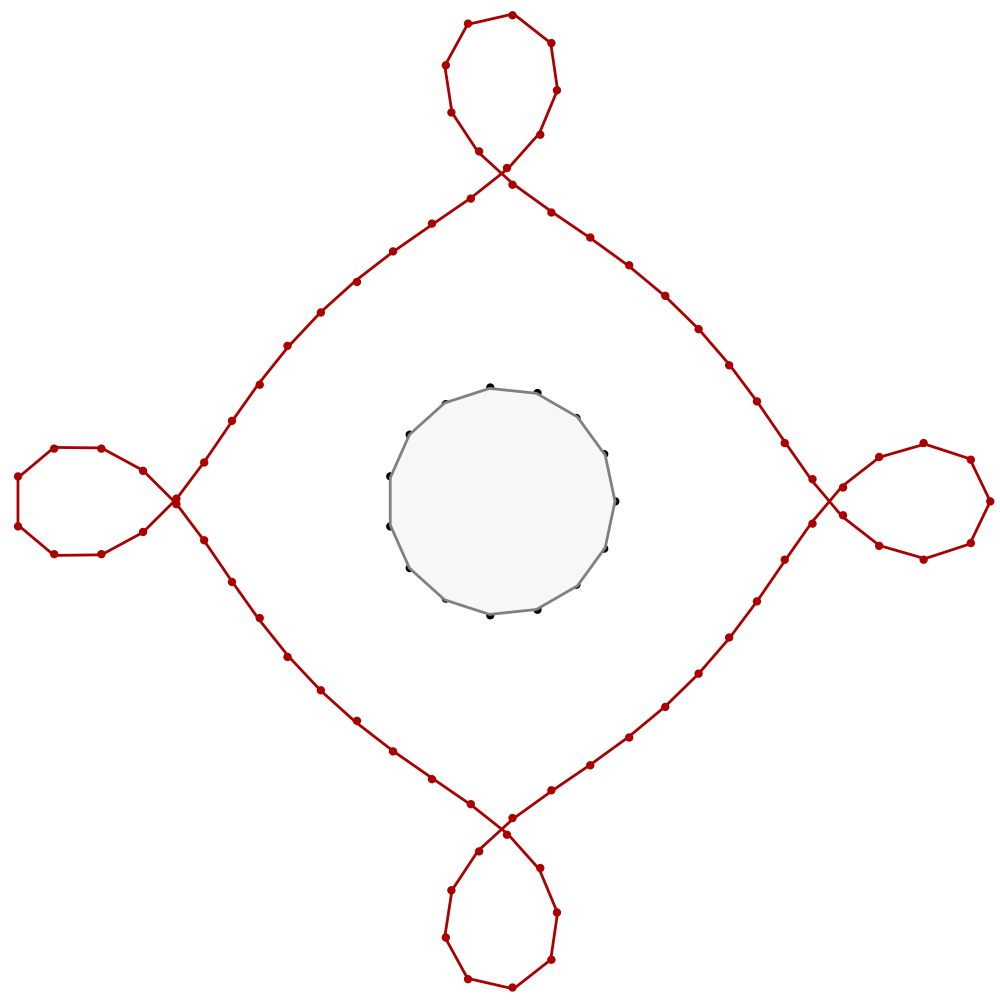}
	\end{minipage}
	\caption{Closed Darboux transformations (or bicycle correspondences) of a discrete circle.}
	\label{fig:discB3}
      \end{figure}
      
Integrable system structures are at the core of many problems in physics,
chemistry and biology; for example the Korteweg-de Vries equation models
waves on shallow water. Our results can be viewed as  prototypes on how to
obtain efficient numerical \emph{periodic} solutions in terms of
recurrence equations  by discretising the integrable
system structure.

\section{Darboux transformations of smooth space curves}\label{sect:two}
In this section, we adapt the Darboux transformations of smooth polarised curves in $\mathbb{R}^n$ from \cite{burstall_semi-discrete_2016} to the special case of $3$--space $\mathbb{R}^3$ and $4$--space $\mathbb{R}^4$ using a quaternionic formalism, with an eye on efficiently obtaining explicit parametrisations of transformations. (For details on the quaternionic setting, we refer the readers to works such as \cite{burstall_conformal_2002, hertrich-jeromin_introduction_2003}.)

Recall that the space of quaternions is given by 
	\[
		\mathbb{H} = \Span_\mathbb{R}\{1, \ii, \jj, \kk\}
	\]
where $\ii^2 = \jj^2 = \kk^2 = \ii\jj\kk =-1$ so that the multiplication is not commutative.
We identify the $4$--space with the quaternions, while we identify the $3$-space with the imaginary quaternions
	\[
		\Im \mathbb{H} = \Span_\mathbb{R}\{\ii, \jj, \kk\}.
	\]
Under the identification, we have
	\[
		\Re (a \overline{b}) = \langle a, b\rangle,
	\]
where the standard Euclidean inner product in $4$--space is denoted by $\langle \cdot, \cdot\rangle$, and the norm by $| \cdot |$. 

\subsection{Darboux transformations via parallel sections}
Let $I \subset \mathbb{R}$ be a smooth interval, $q$ be a non-vanishing \emph{real} quadratic differential acting as polarisation on $I$.
We will refer to the pair $(I, q)$ as \emph{polarised domain} as in \citelist{\cite{musso_laguerre_1999}*{Definition 2.1} \cite{hertrich-jeromin_mobius_2001}*{p.\ 190}}.
Suppose now that a regular curve $x : (I,q) \to \mathbb{R}^4 \cong \mathbb{H}$ is defined on the polarised domain.
Then a curve $x^d : I \to  \mathbb{H}$ is called a \emph{dual curve} \cite[Definition 3.2]{burstall_semi-discrete_2016} if
	\[
		\dif{x} \dif{x}^d = q.
	\]
	
\begin{remark}\label{rema:parameter}
	When we need to fix a parameter $t$ of the domain $I$ to consider explicit examples, we will define a non-vanishing $m : I \to \mathbb{R}$ by
		\[
			q = \frac{1}{m} \dif{t}^2.
		\]
\end{remark}
With the notion of duality, the \emph{Darboux transform} $\hat{x} : (I, q) \to \mathbb{H}$ of $x$ with spectral parameter $\mu$ is given by a Riccati equation in \cite[Equation (2.9)]{burstall_semi-discrete_2016}:
	\begin{equation}\label{eqn:smriccati}
		\dif{\hat{x}} = \mu (\hat{x} - x) \dif{x}^d (\hat{x} - x) = \mu T \dif{x}^d T
	\end{equation}
for a real constant $\mu$ and $T := \hat{x} - x$, and we call $x, \hat{x} : (I, q) \to \mathbb{H}$ a \emph{Darboux pair}.

\begin{remark}\label{rema:circleCongruence}
We note that as shown in \cite[Definition and Lemma 2.2]{burstall_semi-discrete_2016}, the Riccati equation \eqref{eqn:smriccati} is a reformulation of the \emph{tangential cross-ratios}, and implies that a Darboux pair is a Ribaucour pair, namely, they must envelop a common circle congruence.
\end{remark}

Darboux pairs of polarised curves are Möbius invariant notions; therefore, to view the transformation within the realm of conformal geometry, we consider
	\[
		\mathbb{H} \cup \{ \infty\} \cong \mathbb{H}\mathbb{P}^1 := \mathbb{P}(\mathbb{H}^2)
	\]
as the model for the conformal $4$-sphere, where we view $\mathbb{H}^2$ as a quaternionic right vector space.
In this paper, we take advantage of the Möbius invariance and take affine coordinates to associate the conformal $4$-sphere $\mathbb{H}\mathbb{P}^1$ with points in $\mathbb{R}^4 \cong \mathbb{H}$ via
	\[
		\mathbb{H} \ni x \sim L := \psi \mathbb{H} := \begin{pmatrix} x \\ 1 \end{pmatrix} \mathbb{H} \in \mathbb{H}\mathbb{P}^1.
	\]
Therefore, any polarised space curve is now represented as $L : (I, q) \to \mathbb{H}\mathbb{P}^1$, also considered as a $1$-dimensional subbundle of the trivial bundle $\underline{\mathbb{H}}^2 := I \times \mathbb{H}^2$.

Under this setting, we now aim to understand how the Darboux transformations of polarised curves can be interpreted in terms of parallel sections of flat connections defined on the trivial bundle $\underline{\mathbb{H}}^2$, recovering the quaternionic analogue of the result in \cite[Definition and Corollary 2.5]{burstall_semi-discrete_2016}.

To do this, consider a family of (flat) connections $\mathcal{D}_\lambda$ defined on the trivial bundle $\underline{\mathbb{H}}^2$ given by
	\begin{equation}\label{eqn:smoothcalD}
		\mathcal{D}_\lambda := \dif{} + \begin{pmatrix} 0 & \dif x \\ \lambda \dif x^d & 0 \end{pmatrix}, \quad \lambda \in \mathbb{R}.
	\end{equation}

\begin{remark}
	The family of connections $\mathcal{D}_\lambda$ is trivially flat as the domain is $1$-dimensional; however, the ground for the emphasis on the flatness is twofold: to mirror the integrable structure of isothermic surfaces in the polarised curve theory, and to note that the parallel sections are well-defined.
\end{remark}
Then we have that $\phi := \begin{pmatrix} \alpha \\ \beta \end{pmatrix}$ is a parallel section of $\mathcal{D}_\mu$ for some $\mu \in \mathbb{R}$, that is, $\mathcal{D}_\mu \phi = 0$,
if and only if
	\begin{equation}\label{eqn:dmupar}
		\dif{\begin{pmatrix} \alpha \\ \beta \end{pmatrix}} = - \begin{pmatrix} \dif x \, \beta\\ \mu \dif x^d \, \alpha \end{pmatrix}.
	\end{equation}
Under this setting, the $\mathcal{D}_\mu$--parallel sections can be characterised as follows:
\begin{lemma}[cf.\ {\cite[Corollary~5.4.4]{hertrich-jeromin_introduction_2003}}]
	Given a polarised curve $x :  (I, q) \to \mathbb{H}$, we have $\hat{x} := x + \alpha \beta^{-1}$ is a Darboux transform of $x$ with parameter $\mu$ if and only if $\phi := \begin{psmallmatrix} \alpha \\ \beta \end{psmallmatrix}$ is $\mathcal{D}_\mu$--parallel.
\end{lemma}
\begin{proof}
	First, assuming that $\phi$ is $\mathcal{D}_\mu$--parallel, define $\hat{x} := x + \alpha\beta^{-1}$.
	Then it is straightforward to see via the differential equations on $\alpha$ and $\beta$ \eqref{eqn:dmupar} that
		\[
			\dif{\hat{x}} = \dif{x} + \dif \alpha\, \beta^{-1} - \alpha \beta^{-1} \dif \beta \,\beta^{-1} =  \mu (\hat{x} - x)\dif{x}^d  (\hat{x} - x),
		\]
	so that $\hat{x}$ solves the Riccati equation \eqref{eqn:smriccati}.

	On the other hand, let $\hat{x}$ be a solution to the Riccati equation \eqref{eqn:smriccati}.
	Set $T := \hat{x} - x$ and define $\beta$ so that $\beta$ solves
		\[
			(\dif{} + T^{-1} \dif \hat{x}) \beta = 0.
		\]
	Putting $\alpha := T \beta$, we have
		\[
			0 = \dif \beta + T^{-1} \dif{\hat{x}} \, \beta = \dif \beta + \mu \dif{x}^d \, \alpha,
		\]
	while
		\[
			0 = (\dif{\hat{x}} - \mu T \dif{x}^d \, T)\beta = \dif{\hat{x}} \,\beta - \mu T \dif{x}^d\, \alpha = \dif{x}\,\beta + \dif{T}\,\beta + T \dif{\beta} = \dif{x}\,\beta + \dif{\alpha}.
		\]
	Therefore, $\phi$ is $\mathcal{D}_\mu$--parallel.
\end{proof}

Now consider the gauge transformation
	\[
		\dif{}_\lambda := \mathcal{G} \bullet \mathcal{D}_\lambda,
	\]
where $\mathcal{G} = ( e \: \psi)$ for $e = \begin{psmallmatrix} 1 \\ 0 \end{psmallmatrix}$.
This gives a $1$-parameter family of (flat) connections
	\[
		\dif{}_\lambda = \dif{} + \lambda \eta,
	\]
with
	\[
		\eta := \begin{pmatrix} x \dif{x}^d & - x \dif{x}^d x\\
				\dif{x}^d & - \dif{x}^d x \end{pmatrix},
	\]
satisfying $\ker \eta =\im \eta = \psi \mathbb{H}$.
Then $\phi = \begin{psmallmatrix} \alpha \\ \beta \end{psmallmatrix}$ is $\mathcal{D}_\mu$--parallel if and only if $\varphi := \mathcal{G} \phi = e \alpha + \psi \beta = \hat{\psi}\beta$ is $\dif{}_\mu$--parallel where
	\[
		\hat{L} := \hat{\psi}\mathbb{H} = \begin{pmatrix} \hat{x} \\ 1 \end{pmatrix}\mathbb{H} = \begin{pmatrix} x + \alpha \beta^{-1} \\ 1 \end{pmatrix}\mathbb{H}.
	\]
Therefore, we conclude:
\begin{theorem}[cf.\ {\cite[Lemma~5.4.12]{hertrich-jeromin_introduction_2003}}]\label{thm:one}
	The Darboux transforms of a polarised curve $x : (I, q) \to \mathbb{H}$ with parameter $\mu$ are given by the  $\dif{}_\mu$--parallel sections.
\end{theorem}
	
\begin{remark}
  The family of connections $\dif{}_\lambda$ defined on $\underline{\mathbb{H}}^2$ is the quaternionic analogue of the family of connections defined on $I \times \mathbb{R}^{n+1,1}$ introduced in \cite[Equation (2.10)]{burstall_semi-discrete_2016}.
\end{remark}

We can identify the sufficient condition for the Darboux transform $\hat{x}$ of a polarised curve in any $3$-sphere to take values again in the same $3$-sphere:
\begin{lemma}[cf.\ {\cite[Lemma~5.4.16]{hertrich-jeromin_introduction_2003}}]\label{lemm:r3condition}
	Given a polarised curve $x : (I, q) \to S^3$ in some $3$-sphere $S^3 \subset S^4$  with associated connection $\dif{}_\lambda$, let $\varphi = \hat{\psi} \beta$ be $\dif{}_\mu$--parallel.
	Then $\hat{x}: (I, q) \to S^3$ if and only if $\hat{x}(t_0) \in S^3$ for some $t_0 \in I$.
\end{lemma}
\begin{proof}
	Since the necesscity is obvious, we show the sufficiency.
	Applying a suitable stereographic projection to the $S^3$, we will prove the statement for curves in $\Im \mathbb{H} \cong \mathbb{R}^3$.
	Now consider the hermitian form
		\begin{equation}\label{eqn:hermform}
			\left( \begin{pmatrix} a\\ b \end{pmatrix}, \begin{pmatrix} c\\ d \end{pmatrix}\right) = \bar a d + \bar b c,
		\end{equation}
	for $a, b, c, d \in \mathbb{H}$.
	Then we have that $(\phi\lambda, \tilde\phi\tilde\lambda) = \bar\lambda(\phi, \tilde\phi) \tilde \lambda$ for $\lambda, \tilde{\lambda} \in \mathbb{H}$ and $\phi, \tilde{\phi} \in \mathbb{H}^2$; furthermore, it is straightforward to check that $\left( \begin{psmallmatrix} \alpha\\ 1 \end{psmallmatrix}, \begin{psmallmatrix} \alpha\\ 1 \end{psmallmatrix}\right) = 0$ if and only if $\Re \alpha = 0$.

	Thus, if $x$ is as given, then we have $(\psi, \psi) = 0$ for $\psi = \begin{psmallmatrix} x \\ 1 \end{psmallmatrix}$.
	Now since $\varphi = e \alpha + \psi \beta$ is $\dif{}_\mu$--parallel, we have
		\[
			\dif{((\varphi, \varphi))} = -\mu \left( (\eta\varphi , \varphi) + (\varphi, \eta\varphi)\right) = 0,
		\] 
	where we used that $\eta\varphi = \eta e\alpha = \psi \dif{x}^d\,\alpha$.
	Thus, $(\varphi,\varphi)$ is constant.
	
	If at $t = t_0$, we have $\hat{x} = x + \alpha \beta^{-1}$ is pure imaginary, then $\alpha \beta^{-1}$ must also be pure imaginary for $t = t_0$, so that
		\[
			(\varphi,\varphi) = \bar\alpha\beta +\bar\beta \alpha =
			2\Re(\alpha\bar\beta) = 2\beta\bar\beta\Re(\alpha\beta^{-1}) =0.
		\]
	Hence, $(\varphi,\varphi) \equiv 0$ on $I$, and
		\[
			(\hat \psi, \hat\psi) = \bar\beta^{-1} (\varphi,\varphi)\beta^{-1} \equiv 0,
		\]
	giving us the desired conclusion.
\end{proof}


In fact, a similar result holds for curves in any $2$-sphere:
\begin{corollary}\label{cor:2sphere}
	Let $x : (I, q) \to S^2$ be a curve into a $2$-sphere $S^2$. Then the Darboux transform $\hat{x}$ takes values in the same $2$-sphere if and only if $x(t_0) \in S^2$ for some $t_0 \in I$.
\end{corollary}
\begin{proof}
	Viewing the given $2$-sphere $S^2$ as the intersection of $1$-parameter family of $3$-spheres, also called the elliptic sphere pencil of $3$-spheres (see \cite[Definition 1.2.3]{hertrich-jeromin_introduction_2003} for example), if $\hat{x}(t_0) \in S^2$, then $\hat{x}(t_0)$ takes value in every $3$-sphere of the elliptic sphere pencil.
	Thus Lemma~\ref{lemm:r3condition} implies that $\hat{x}$ must be in the intersection of all $3$-spheres in the elliptic sphere pencil, the starting $2$-sphere $S^2$.
\end{proof}


Given a polarised curve $x$ with its associated connection $\dif{}_\lambda$, the next proposition shows that one can gauge $\dif{}_\lambda$ to obtain the $1$-parameter family of  connections associated with the Darboux transform $\hat{x}$ of $x$.
\begin{proposition}\label{prop:gauge}
	Let $L, \hat{L} : (I, \frac{\dif{t}^2}{m}) \to \mathbb{HP}^1$ be a Darboux pair with spectral parameter $\mu$, with respective associated connections $\dif{}_\lambda$ and $\hat{\dif{}}_\lambda$.
	For the splitting $\underline{\mathbb{H}}^2 = L \oplus \hat{L}$, denote by $\pi$ and $\hat{\pi}$ the projections onto $L$ and $\hat{L}$ respectively.
	Defining
		\[
			r_\lambda := \pi + \frac{\mu - \lambda}{\mu} \hat{\pi},
		\]
	one has that
		\[
			\hat{\dif{}}_\lambda = r_\lambda \bullet \dif{}_\lambda.
		\]
\end{proposition}

\begin{proof}
	 As $L$ and $\hat{L}$ are Darboux pair with spectral parameter $\mu$, we have $\dif{}_\mu$--parallel $\varphi$ and $\hat{\dif{}}_\mu$--parallel $\hat{\varphi}$ such that $L = \hat{\varphi} \mathbb{H}$ while $\hat{L} = \varphi \mathbb{H}$.
	Thus, we only need to show that $\hat{\dif{}}_\lambda$ and $r_\lambda \bullet \dif{}_\lambda$ agree on $\varphi$ and $\hat{\varphi}$.
	Using the fact that $\dif{}_\mu \varphi = 0$ and $\hat{\eta}\varphi = 0$, we have that
		\[
			(r_\lambda \bullet \dif{}_\lambda) \varphi
				= r_\lambda (\dif{\varphi} + \eta \varphi \lambda)  \frac{\mu}{\mu - \lambda} 
				= -\eta \varphi \mu = \dif{\varphi} = \hat{\dif{}}_\lambda \varphi
		\]
	while $\hat{\dif{}}_\mu \hat{\varphi} = 0$ and $\eta\hat{\varphi} = 0$ implies that
		\[
			(r_\lambda \bullet \dif{}_\lambda) \hat{\varphi}
				= r_\lambda (\dif{\hat{\varphi}})
				= r_\lambda (-\mu\hat{\eta}\hat{\varphi})
				= \hat{\eta} \hat{\varphi} (\lambda - \mu)
				= \hat{\dif{}}_\lambda \hat{\varphi},
		\]
	giving us the desired conclusion.
\end{proof}

\begin{remark}
	The gauge $r_\lambda^{-1}$ has a simple pole at $\lambda = \mu$; however, the above Proposition~\ref{prop:gauge} shows that $\hat{\dif{}}_\lambda$ is real--analytic across $\lambda = \mu$.
\end{remark}

\subsection{Monodromy of Darboux transforms}\label{subsect:twotwo}
The current setup of treating Darboux transforms as parallel sections of a connection allows us to reduce the problem of monodromy to finding \emph{sections with multipliers}:
Assuming that a polarised space curve $x$ with associated connection $\dif{}_\lambda$ has period $M$, then its Darboux transform $\hat{x}$ with spectral parameter $\mu$ has period $M$ if and only if $\dif{}_\mu$--parallel $\varphi$ is a \emph{section with multiplier}, that is, $\varphi(t + M) = \varphi(t) h$ where $h \in \mathbb{H}_*$.
Along with the fact that parallel sections of $\dif{}_\mu$ for some fixed spectral parameter $\mu$ form a vector space, one can calculate the monodromy of Darboux transforms (see Figures~\ref{fig:DTsmooth1} and \ref{fig:DTsmooth2}).
We will illustrate this explicitly with the next example of a circle.

\begin{figure}
	\centering
	\begin{minipage}{0.45\textwidth}
		\includegraphics[width=\linewidth]{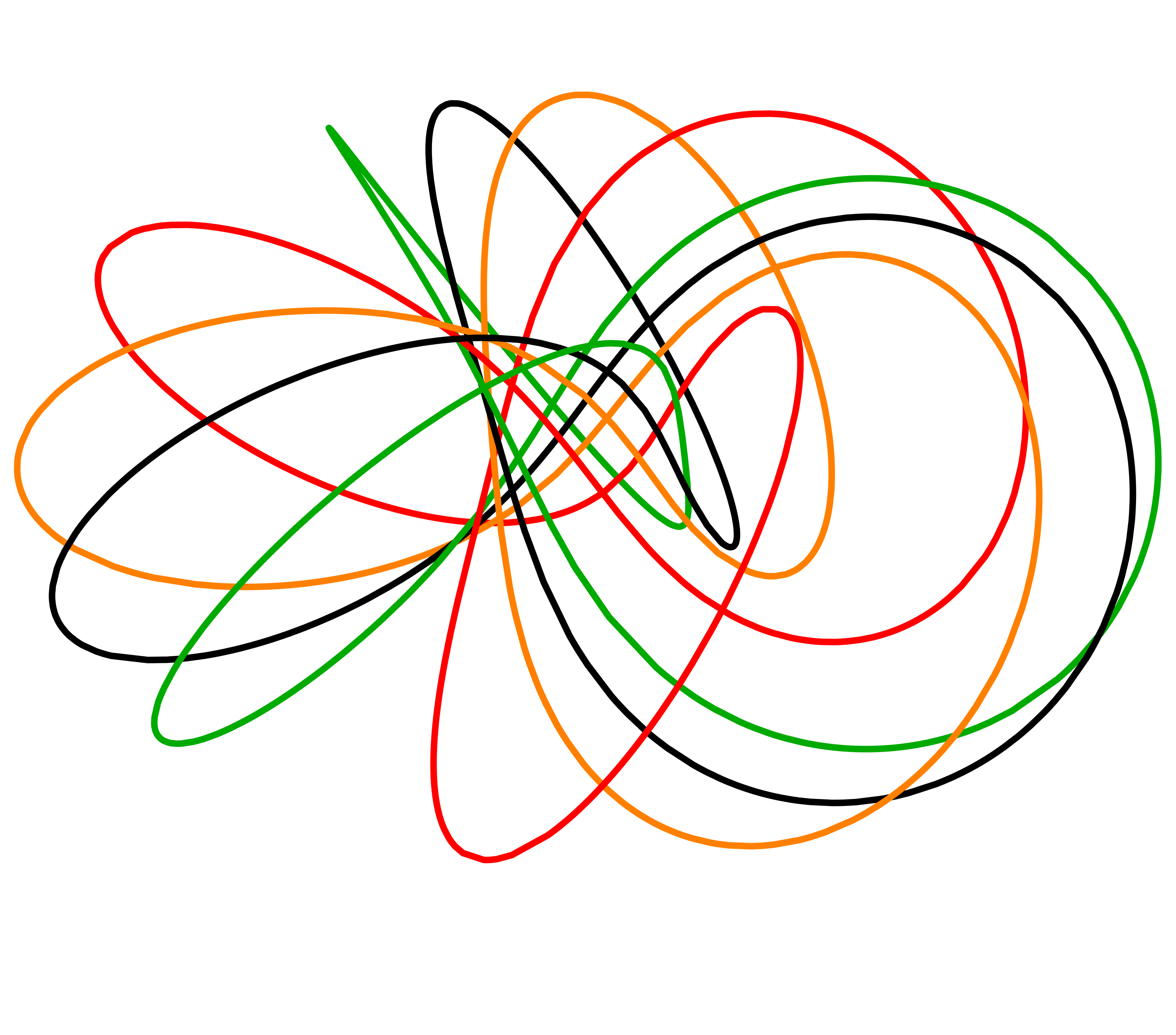}
	\end{minipage}
	\begin{minipage}{0.45\textwidth}
		\includegraphics[width=\linewidth]{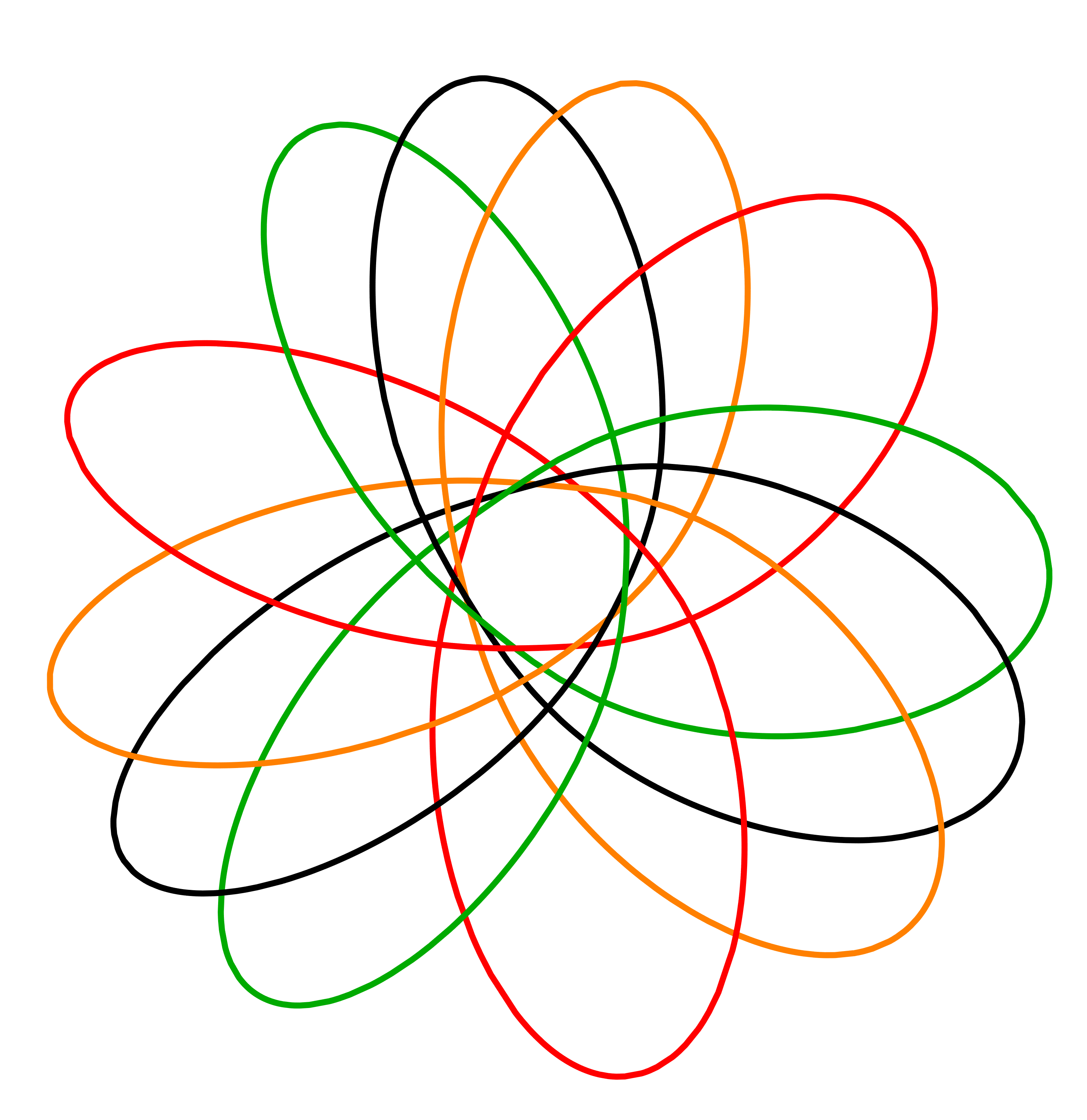}
	\end{minipage}
	\caption{Original $(2,3)$-torus knot (in black), and its closed Darboux transforms (on the left), also viewed from the top (on the right).}
	\label{fig:DTsmooth1}
\end{figure}

\begin{figure}
	\centering
	\begin{minipage}{0.45\textwidth}
		\includegraphics[width=\linewidth]{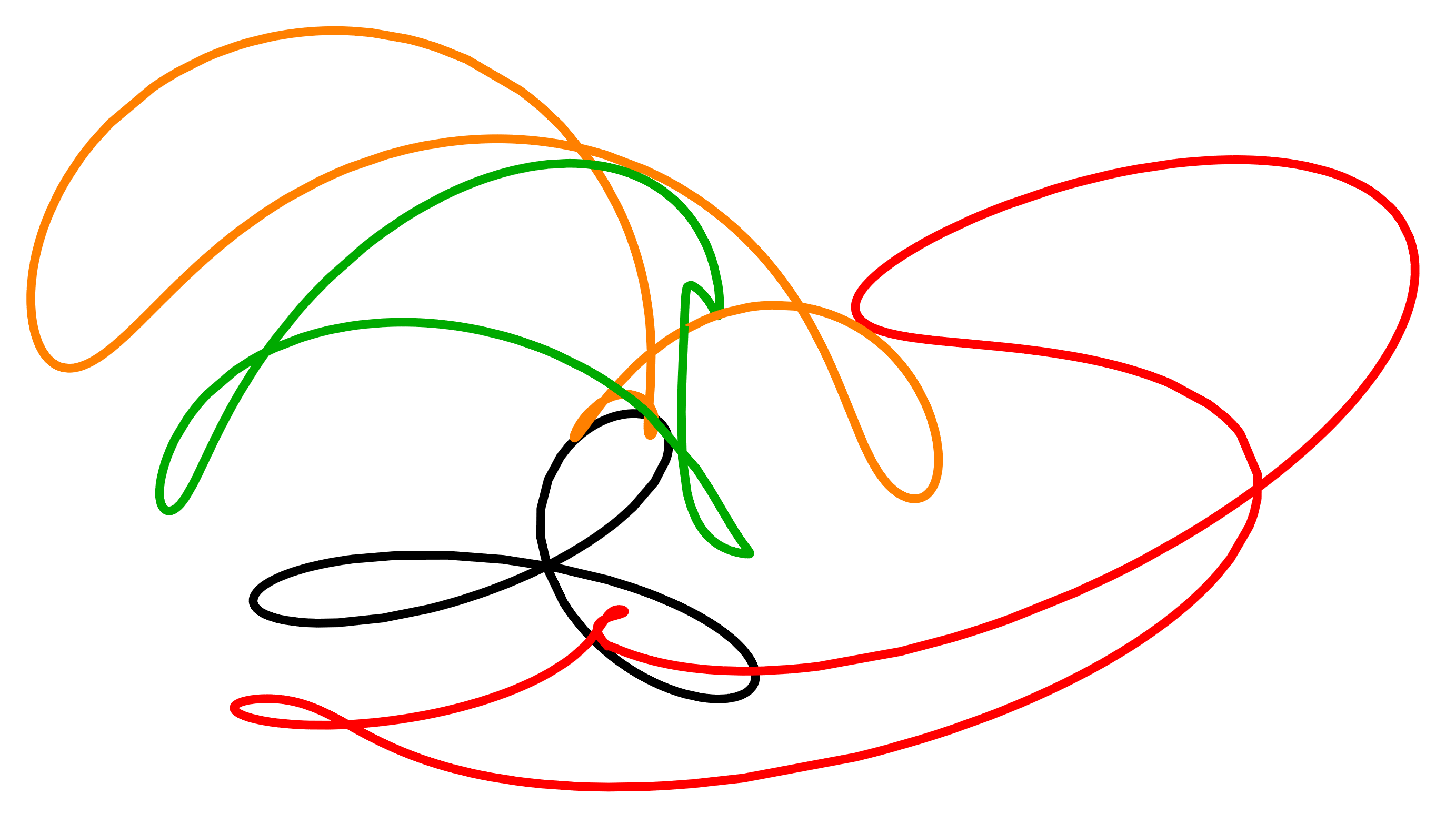}
	\end{minipage}
	\caption{Original planar curve given by $x(t) = \jj \cos(3t) e^{i t}$ (in black), and its closed Darboux transforms.}
	\label{fig:DTsmooth2}
\end{figure}

\begin{example}\label{exam:circle}
In this example, we consider the Darboux transforms of a singly-wrapped circle: let $x(t) = \jj e^{\ii t}$ be polarised by $q = \dif{t}^2 = \frac{\dif{t}^2}{m}$ as in Remark~\ref{rema:parameter}.
Noting that the $\mathcal{D}_\mu$--parallel condition \eqref{eqn:dmupar} can be reformulated as
		\begin{equation}\label{eqn:odeAlpha}
			\alpha'' = - x'' \beta - x' \beta' = x''(x')^{-1} \alpha' + \frac{\mu}{m} \alpha, \quad \beta' = -\mu (x^d)' \alpha,
		\end{equation}
where $'$ denotes the differentiation with respect to $t$, the solutions to the differential equation for $\alpha =: \alpha_0 + \jj \alpha_1$ \eqref{eqn:odeAlpha} can be written as
	\[
		\alpha = (c_0^- \alpha_0^- + c_0^+ \alpha_0^+) +  \jj (c_1^- \alpha_1^- + c_1^+ \alpha_1^+)
	\]
for some constants of integration $c_0^\pm, c_1^\pm \in \mathbb{C}$, where
	\[
		\alpha_0^\pm = e^{\frac{\ii}{2}(-1\pm\sqrt{1-4\mu})t}, \quad \alpha_1^\pm = e^{\frac{\ii}{2}(1\pm\sqrt{1-4\mu})t}.
	\]
Writing $s := \sqrt{1-4\mu}$, we then have $\beta = -(x')^{-1}\alpha' =: \beta_0 + j \beta_1$ for
	\begin{align*}
		\beta_0 &= -\frac{1}{2} e^{-\ii t} \left(c_1^- (1 - s)\alpha_1^- + c_1^+ (1 + s)\alpha_1^+\right) \\
		\beta_1 &= \frac{1}{2}e^{\ii t}\left(c_0^- (1 + s)\alpha_0^- + c_0^+ (1 - s)\alpha_0^+\right).
	\end{align*}
Noting that
	\[
		T = \alpha \beta^{-1} = (\alpha_0 + \jj \alpha_1)(\beta_0 + \jj \beta_1)^{-1} = \frac{1}{|\beta|^2}((\alpha_0 \overline{\beta_0} + \overline{\alpha_1}\beta_1) + \jj (\alpha_1 \overline{\beta_0} - \overline{\alpha_0}\beta_1)).
	\]
one obtains a Darboux transforms taking values in $\Im \mathbb{H} \cong \mathbb{R}^3$ by choosing constants of integration so that
	\[
		\Re(c_0^- \overline{c_1^-} - c_0^+ \overline{c_1^+}) = 0
	\]
via Lemma~\ref{lemm:r3condition}.
Similarly, if one chooses constants of integration so that
	\[
		c_0^- \overline{c_1^-} - c_0^+ \overline{c_1^+} = 0,
	\]
then the resulting Darboux transform takes values in the $\jj \kk$--plane via Corollary~\ref{cor:2sphere}.
In particular, for $c_0^\pm = 0$ so that $\alpha_0 = 0 = \beta_1$, we obtain for $s = \sqrt{1 - 4\mu}$,
	\begin{align*}
		\hat{x} &= x + T = \jj (e^{\ii t} + \alpha_1 \beta_0^{-1}) \\
			&= \jj \left(\frac{-e^{\ii t} \left( {c_1^+} (1 - s) e^{\ii s t} + {c_1^-} (1+s)\right)}
			{{c_1^+} (1+s) e^{\ii s t} + {c_1^-} (1 - s )}\right).
	\end{align*}

To consider the monodromy, first note that since $\beta = -(x')^{-1}\alpha'$, we see that $\varphi = e\alpha + \psi\beta$ is periodic if and only if $\alpha$ is.
Now, $\alpha_*^\pm(t+ 2\pi) = \alpha_*^\pm(t) h^\pm$ for $* = 0, 1$ with
	\[
		h^\pm = - e^{\pm \ii \pi \sqrt{1-4\mu}}.
	\]
Therefore, at the resonance points, i.e.\ $h^+ = h^-$, we have $\mu=\frac{1-k^2}4$ for some $k \in \mathbb{Z}$, and we get a $\mathbb{HP}^1$--worth of closed Darboux transforms, that is, every choice of initial conditions gives a closed Darboux transform.
Restricting to those Darboux transforms in the $\jj \kk$--plane, we obtain the following explicit paremetrisations
	\[
		\hat x= \jj \left(\frac{-e^{\ii t} \left(c_1^+ (1 - k) e^{\ii k t} + c_1^- (1+k)\right)}{c_1^+ (1 + k) e^{\ii k t}+c_1^-(1-k) }\right)
	\]
which are clearly $2\pi$--periodic for $k \in \mathbb{Z}$.
For examples of closed Darboux transforms of the circle in both $3$--space and the plane, see Figure \ref{fig:DTsmooth}.
\begin{figure}
	\centering
	\begin{minipage}{0.48\textwidth}
		\includegraphics[width=\linewidth]{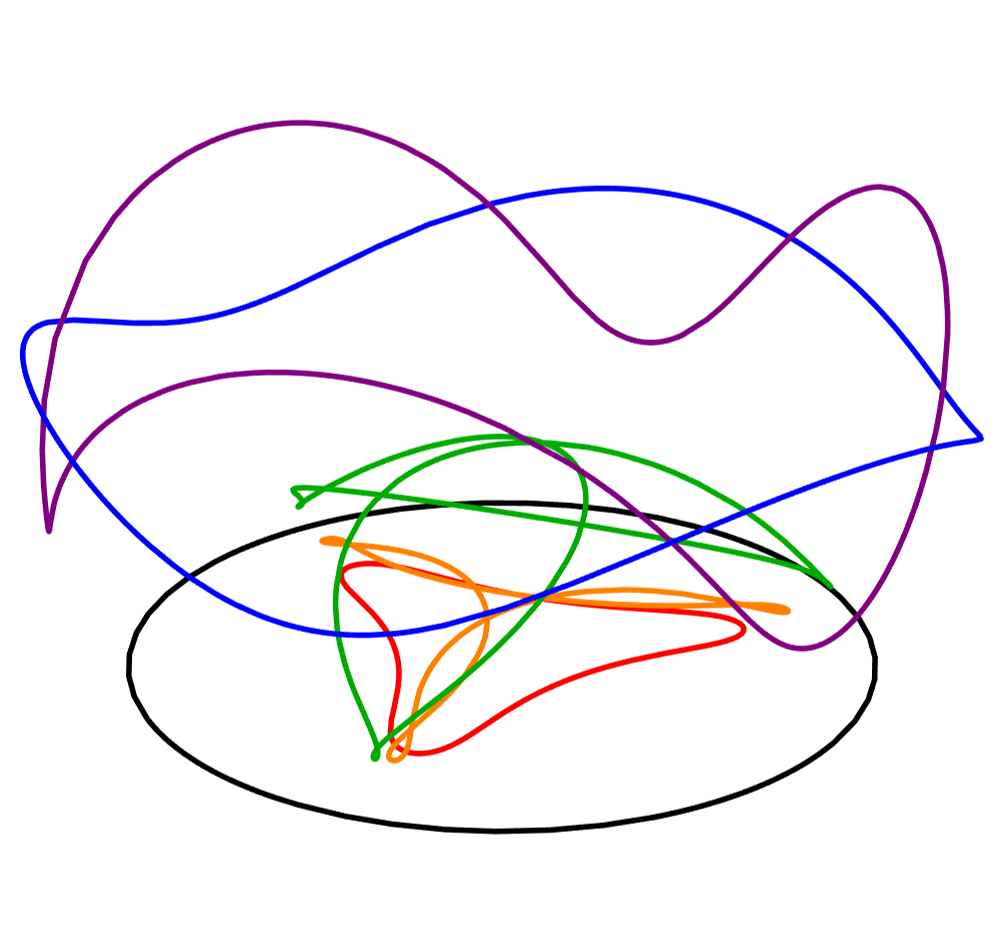}
	\end{minipage}
	\begin{minipage}{0.48\textwidth}
		\includegraphics[width=\linewidth]{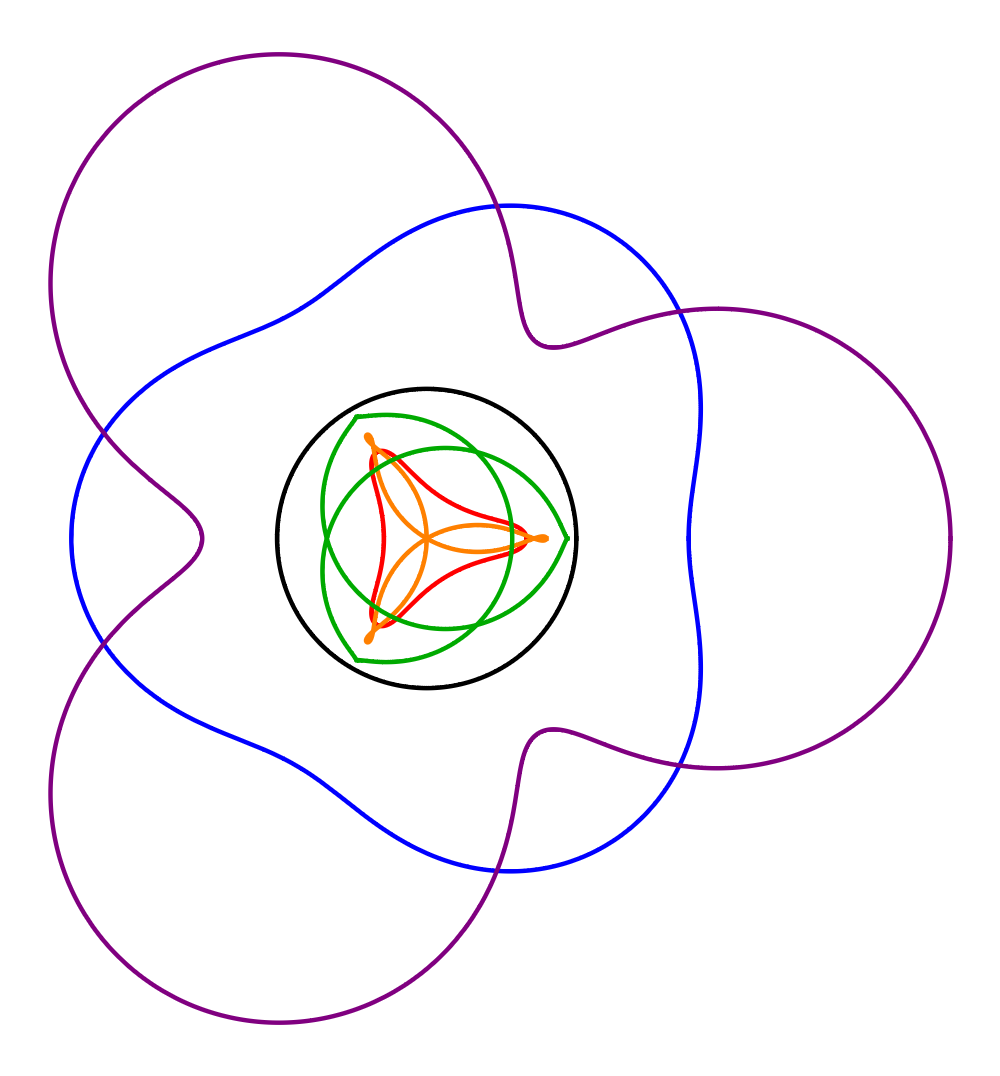}
	\end{minipage}
	\caption{Darboux transforms of the circle (drawn in black) at the resonance point with $k = 3$. On the left are those in the $3$--space with $c_0^+ = 0.5\ii, c_0^- = 0, c_1^+ = 1, c_1^- = -4, -2, -1.2, -0.1, 0.25$, on the right are those in the plane with the same constants except $c_0^+ = 0$.}
	\label{fig:DTsmooth}
\end{figure}

\begin{remark}
	We note here that by Corollary~\ref{cor:2sphere}, every Darboux transform of a circle must be contained in some $2$-sphere, determined by the circle and an initial point of the Darboux transform.
\end{remark}
\end{example}


\subsection{Arc-length polarisations and the bicycle monodromy}\label{subsect:twothree}
We now consider the integrable reduction of Darboux transformations by requiring that both curves of the Darboux pair are arc-length polarised:
\begin{definition}[{\cite[p.\ 48]{burstall_semi-discrete_2016}}]
	A curve $x : (I, q) \to \mathbb{R}^4 \cong \mathbb{H}$ is \emph{arc-length polarised} if
		\[
			q = |{\dif{x}}|^2.
		\]
\end{definition}
Given an arc-length polarised curve $x$, the condition for the Darboux transform $\hat{x}$ to be again arc-length polarised is identified in \cite{cho_infinitesimal_2020} in the case of plane curves.
Excluding the trivial case of curves reflected across a certain plane (see Figure~\ref{fig:counter}), the analogous statement for space curves can be proven similarly; for the sake of completeness, we give an independent argument here.
\begin{figure}
	\includegraphics[width=0.5\textwidth]{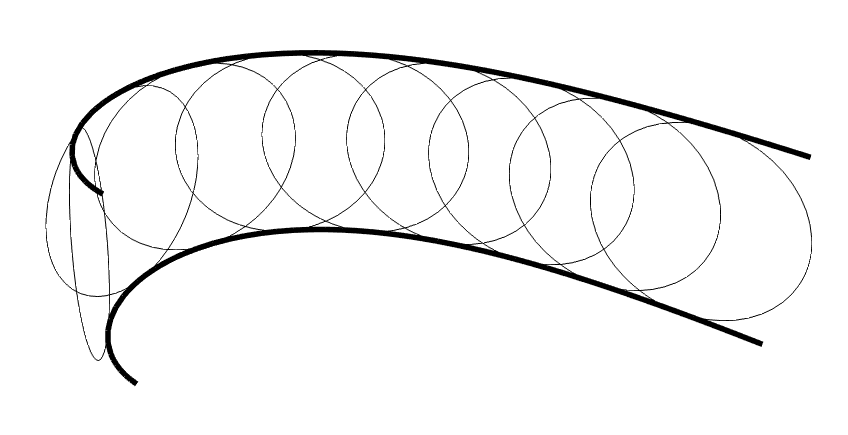}
	\caption{Trivial case of Darboux transformation keeping arc-length polarisation in $3$--space.}
	\label{fig:counter}
\end{figure}
\begin{lemma}[{\cite{cho_infinitesimal_2020}}]\label{lemm:bicycle}
	Let $x, \hat{x} : (I, q) \to \mathbb{H}$ be a (non-trivial) Darboux pair with spectral parameter $\mu$, and further assume that $x$ is arc-length polarised.
	Then $\hat{x}$ is also arc-length polarised if and only if $|\hat{x} - x|^2 = \frac{1}{\mu} > 0$ at one point $t_0 \in I$.
\end{lemma}
\begin{proof}
	We first gather some conditions coming from the given assumption that $x$ is arc-length polarised.
	For an arc-length polarised curve $x$ so that
		\[
			\dif{x}^d = q \dif{x}^{-1} = |{\dif{x}}|^2 \dif{x}^{-1} = \dif{\bar{x}},
		\]
	we calculate using the Riccati equation \eqref{eqn:smriccati}
		\begin{equation}\label{eqn:ode}
			\begin{aligned}
				\dif{(|T|^2)} &= 2 \Re (\overline{T}\dif{T}) = 2 \Re (\overline{T}(-\dif{x} + \mu T \dif{\bar{x}}\, T)) \\
					&= - 2 \Re (\overline{T}\dif{x}) + 2 \mu |T|^2\Re (\dif{\bar{x}} \,T) = 2 \mu\Re (\overline{T}\dif{x}) \left( |T|^2 - \frac{1}{\mu}\right).
			\end{aligned}
		\end{equation}
	
	The uniqueness of the solutions to ordinary differential equations tells us that $|T|^2 = \frac{1}{\mu}$ holds at one point if and only if $|T|^2 \equiv  \frac{1}{\mu}$ on $I$. 
	Thus, the necessary direction is immediately justified.
	
	To see the sufficiency, assume that $\hat{x}$ is arc-length polarised so that $|{\dif{\hat{x}}}|^2 = q$.
	Noting that the Riccati equation \eqref{eqn:smriccati} implies
		\[
			|{\dif{\hat{x}}}|^2 = \mu^2 |T|^4 |{\dif{x}}^d|^2 = \mu^2 |T|^4 |{\dif{x}}|^2 ,
		\]
	the assumption that both $x$ and $\hat{x}$ is arc-length polarised tells us
		\[
			|T|^4 \equiv \frac{1}{\mu}.
		\]
	Hence, we only need  suppose for contradiction that $|T|^2 \equiv -\frac{1}{\mu}$.
	Then via the ordinary differential equation \eqref{eqn:ode}, we must have $\Re (\overline{T} \dif{x}) \equiv 0$, implying via the Riccati equation \eqref{eqn:smriccati} that
		\[
			\Re(\dif{\hat{x}} \, \overline{T}) = \mu |T|^2 \Re(T \dif{\bar{x}}) = - \Re(\overline{T} \dif{x}) = 0,
		\]
	on $I$, so that $\langle \dif{\hat{x}}, T \rangle = \langle \dif{x}, T \rangle \equiv 0$.
	Therefore, Remark~\ref{rema:circleCongruence} implies that $\dif{\hat{x}}$ is always parallel to $\dif{x}$, telling us that this is the trivial case.
\end{proof}

Such integrable reduction is known as the \emph{tractrix construction} or the \emph{bicycle correspondence}, while the monodromy of the bicycle correspondence is called the \emph{bicycle monodromy} (see, for example, \cite{hoffmann_discrete_2008, tabachnikov_bicycle_2017, bor_tire_2020}).
As bicycle correspondences are special cases of Darboux transformations, the bicycle monodromy can also be considered in terms of Darboux transformations (see Figure~\ref{fig:smoothC}).

\begin{figure}
	\centering
	\begin{minipage}{0.48\textwidth}
		\includegraphics[width=\linewidth]{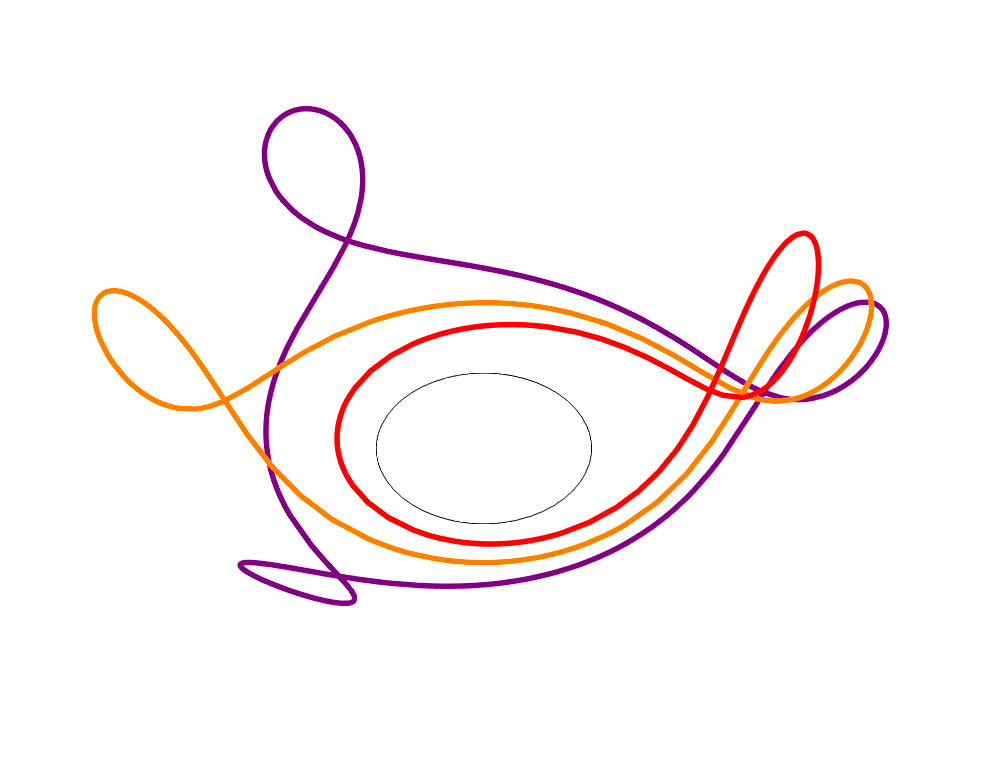}
	\end{minipage}
	\begin{minipage}{0.48\textwidth}
		\includegraphics[width=\linewidth]{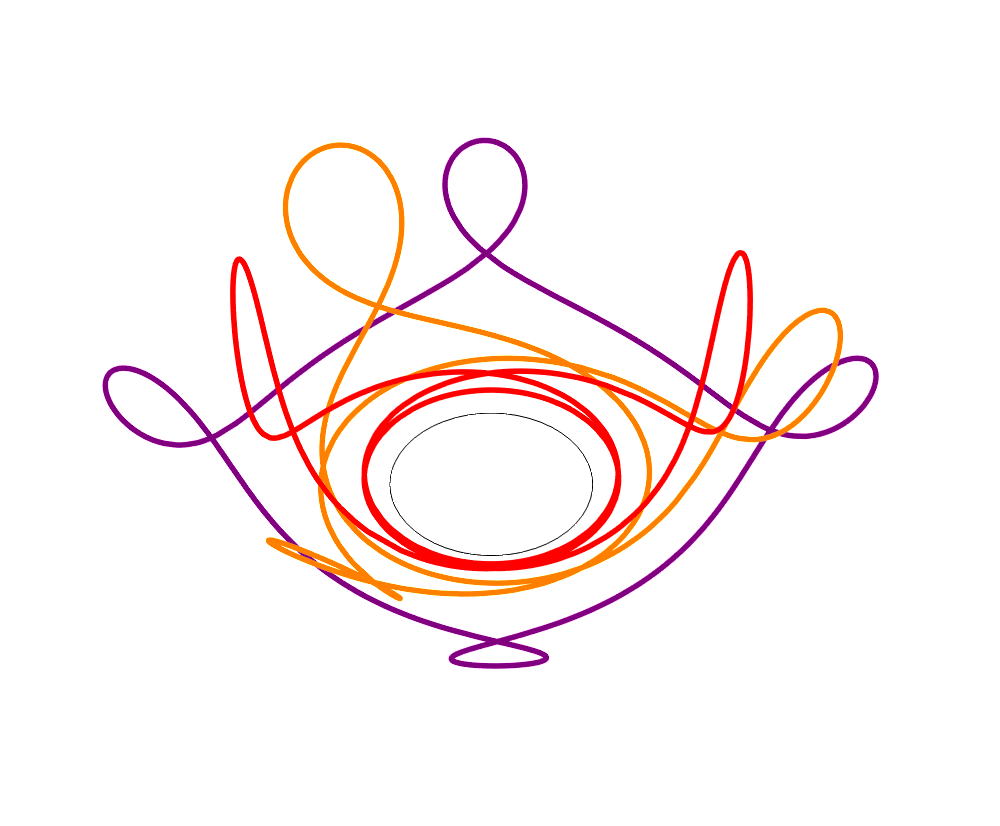}
	\end{minipage}
	\caption{Darboux transforms of the circle (drawn in black) keeping the arc-length polarisation.}
	\label{fig:smoothC}
\end{figure}

\begin{example}\label{exam:smoothB}
	The closed bicycle correspondences in the plane of a circle are called \emph{circletons} in \cite{kilian_dressing_2015}; in this example, we obtain explicit parametrisations of all circletons.
	To make direct comparison with the results in \cite{kilian_dressing_2015}, we consider the transformations in the plane given by $\Span\{1, \ii\} \cong \mathbb{C}$.
	
	For an arc-length polarised circle parametrised via $x(t) = e^{\ii t}$, recall that the Darboux transforms $\hat{x}$ with respect to spectral parameter $\mu$ are given by $\dif{}_\mu$--parallel section $\varphi =e\alpha+\psi\beta$ so that we have $\alpha = c^- \alpha^- + c^+\alpha^+$ with 
		\[
			\alpha^\pm = e^{\frac{\ii}{2} \left(1\pm\sqrt{1-4 \mu}\right) t}\
		\]
	for some constants of integrations $c^\pm \in \mathbb{C}$.
	Hence, we see that
		\[
			\beta = -\frac{1}{2} e^{- \ii t} \left(c^- (1-\sqrt{1-4 \mu}) \alpha^- 
				+ c^+ (1+\sqrt{1-4 \mu}) \alpha^+ \right).
		\]
	Excluding the trivial case of $\mu = \frac{1}{4}$, by Lemma~\ref{lemm:bicycle}, we then have that $\hat{x}$ is also arc-length polarised (and thus arc-length parametrised) if and only if $|\alpha\beta^{-1}|^2 = \frac{1}{\mu} > 0$, a condition when evaluated at $t =0$ becomes
		\[
			| c^+ + c^-|^2 = \frac 1{4\mu}\left|c^-(1-\sqrt{1-4\mu}) + c^+(1 + \sqrt{1-4\mu})\right|^2
		\]
	so that
		\[
			c^+ + c^- = \frac {e^{\ii \tau}}{2\sqrt\mu} \left(c^-(1-\sqrt{1-4\mu}) + c^+(1 + \sqrt{1-4\mu})\right)
		\]
	for some $\tau \in \mathbb{R}$.
	Thus, if $1- \frac {e^{\ii\tau}}{2\sqrt\mu}  (1 + \sqrt{1-4\mu}) \neq 0$, that is $\mu \neq \frac{1}{4\cos^2(\tau)}$, we obtain $c^+ =\chi c^-$ with
		\[
			\chi = \frac{-2 \sqrt {\mu} + e^{\ii \tau}(1-\sqrt{1-4\mu})}{2\sqrt{\mu} - e^{\ii \tau}(1+\sqrt{1-4\mu})}. 
		\]
	Otherwise, we obtain $c^-=0$ for any choice of $c^+$.
	Therefore, all arc-length polarised Darboux transforms $\hat{x}$ of an arc-length polarised circle $x$ are given by
		\[
			\hat{x}(t) = \frac{e^{\ii t} \left( \chi  \left(\sqrt{1-4 \mu}-1\right) e^{\ii \sqrt{1-4\mu} t}-\left(\sqrt{1-4\mu}+1\right)\right)}
				{{\chi}\left(\sqrt{1-4 \mu}+1\right) e^{\ii \sqrt{1-4 \mu} t}-\left(\sqrt{1-4 \mu}-1\right)}.
		\]
	
	We now investigate the monodromy of the Darboux transformation $\hat{x}$ over a single period.
	Since we have explicit formulas, it is straightforward to calculate that since we have $\mu > 0$, the Darboux transformation is a closed curve if and only if either $c^- = 0$ or $\chi = 0$.
	
\begin{figure}
	\begin{minipage}{0.193\textwidth}
		\includegraphics[width=\linewidth]{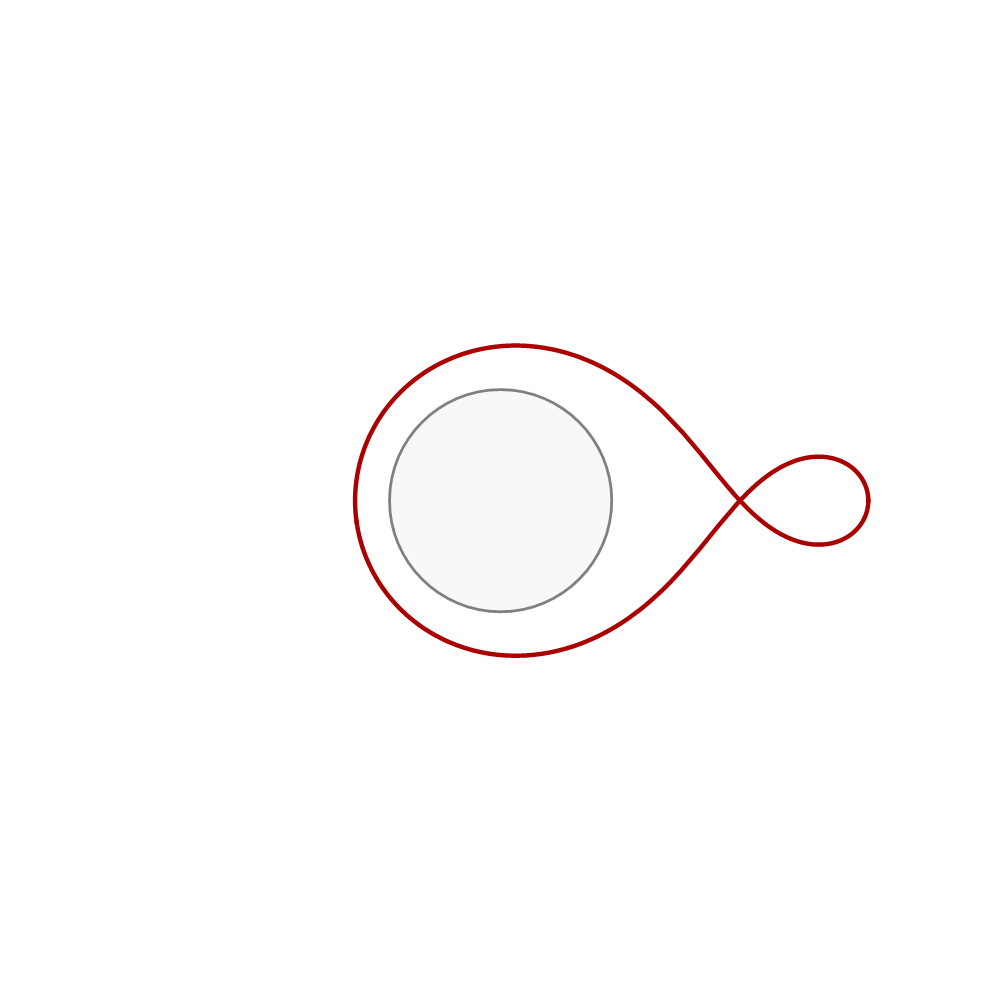}
	\end{minipage}
	\begin{minipage}{0.193\textwidth}
		\includegraphics[width=\linewidth]{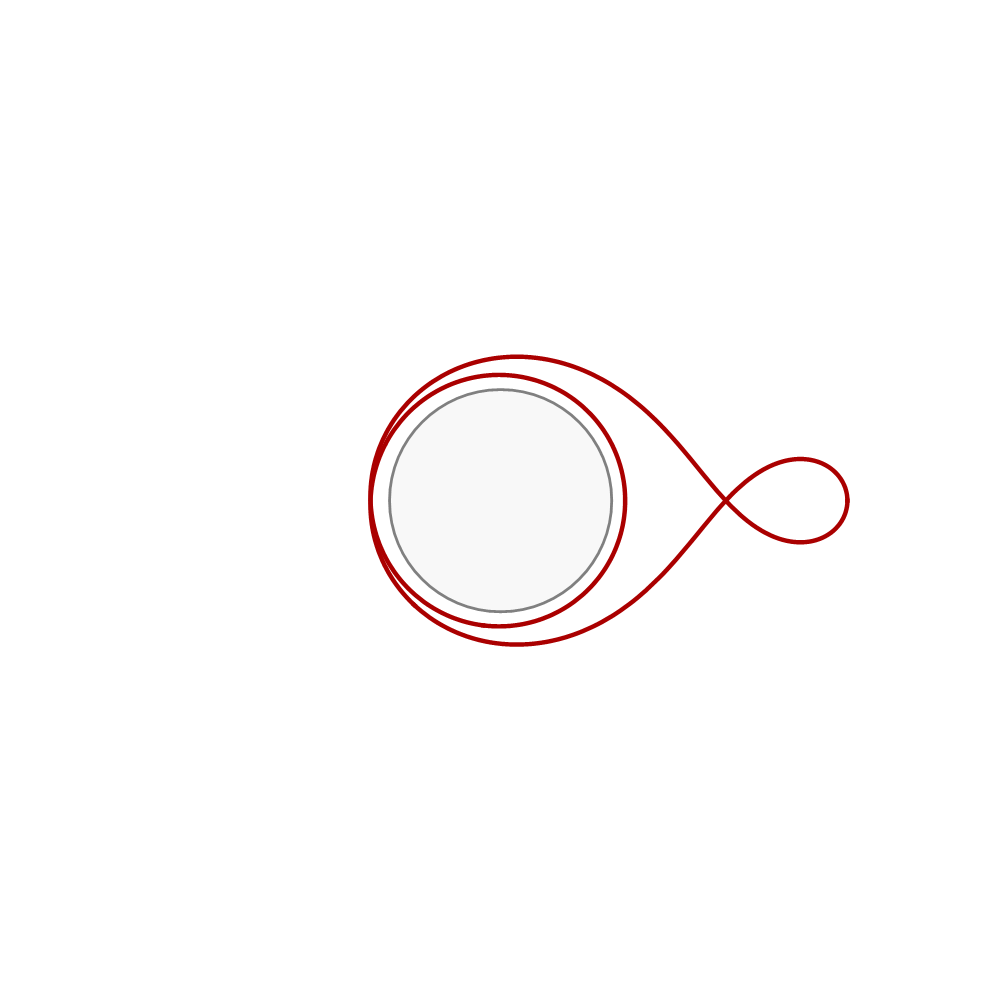}
	\end{minipage}
	\begin{minipage}{0.193\textwidth}
		\includegraphics[width=\linewidth]{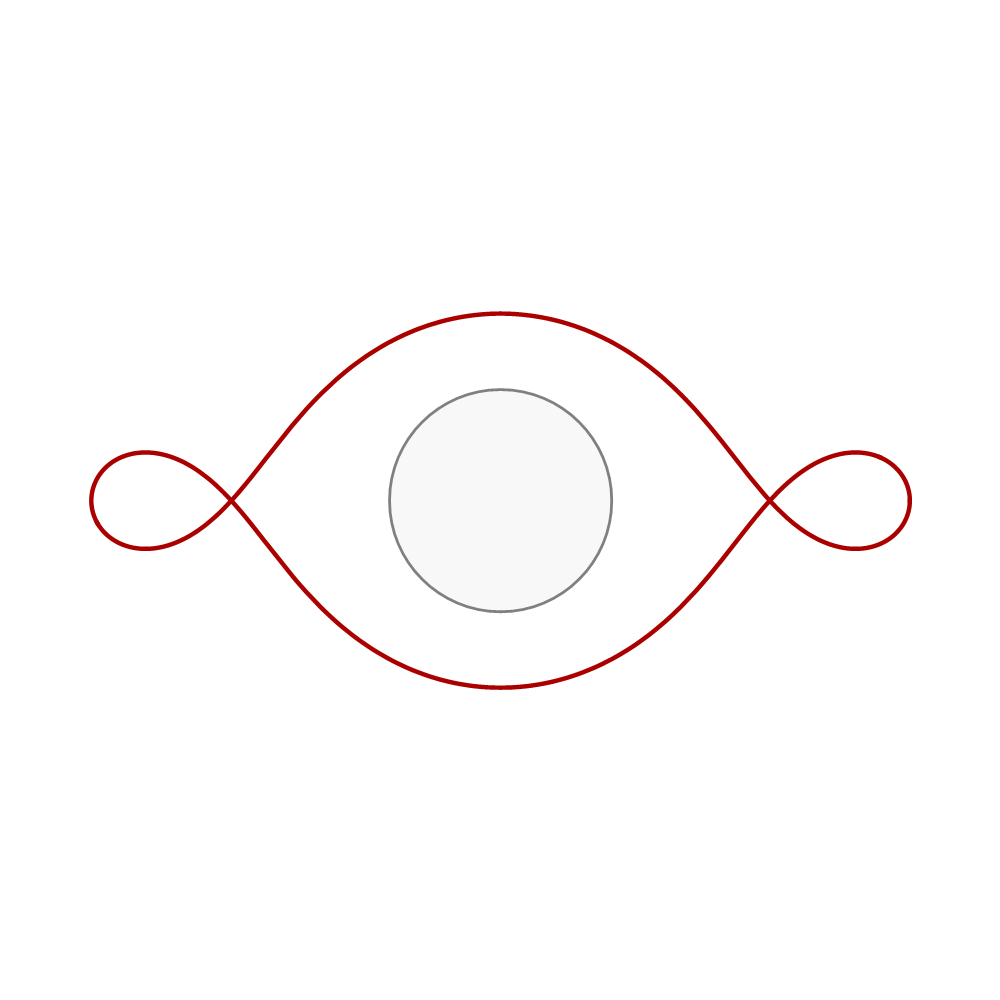}
	\end{minipage}
	\begin{minipage}{0.193\textwidth}
		\includegraphics[width=\linewidth]{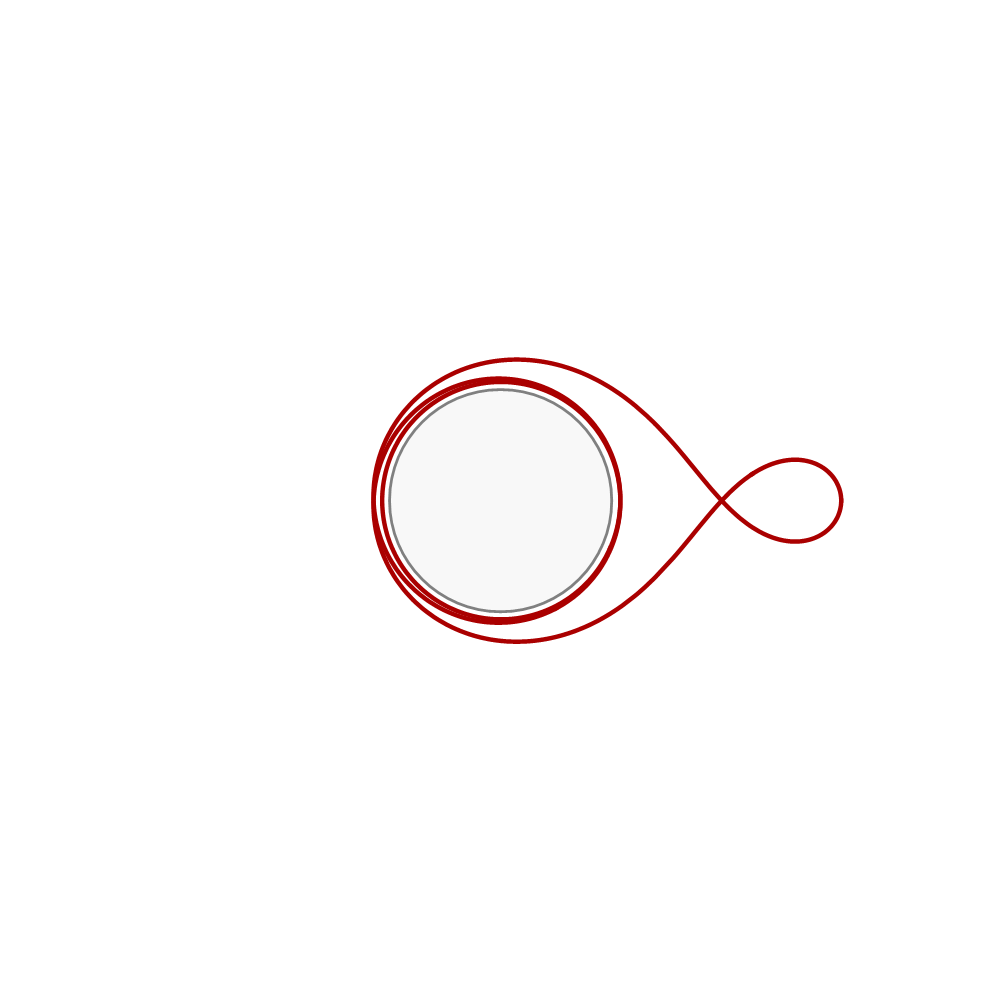}
	\end{minipage}
	\begin{minipage}{0.193\textwidth}
		\includegraphics[width=\linewidth]{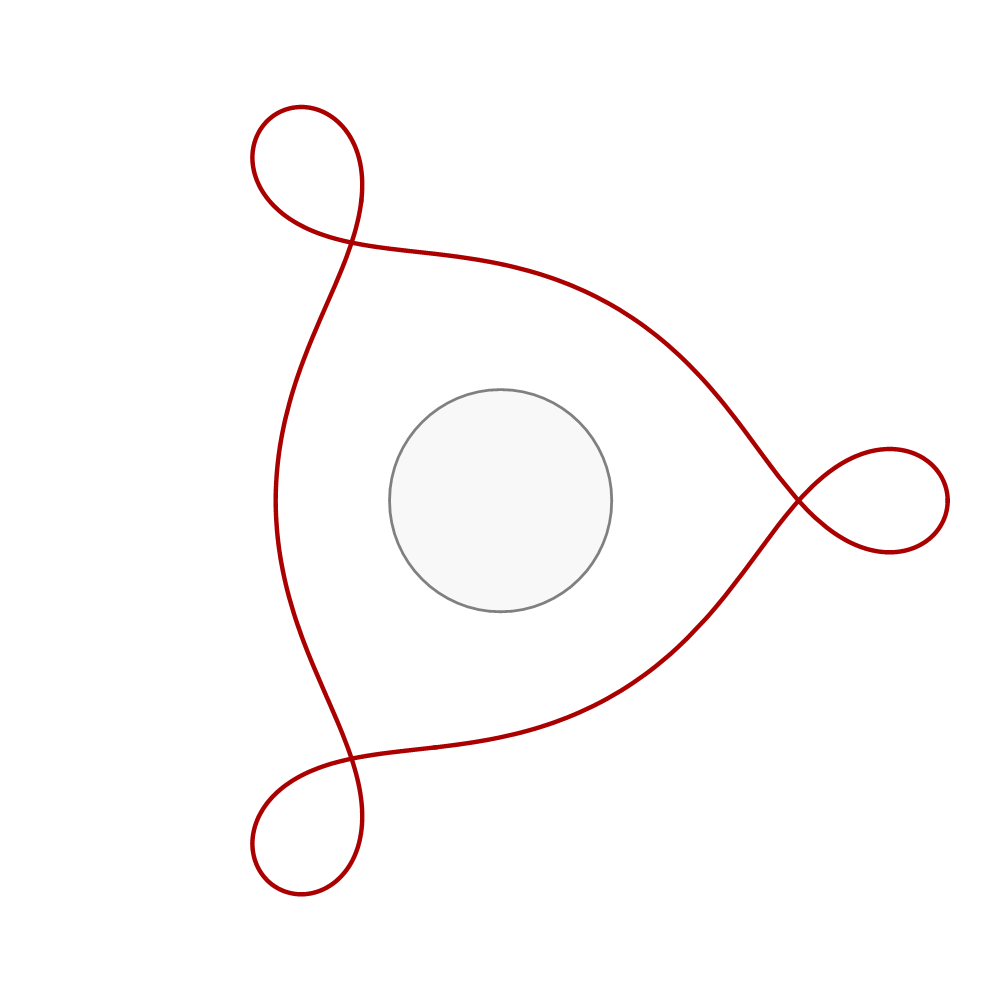}
	\end{minipage}
	\\
	\begin{minipage}{0.193\textwidth}
		\includegraphics[width=\linewidth]{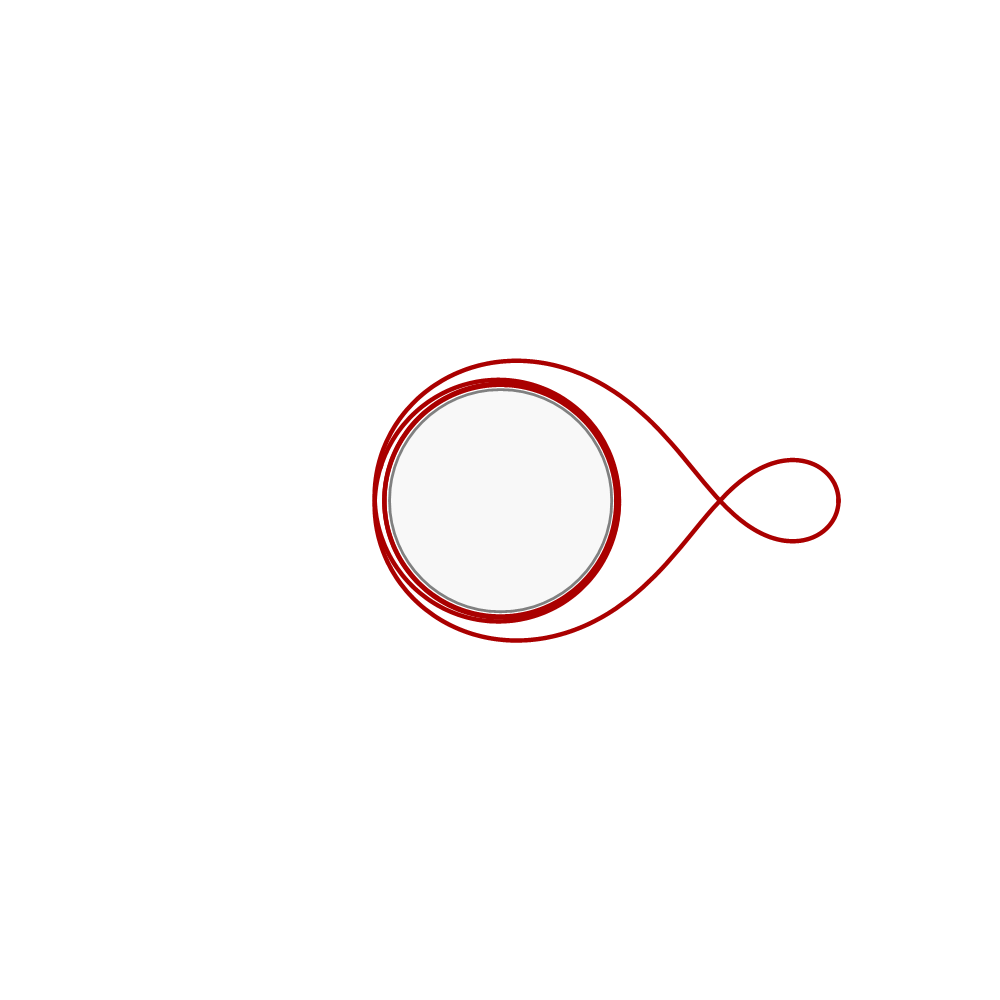}
	\end{minipage}
	\begin{minipage}{0.193\textwidth}
		\includegraphics[width=\linewidth]{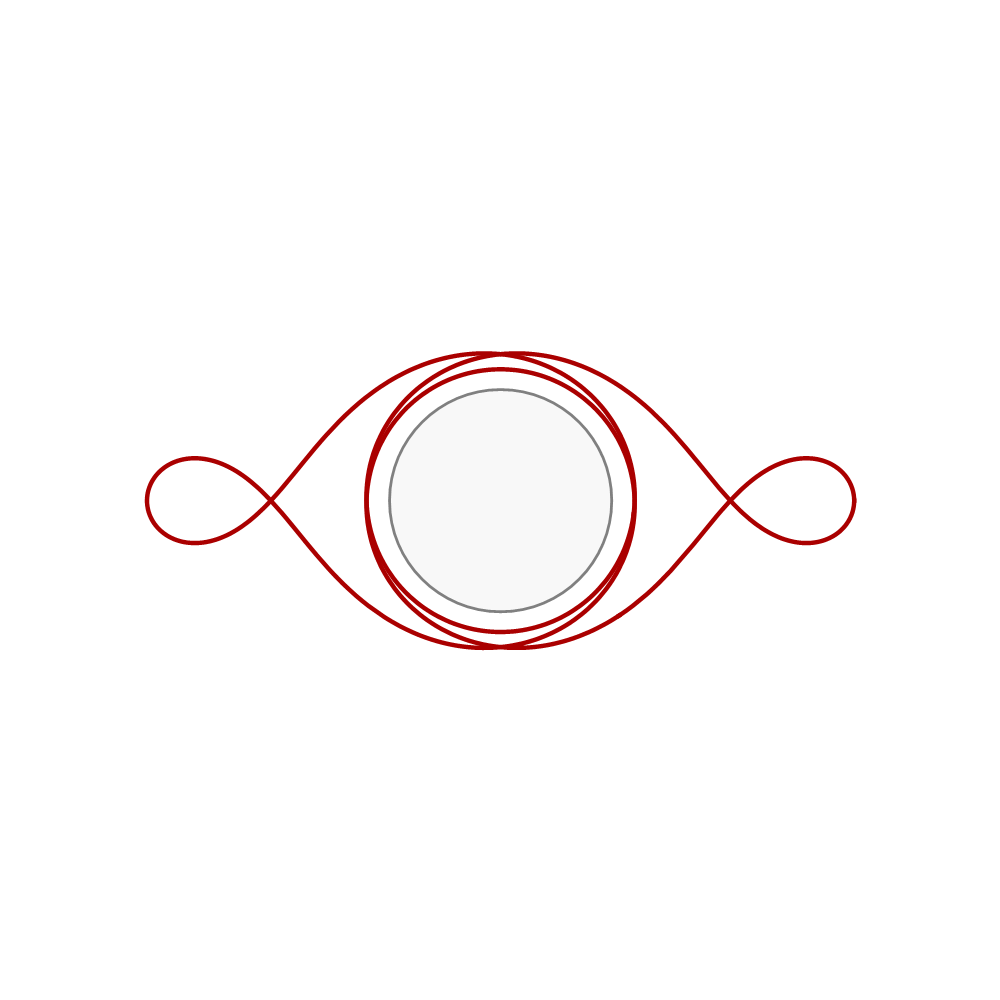}
	\end{minipage}
	\begin{minipage}{0.193\textwidth}
		\includegraphics[width=\linewidth]{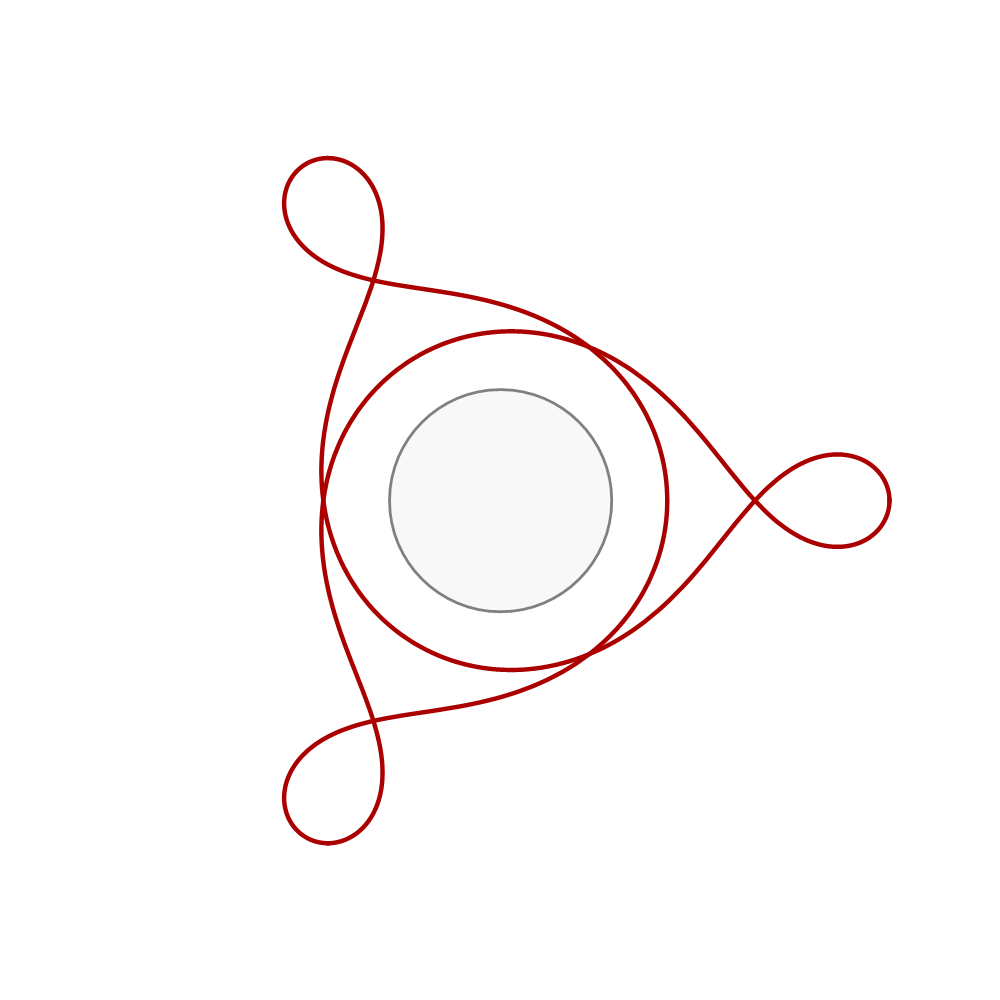}
	\end{minipage}
	\begin{minipage}{0.193\textwidth}
		\includegraphics[width=\linewidth]{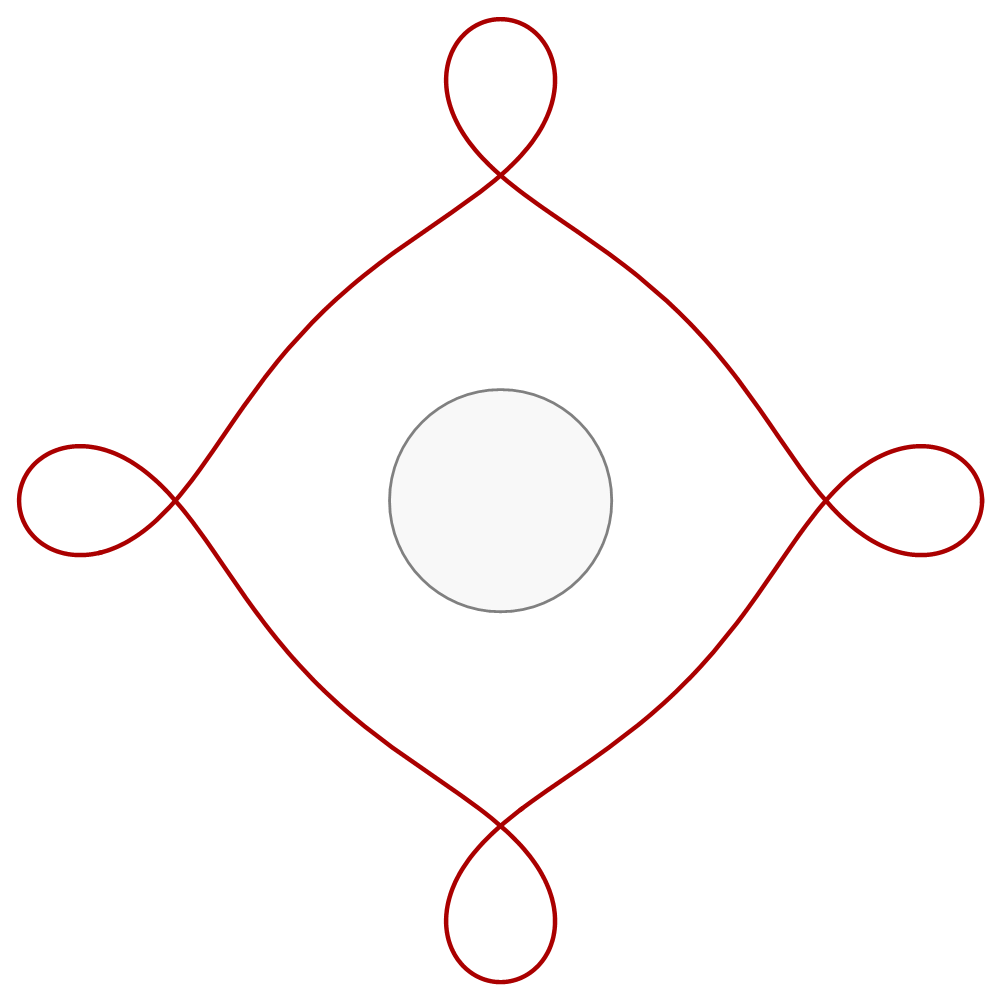}
	\end{minipage}
	\caption{Circletons as seen in \cite[Figure 5.1]{kilian_dressing_2015} and \cite[Figure 21]{bor_tire_2020}, drawn with $\tau = \pi$ and $(k, \ell) = (1,2)$, $(1,3)$,
	 $(2,3)$, $(1,4)$, $(3,4)$, $(1,5)$, $(2,5)$, $(3,5)$, $(4,5)$ over $\ell$--fold cover of the circle.}
	\label{fig:smoothB}
\end{figure}
	To investigate the resonance points of the transformation to obtain non-trivial transforms, we first calculate as in Example~\ref{exam:circle} that
		\[
			\alpha_{\pm}(t+ 2\pi) = \alpha_\pm(t) e^{\ii \pi} e^{\pm \ii \pi \sqrt{1-4\mu}}.
		\]
	Thus, $\alpha_\pm$ have the same multiplier if and only if
		\[
			\sqrt{1-4\mu}= k \in \mathbb{Z},
		\]
	a contradiction in the current example since now we have $\mu > 0$.
	
	Thus, we consider the $\ell$--fold cover of $[0, 2\pi]$ and calculate
		\[
			\alpha_{\pm}(t+ 2\pi \ell) = \alpha_\pm(t) e^{\ii \ell \pi } e^{\pm \ii \sqrt{1-4\mu} \ell \pi},
		\]
	allowing us to deduce that, $\alpha_\pm$ have the same multiplier if and only if 
		\[
			\ell \sqrt{1-4\mu}= k \in \mathbb{Z}.
		\]
		
	Thus, we have circletons over $\ell$--fold cover of $[0, 2\pi]$ if and only if $\mu=\frac{\ell^2-k^2}{4 \ell^2}$ and $\ell > k > 0$ by Lemma~\ref{lemm:bicycle}, recovering the result of \cite{kilian_dressing_2015}.
	For examples of circletons, see Figure~\ref{fig:smoothB}.
\end{example}

\begin{remark}\label{rema:non-constant}
	The linearisation of Riccati-type equations for finding Darboux transforms of polarised curves and the subsequent investigation of monodromy can also be theoretically carried out in the lightcone model of Möbius geometry as introduced in \cite{burstall_semi-discrete_2016}.
	However, a quick calculation in the lightcone model using circles polarised by arc-length yields second-order ordinary differential equations with \emph{non-constant coefficients}, as opposed to that with \emph{constant coefficients} in the quaternionic model.
	Thus, the quaternionic approach enables us to efficiently obtain closed-form solutions, which is central to the explicit investigation of monodromy.
\end{remark}

\section{Monodromy of discrete Darboux transformations}\label{sect:three}
Having discussed the smooth theory in detail, we now aim to investigate the monodromy of Darboux transformations of a discrete closed curve.
The gauge theoretic description of Darboux transformations was central to the consideration of the monodromy in the smooth case; therefore, we briefly review the discrete gauge theory here (for a more detailed introduction, see, for example, \cite{burstall_isothermic_2011, burstall_notes_2017}).

Let now $\domain$ denote a discrete interval, simply-connected in the sense of \cite[\S 2.3]{burstall_discrete_2020}.
A \emph{bundle} on $\domain$ assigns a set $V_i$ to each vertex $i \in \domain$, and we call $\sigma : \domain \to \cup V_i$ a \emph{section} if $\sigma_i \in V_i$ for all $i \in \domain$.

A \emph{discrete connection} assigns a bijection $r_{ji}: V_i \to V_j$ on each oriented edge $(ij)$ so that $r_{ij} r_{ji} = \iden_i$, while a discrete gauge transformation $\mathcal{G}$ acts on a connection $r_{ji}$ via
	\[
		(\mathcal{G} \bullet r)_{ji} = \mathcal{G}_j \circ r_{ji} \circ \mathcal{G}_i^{-1}
	\]
where $\mathcal{G}_i : V_i \to V_i$ is an automorphism defined at every vertex $i \in \domain$.

A section $\sigma$ is $r$--parallel if on any oriented edge $(ij)$, we have
	\[
		r_{ji} \sigma_i = \sigma_j.
	\]
Note that since we are working with discrete intervals, we have that every discrete connection is \emph{flat}.

\subsection{Flat connection of a discrete polarised curve}
Let $x : (\domain, \tfrac{1}{m}) \to \mathbb{R}^4 \cong \mathbb{H}$ be a discrete curve defined on a polarised domain $(\domain, \tfrac{1}{m})$ for a strictly positive or negative function $m$ defined on (unoriented) edges.
Recalling the definition of exterior derivatives \cite[Definition 2.2]{burstall_discrete_2020} for $0$-forms $x$
	\[
		\dif{x}_{ij} := x_i - x_j,
	\]	
the \emph{dual curve} $x^d : (\domain, \tfrac{1}{m}) \to \mathbb{H}$ of a discrete polarised curve $x : (\domain, \tfrac{1}{m}) \to \mathbb{H}$ is defined via the discrete $1$-form
	\[
		\dif{x}^d_{ij} = \frac{1}{m_{ij}} \dif{x}_{ij}^{-1}.
	\]
As in the smooth case,  let $\psi = \begin{psmallmatrix} x \\ 1 \end{psmallmatrix}$ so that we can view $L = \psi \mathbb{H}$ as a line subbundle of the trivial vector bundle $\underline{\mathbb{H}}^2 := I \times \mathbb{H}^2$.
On every edge $(ij)$, define a linear isomorphism $\mathcal{D}^\lambda_{ji} : \{i\} \times  \mathbb{H}^2 \to \{j\} \times \mathbb{H}^2$ for some $\lambda \in \mathbb{R}$ via
	\begin{equation}\label{eqn:discretecalD}
		\mathcal{D}^\lambda_{ji} := \iden_{ji} +
			\begin{pmatrix}
				0 & \dif{x}_{ij} \\
				\lambda \dif{x}^d_{ij} & 0
			\end{pmatrix}
	\end{equation}
for the identity map $\iden_{ji} :  \{i\} \times  \mathbb{H}^2 \to \{j\} \times \mathbb{H}^2$.
Then it is immediate (with some abuse of notation on the identity map) that
	\[
		\mathcal{D}^\lambda_{ij} \mathcal{D}^\lambda_{ji} = \mathcal{D}^\lambda_{ji} \mathcal{D}^\lambda_{ij} = \left( 1 - \frac{\lambda}{m_{ij}} \right) \iden,
	\]
implying that $\mathcal{D}^\lambda$ does not define a discrete connection.

Therefore, we instead consider the projective transformation $(\mathcal{D}^\lambda)^P_{ji} : \{i\} \times \mathbb{HP}^1 \to \{j\} \times \mathbb{HP}^1$ induced by $\mathcal{D}^\lambda_{ji}$, so that $(\mathcal{D}^\lambda)^P_{ji}$ is a (flat) connection defined on the trivial bundle $\underline{\mathbb{HP}}^1$.

\begin{remark}
	Alternatively, one could normalise $\mathcal{D}^\lambda_{ji}$ so that $\mathcal{D}^\lambda_{ji}$ defines a discrete connection on the trivial bundle $\underline{\mathbb{H}}^2$.
	The choice for the projective bundle $\underline{\mathbb{HP}}^1$ is made to avoid the introduction of square root terms in the discrete connection, and to keep the similarity in the expression of $\mathcal{D}_\lambda$ of the smooth case \eqref{eqn:smoothcalD} and $\mathcal{D}^\lambda$ in the discrete case \eqref{eqn:discretecalD}.
\end{remark}
%
%
	

Now for $\mathcal{G} = (e \: \psi)$, we have that
	\begin{align*}
		\dif{}^\lambda_{ji} &:= (\mathcal{G} \bullet \mathcal{D}^\lambda)_{ji}
			= \iden_{ji} + \lambda
			\begin{pmatrix}
				 x_j \dif x^d_{ij} & - x_j \dif x^d_{ij} x_i \\
				 \dif{x}^d_{ij} & - \dif{x}^d_{ij} x_i
			\end{pmatrix} \\
			&=: \iden_{ji} + \lambda \eta_{ji}
	\end{align*}
where
	\[
		\im \eta_{ji} =  \begin{pmatrix} x_j \\ 1 \end{pmatrix} \mathbb{H} \quad\text{and}\quad \ker \eta_{ji} =  \begin{pmatrix} x_i \\ 1 \end{pmatrix} \mathbb{H}.
	\]

\begin{definition}\label{def:discFlat}
	Defining $(\dif{}^\lambda)^P_{ji} : \{i\} \times \mathbb{H}\mathbb{P}^1 \to \{j\} \times \mathbb{H}\mathbb{P}^1$ to be the projective isomorphism induced by $\dif{}^\lambda_{ji} : \{i\} \times \mathbb{H}^2 \to \{j\} \times \mathbb{H}^2$, we call the discrete connection $(\dif{}^\lambda)^P_{ji}$ the \emph{discrete (flat) connection} associated to $x$.
\end{definition}

The central ethos of discrete differential geometry dictates that the discrete integrable structure is inherent in the transformations and permutability of the smooth integrable structure.
In our case, the integrable structure of discrete polarised curves, represented by the discrete (flat) connections, should be related to the Darboux transformations of smooth polarised curves, represented by the gauge transformations introduced in Proposition~\ref{prop:gauge}.
The next proposition clarifies this relationship: the corresponding points of successive Darboux transforms of a smooth curve give the discrete curve, the spectral parameters of the Darboux transformations become the discrete polarisation (see Figure~\ref{fig:integ}), and the gauge transformation of the Darboux transformations induces the discrete (flat) connection:
\begin{proposition}\label{prop:discFlat}
	For the splitting $\underline{\mathbb{H}}^2 = L_i \oplus L_j$, denote by $\pi$ the projection onto the line bundle $L$ defined at every vertex.
	Then we have that
		\[
			\dif{}^\lambda_{ji} =  \pi_i +  \frac{m_{ij} - \lambda}{m_{ij}} \pi_j.
		\]
\end{proposition}
\begin{proof}
	Due to the splitting, we only need to see that the two sides agree for $\psi_i$ and $\psi_j$.
	Note that
		\[
			\dif{}^\lambda_{ji} \psi_i = ( \iden_{ji} + \lambda \eta_{ji})  \psi_i = \psi_i
		\]
	since $\psi_i \in \ker \eta_{ji}$, so that they agree on $\psi_i$.
	On the other hand, using the fact that
		\[
			\eta_{ji}\psi_j = \psi_j  (\dif{x}^d_{ij} x_j - \dif{x}^d_{ij} x_i) = - \psi_j  \dif{x}^d_{ij} \dif{x}_{ij} = - \frac{1}{m_{ij}}\psi_j,
		\]
	we see
		\[
			\dif{}^\lambda_{ji} \psi_j = (\iden_{ji} + \lambda \eta_{ji})  \psi_j = \psi_j + \lambda \eta_{ji} \psi_j = \psi_j -\frac{\lambda}{m_{ij}} \psi_j = \left(\frac{m_{ij} - \lambda}{m_{ij}} \right)\psi_j,
		\]
	giving us the desired conclusion.
\end{proof}

\begin{figure}
	\centering
	\includegraphics[width=\textwidth]{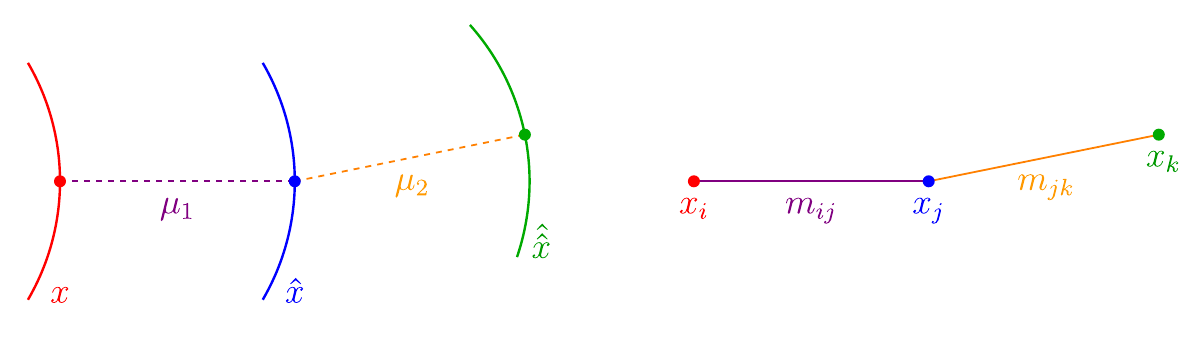}
	\caption{The integrable structure of discrete polarised curves coming from Darboux transformations of smooth polarised curves.}
	\label{fig:integ}
\end{figure}

\subsection{Discrete Darboux transformations via parallel sections}

Now we recall the definition of Darboux transformations of discrete polarised curves given in \cite[Definition 3.1]{cho_discrete_2021-1}, adapted for curves in $\mathbb{R}^4 \cong \mathbb{H}$:
\begin{definition}
	Two discrete polarised curves $x, \hat{x} : (I, \frac{1}{m}) \to \mathbb{H}$ are called a \emph{Darboux pair} with spectral parameter $\mu$ if on every edge $(ij)$, the cross-ratios of the four points $x_i, x_j, \hat{x}_j, \hat{x}_i$, denoted by $\cratio(x_i, x_j, \hat{x}_j, \hat{x}_i)$, satisfy
		\begin{equation}\label{eqn:cr}
			\cratio(x_i, x_j, \hat{x}_j, \hat{x}_i) = (x_i - x_j) (x_j - \hat{x}_j)^{-1} (\hat{x}_j - \hat{x}_i) (\hat{x}_i - x_i)^{-1} = \frac{\mu}{m_{ij}}.
		\end{equation}
\end{definition}

\begin{remark}\label{rema:nondeg}
	Throughout the paper, we assume for non-degeneracy that $\mu \neq m_{ij}$ for any edge $(ij)$.
\end{remark}

\begin{remark}\label{rema:cratio}
	The cross-ratios condition \eqref{eqn:cr} implies that the four points $x_i, x_j, \hat{x}_j, \hat{x}_i$ are \emph{concircular}.
\end{remark}

The cross-ratios condition \eqref{eqn:cr} is equivalent to the discrete Riccati equation:
	\begin{equation}\label{eqn:riccati2}
		\dif{\hat{x}}_{ij}  = \frac{\mu}{m_{ij}} (\hat{x}_j - x_j)(x_i - x_j)^{-1}(\hat{x}_i - x_i) =  \mu (\hat{x}_j - x_j)\dif{x}^d_{ij}(\hat{x}_i - x_i),
	\end{equation}
and defining $T := \hat{x} - x$, we have that
	\begin{equation}\label{eqn:riccati1}
		\dif{T}_{ij} = - \dif{x}_{ij} +  \mu T_j\dif{x}^d_{ij}T_i.
	\end{equation}
As in the smooth case, we show that Darboux transforms of a discrete polarised curve can be obtained via the parallel sections of its associated family of connections.
\begin{theorem}\label{thm:discDar}
	Given a discrete polarised curve $x : (I, \frac{1}{m}) \to \mathbb{H}$, a discrete polarised curve $\hat{x} : (I, \frac{1}{m}) \to \mathbb{H}$ is a Darboux transform of $x$ with spectral parameter if and only if $\hat{L} = \hat{\psi} \mathbb{H} = \begin{psmallmatrix} \hat{x} \\ 1 \end{psmallmatrix} \mathbb{H}$ is $(\dif{}^{\mu})^P$--parallel.
\end{theorem}
\begin{proof}
	First assume $\hat{L} = \hat{\psi} \mathbb{H} = \begin{psmallmatrix} \hat{x} \\ 1 \end{psmallmatrix} \mathbb{H}$ is $(\dif{}^{\mu})^P$--parallel, that is, on any fixed edge $(ij)$,
		\[
			(\dif{}^{\mu})^P_{ji} \hat{L}_i = \hat{L}_j.
		\]
	Then since we have $\dif{}^{\mu}_{ji} = (\mathcal{G} \bullet \mathcal{D}^{\mu})_{ji}$, there is 
		\[
			\begin{pmatrix} a \\ b \end{pmatrix} = \phi : I \to \mathbb{H}^2
		\]
	such that
		\begin{equation}\label{eqn:parallel1}
			\mathcal{D}^{\mu}_{ji} \phi_i = \phi_j
		\end{equation}
	with $\hat{L} = \mathcal{G}\phi \mathbb{H}$; hence, we have $\hat{x} = x + a b^{-1}$.
	Noting that the condition \eqref{eqn:parallel1} can be restated as
		\begin{equation}\label{eqn:diffPar}
			\dif{\begin{pmatrix} a \\ b \end{pmatrix}}_{ij} = -\begin{pmatrix} \dif{x}_{ij} b_i \\ \mu \dif{x}^d_{ij} a_i \end{pmatrix},
		\end{equation}
	we deduce with $T := a b^{-1}$ that
		\begin{align*}
			 - \dif{x}_{ij} +  \mu T_j\dif{x}^d_{ij}T_i
				&= \dif{a}_{ij} b^{-1}_i + a_j\dif{(b^{-1})}_{ij}\\
				&= a_i b^{-1}_i - a_j b^{-1}_j = \dif{T}_{ij}.
		\end{align*}
	Therefore, $\hat{x}$ is a solution to the discrete Riccati equation \eqref{eqn:riccati2}, and thus a Darboux transform of $x$.
	
	Conversely, assume that $\hat{x}$ solves the Riccati equation \eqref{eqn:riccati2}, that is, there is $T : I \to \mathbb{H}$ satisfying the Riccati equation \eqref{eqn:riccati1} and $\hat{x} = x + T$.
	On any edge $(ij)$, define $b : I \to \mathbb{H}$ recursively via the equation
		\[
			b_j = (1 + T^{-1}_j \dif{\hat{x}}_{ij})b_i,
		\]
	and let $a := T b$.
	Then we have
		\begin{align*}
			- \mu \dif{x}^d_{ij} a_i &= - \mu \dif{x}^d_{ij} T_i b_i = - \mu T^{-1}_j T_j \dif{x}^d_{ij} T_i b_i \\ 
				&= - T^{-1}_j \dif{\hat{x}}_{ij} b_i = \dif{b}_{ij},
		\end{align*}
	while
		\begin{align*}
			- \dif{x}_{ij} b_i &= (\dif{T}_{ij} - \mu T_j \dif{x}^d_{ij} T_i) b_i = a_i - T_j b_i - \mu T_j \dif{x}^d_{ij} a_i \\
				&= a_i - T_j b_i + T_j \dif{b}_{ij} = a_i - T_j b_i + T_j b_i - T_j b_j = \dif{a}_{ij}.
		\end{align*}
	Therefore, $\phi := \begin{pmatrix} a \\ b \end{pmatrix}$ satisfies \eqref{eqn:diffPar}, i.e., $\mathcal{D}^\mu_{ji} \phi_i = \phi_j$.
	But since we have that $\hat{L} = \hat{\psi} \mathbb{H} = \mathcal{G}\phi \mathbb{H}$, 
		\[
			\dif{}^{\mu}_{ji}\hat{\psi}_i \mathbb{H} = \mathcal{G}_j \mathcal{D}^\mu_{ji} \mathcal{G}^{-1}_i \mathcal{G}_i \phi_i \mathbb{H} = \mathcal{G}_j \mathcal{D}^\mu_{ji} \phi_i \mathbb{H} = \mathcal{G}_j \phi_j \mathbb{H} = \hat{\psi}_j \mathbb{H},
		\]
	telling us that $\hat{L}$ is $(\dif{}^{\mu})^P$--parallel.
\end{proof}
Similar to the smooth case, the choice of initial condition completely determines whether the Darboux transform $\hat{x}$ of a discrete curve in some $3$-sphere $S^3$ again takes values in the same $S^3$:
\begin{lemma}\label{lemm:r3disc}
	Given a discrete polarised curve $x : (I, \frac{1}{m}) \to S^3$ taking values in some $3$-sphere $S^3$ and $(\dif{}^{\mu})^P$--parallel $\hat{L} = \hat{\psi}\mathbb{H} = \begin{psmallmatrix} \hat{x} \\ 1 \end{psmallmatrix}\mathbb{H}$, the Darboux transform $\hat{x}$ takes values in $S^3$ if and only if $\hat{x}_i \in S^3$ for some $i \in I$.
\end{lemma}
\begin{proof}
	As in the smooth case, we will apply a suitable stereographic projection and prove the statement for curves in $\Im \mathbb{H} \cong \mathbb{R}^3$.
	Letting $\phi = \begin{psmallmatrix} a \\ b \end{psmallmatrix}$ so that $\hat{L} = \mathcal{G}\phi\mathbb{H}$ with $\mathcal{D}^\mu_{ji} \phi_i = \phi_j$, define $\varphi := \mathcal{G}\phi = ea + \psi b$ so that
		\[
			\dif{}^\mu_{ji} \varphi_i = \varphi_j.
		\]
	Then we can directly verify that
		\[
			\eta_{ji} \varphi_i = \psi_j \dif{x}_{ij}^d a_i,
		\]
	while using the hermitian form \eqref{eqn:hermform},
		\[
			(\psi_j, \varphi_i) = (\psi_j, ea_i + \psi_i b_i) = a_i + \dif{x}_{ij} b_i.
		\]
	Thus, on any edge $(ij)$, using that $(\psi_i, \psi_i) = 0$, we have
		\begin{align*}
			(\varphi_j, \varphi_j)
				&= (\dif{}^\mu_{ji} \varphi_i, \dif{}^\mu_{ji} \varphi_i)  \\
				&= (\varphi_i, \varphi_i) + \mu\left(\overline{\dif{x}_{ij}^d a_i}( \psi_j , \varphi_i) + (\varphi_i,  \psi_j )\dif{x}_{ij}^d a_i\right) \\
				&= (\varphi_i, \varphi_i) - \frac{\mu}{m_{ij}}\left(\overline{a}_i  b_i  +  \overline{b}_i a_i\right) 
				= \left(1 -  \frac{\mu}{m_{ij}} \right) (\varphi_i, \varphi_i).
		\end{align*}
	Thus the non-degeneracy condition iterated in Remark~\ref{rema:nondeg} tells us that $(\varphi, \varphi) \equiv 0$ on $I$ if and only if it vanishes on one vertex $i \in I$.
	Finally, the relation $(\hat{\psi}, \hat{\psi}) = \overline{b}^{-1} (\varphi, \varphi) b^{-1}$ allows us to obtain the desired conclusion.
\end{proof}

In fact, we also obtain the discrete counterpart of Corollary~\ref{cor:2sphere} on curves in the $2$-sphere, where the proof is verbatim:
\begin{corollary}\label{cor:discs2}
	Let $x : (I, \frac{1}{m}) \to S^2$ be a curve into some $2$-sphere $S^2$. Then the Darboux transform $\hat{x}$ takes values in the same $2$-sphere if and only if $\hat{x}_i \in S^2$ for some $i \in I$.
\end{corollary}

\subsection{Discrete monodromy}\label{subsect:discMon}
The discrete monodromy was investigated in \cite{matthaus_discrete_2003} (see also \cite{pinkall_new_2007}) for the case of Darboux transformations of planar curves with normalised polarisation so that $m_{ij} \equiv 1$.
In this section, we obtain a generalisation of this result; in fact, the result is immediate due to the parallel sections formulation of Darboux transformations.

As in the smooth case, we call a $(\dif{}^{\mu})^P$--parallel section $\hat{L}$ a \emph{global parallel section} if $\hat{L}_n = \hat{L}_{n + M}$ for all $n \in I$ for some fixed $M \in \mathbb{Z}$, that is, $\phi = \begin{psmallmatrix} a \\ b \end{psmallmatrix}$ is a \emph{section with a multiplier}.
 
Supposing that $\hat{L}$ is a Darboux transform of $L$, that is, $\hat{L}$ is $(\dif{}^{\mu})^P$--parallel, we have that
	\[
		\hat{L}_{n + M} = \hat{\psi}_{n + M}\mathbb{H} = \left(\prod_{\kappa = n}^{n+M - 1}\dif{}^{\mu}_{( \kappa , \kappa+1)}\right) \hat{\psi}_{n}\mathbb{H} = \hat{L}_{n}.
	\]
Therefore, denoting the \emph{monodromy matrix} as 
	\[
		\mathcal{M}_{r, \mu} :=  \prod_{\kappa = n}^{n+M - 1}\dif{}^\mu_{(\kappa,\kappa+1)},
	\]
we see that finding sections with multipliers amounts to finding the eigenvectors of $\mathcal{M}_{r, \mu}$ (see, for example, Figures~\ref{fig:DTdisc1} and \ref{fig:DTdisc2}).

\begin{figure}
	\centering
	\begin{minipage}{0.45\textwidth}
		\includegraphics[width=\linewidth]{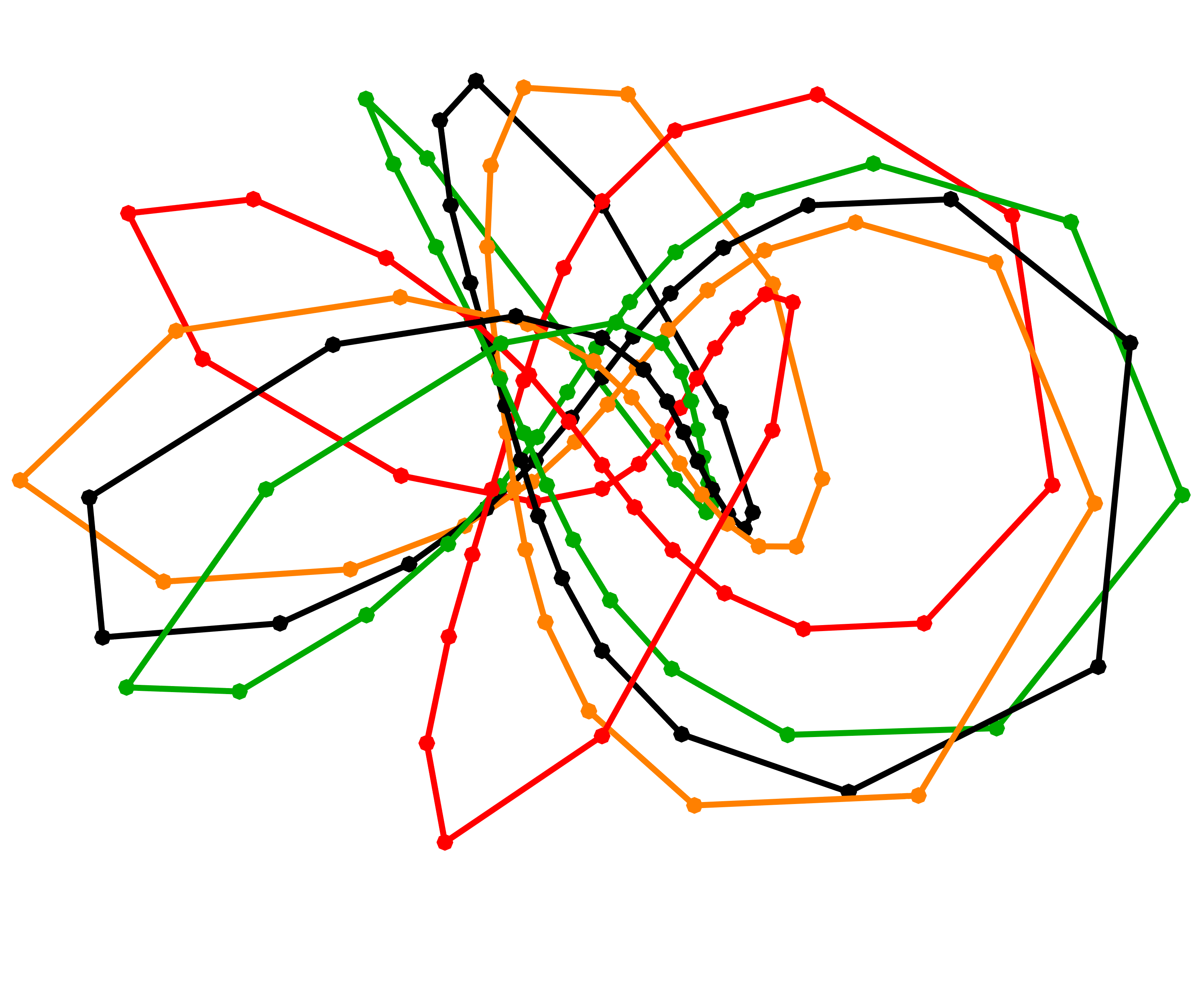}
	\end{minipage}
	\begin{minipage}{0.45\textwidth}
		\includegraphics[width=\linewidth]{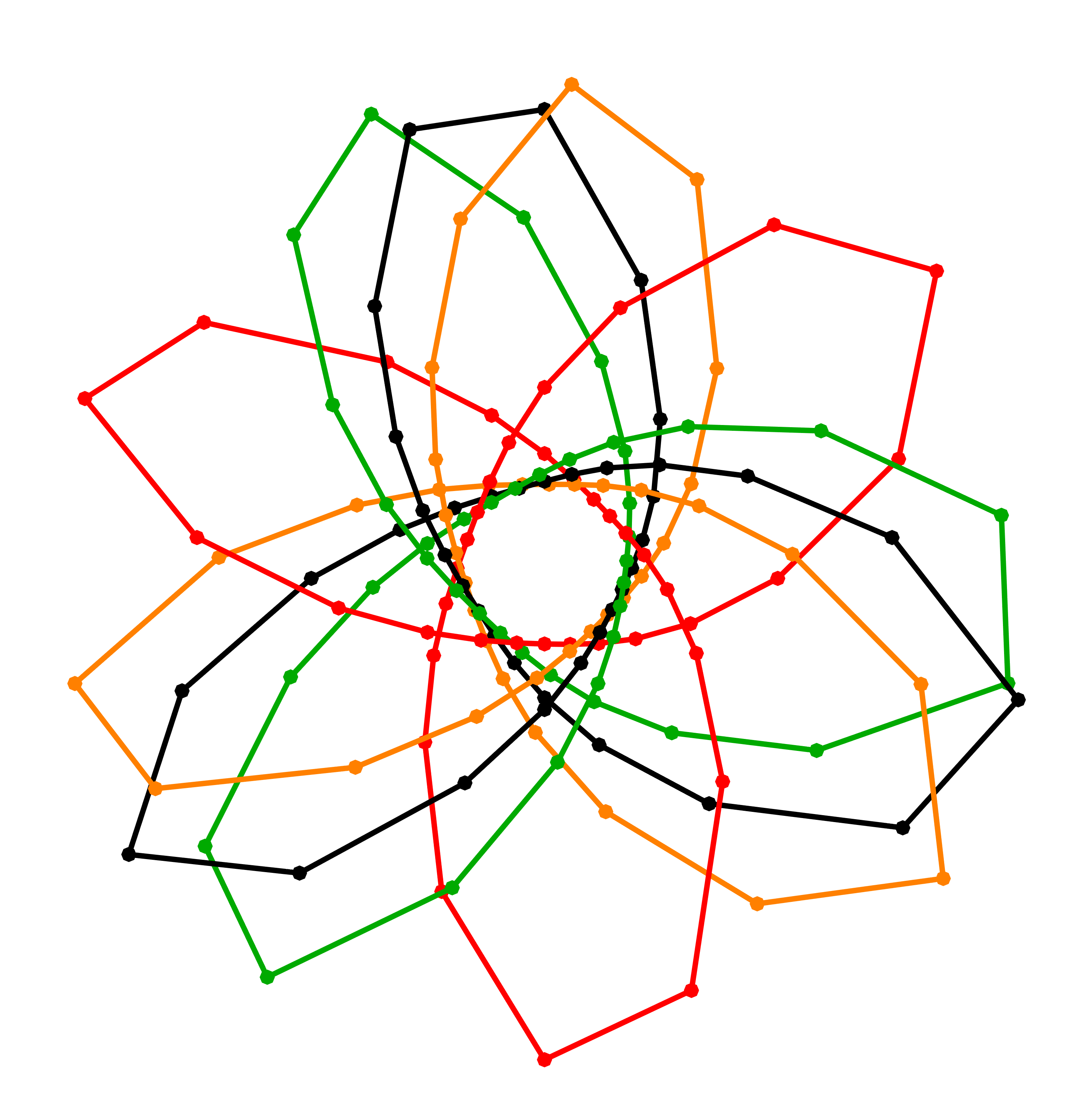}
	\end{minipage}
	\caption{Original $(2,3)$-discrete torus knot (in black), and its closed Darboux transforms (on the left), also viewed from the top (on the right).}
	\label{fig:DTdisc1}
\end{figure}

\begin{figure}
	\centering
	\begin{minipage}{0.45\textwidth}
		\includegraphics[width=\linewidth]{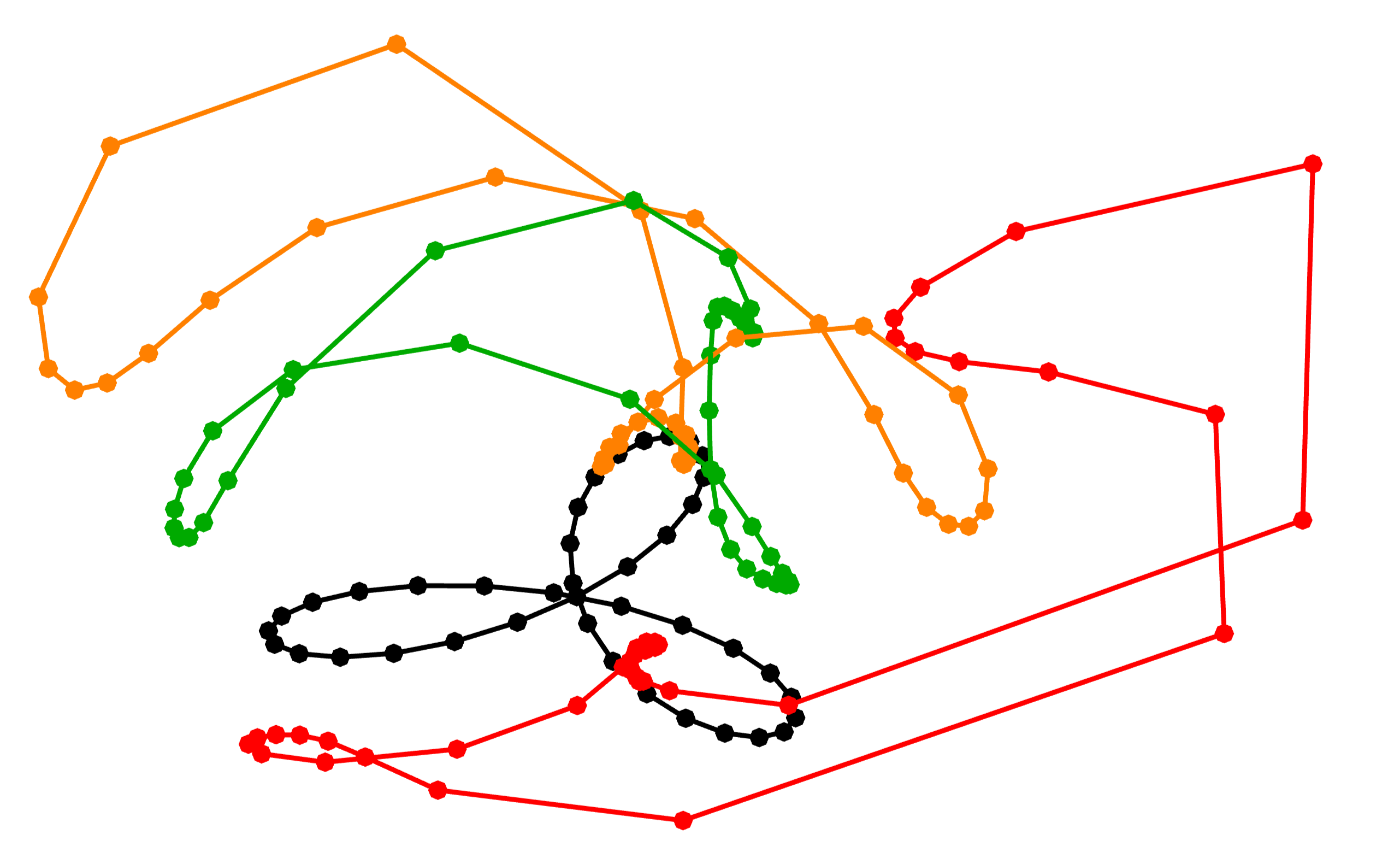}
	\end{minipage}
	\caption{Original planar curve given by non-uniform samplings of $x(t) = \jj \cos(3t) e^{i t}$ (in black), and its closed Darboux transforms.}
	\label{fig:DTdisc2}
\end{figure}

We illustrate this with the next example:
\begin{example}\label{exam:disc}
	In this example, we offer an explicit discrete parametrisation of all planar Darboux transforms of the discrete polarised circle.
	Suppose $x: (I, \frac{1}{m}) \to \mathbb{H}$ given by $x_n = \jj e^{\frac{2 \pi \ii}{M} n}$ for some $M \in \mathbb{N}$ is polarised by $m_{ij} = | 1 - e^\frac{2\pi \ii}{M}|^{-2}$ for any edge $(ij)$.
	To calculate the Darboux transforms, we will find $\phi = \begin{psmallmatrix} a \\ b \end{psmallmatrix}$ satisfying the difference relations \eqref{eqn:diffPar}.
	
	To do this, first we eliminate the $b$ from the pair of difference relations by noting that on any three consecutive vertices $(ijk)$, we have
		\begin{equation}\label{eqn:bdiff}
			\dif{b}_{ij} = b_i - b_j = -(\dif{x}_{ij})^{-1}\dif{a}_{ij} + (\dif{x}_{jk})^{-1}\dif{a}_{jk}
		\end{equation}
	so that we obtain a second order linear recurrence relation on $a$:
		\begin{equation}\label{eqn:diffPar2}
			\mu\dif{x}^d_{ij} a_i  -(\dif{x}_{ij})^{-1}\dif{a}_{ij} + (\dif{x}_{jk})^{-1}\dif{a}_{jk} = 0.
		\end{equation}
	Conversely, every solution $a$ to the recurrence relation \eqref{eqn:diffPar2}, and $b$ defined via \eqref{eqn:bdiff} gives a Darboux transform of $x$.
	
	In the specific case of discrete circles, the recurrence relation \eqref{eqn:diffPar2} reads
		\[
			e^\frac{2\pi \ii}{M} a_k - (1 + e^\frac{2\pi \ii}{M})a_j + (1 - \hat{\mu}| 1 - e^\frac{2\pi \ii}{M}|^2) a_i = 0.
		\]
	Writing $a = a_0 + \jj a_1$ for some complex valued discrete functions $a_0$ and $a_1$ on $I$, we then have
		\[
			a = (c_0^- a_0^- + c_0^+ a_0^+) + \jj (c_1^- a_1^- + c_1^+ a_1^+)
		\]
	for some constants of integration $c_0^\pm, c_1^\pm \in \mathbb{C}$ where
		\begin{align*}
			a_{0,n}^\pm &= \left(\frac{1}{2}\left(e^{-\frac{2\pi \ii}{M}} (1 \mp s) + (1 \pm s)\right)\right)^n \\
			a_{1,n}^\pm &= \left(\frac{1}{2}\left(e^\frac{2\pi \ii}{M} (1 \pm s) + (1 \mp s)\right)\right)^n,
		\end{align*}
	where $s = \sqrt{1 - 4\mu}$ as before.
	Then using $b_i = -(\dif{x}_{ij})^{-1} \dif{a}_{ij}$, we find that $b =: b_0 + \jj b_1$ for
		\begin{align*}
			b_{0,n} &= -\frac{1}{2} e^{-\frac{2\pi \ii}{M}n}\left(c_1^- (1- s) a_{1,n}^- + c_1^+ (1 + s) a_{1,n}^+\right) \\
			b_{1,n} &= \frac{1}{2} e^{\frac{2\pi \ii}{M}n}\left(c_0^- (1+ s) a_{0,n}^- + c_0^+ (1 - s) a_{0,n}^+\right).
		\end{align*}
	Therefore, as in the smooth case, Lemma~\ref{lemm:r3disc} ensures that the Darboux transform takes values in $\Im \mathbb{H} \cong \mathbb{R}^3$ by choosing the constants so that
		\[
			\Re (c_0^- \overline{c_1^-} - c_0^+ \overline{c_1^+}) = 0,
		\]
	while Corollary~\ref{cor:discs2} implies that the Darboux transform takes values in the $\jj \kk$--plane by choosing the constants so that
		\[
			c_0^- \overline{c_1^-} - c_0^+ \overline{c_1^+} = 0.
		\]
	For example, letting $c_0^\pm = 0$ so that $a_0 \equiv 0 \equiv b_1$, we have that
		\[
			\hat{x}_n = \jj\scalebox{0.85}{$\left(\dfrac{-e^{\frac{2 \pi \ii}{M} n}\left(c_1^+ (1 - s)(e^{\frac{2 \pi \ii}{M}}(1 + s) + (1 - s))^n + c_1^- (1+s)(e^{\frac{2 \pi \ii}{M}}(1 - s) + (1 + s))^n \right)}
				{ c_1^+ (1 + s)(e^{\frac{2 \pi \ii}{M}}(1 + s) + (1 - s))^n + c_1^- (1-s)(e^{\frac{2 \pi \ii}{M}}(1 - s) + (1 + s))^n }\right)$}.
		\]

	To consider the monodromy, note that by  the difference relations \eqref{eqn:diffPar} on $a$ and $b$, we see that $\varphi$ is periodic if and only if $a$ is.
	Since the period of the discrete circle is $M$, we calculate that $(a_*^\pm)_{n+M} = (a_*^\pm)_n h^\pm$ for $* = 0, 1$ where
		\[
			h^\pm = (a_0^\pm)_M = (a_1^\pm)_M.
		\]
	Therefore, the resonance points occur when $h^+ = h^-$, that is,
		\[
			\left(\frac{e^\frac{2\pi \ii}{M} (1 + s) + (1 - s)}{e^\frac{2\pi \ii}{M} (1 - s) + (1 + s)}\right)^M = e^{2\pi \ii k}
		\]
	for some $k \in \mathbb{Z}$.
	Hence, we have resonance points when
		\[
			\mu =\frac{1}{4}\left(1 - \cot^2\frac{\pi}{M}\tan^2\frac{k \pi}{M}\right).
		\]
	for $k \in \mathbb{Z}$.
	For spatial and planar examples of closed Darboux transforms of the discrete circle, see Figures~\ref{fig:spatDar} and \ref{fig:planDar}, respectively.
	\begin{remark}
		We remark here that similar to the smooth case, Corollary~\ref{cor:discs2} implies every Darboux transform of the discrete circle must be contained in some $2$-sphere, determined by the circle and an initial point of the Darboux transform.
	\end{remark}
	\begin{figure}
		\centering
		\begin{minipage}{0.48\textwidth}
			\includegraphics[width=\linewidth]{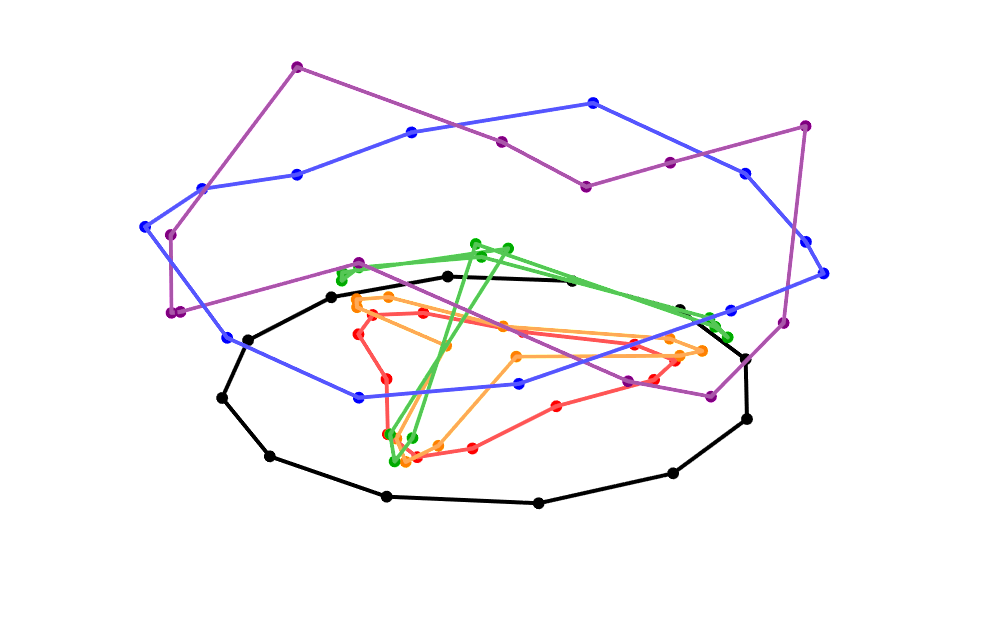}
		\end{minipage}
		\begin{minipage}{0.48\textwidth}
			\includegraphics[width=\linewidth]{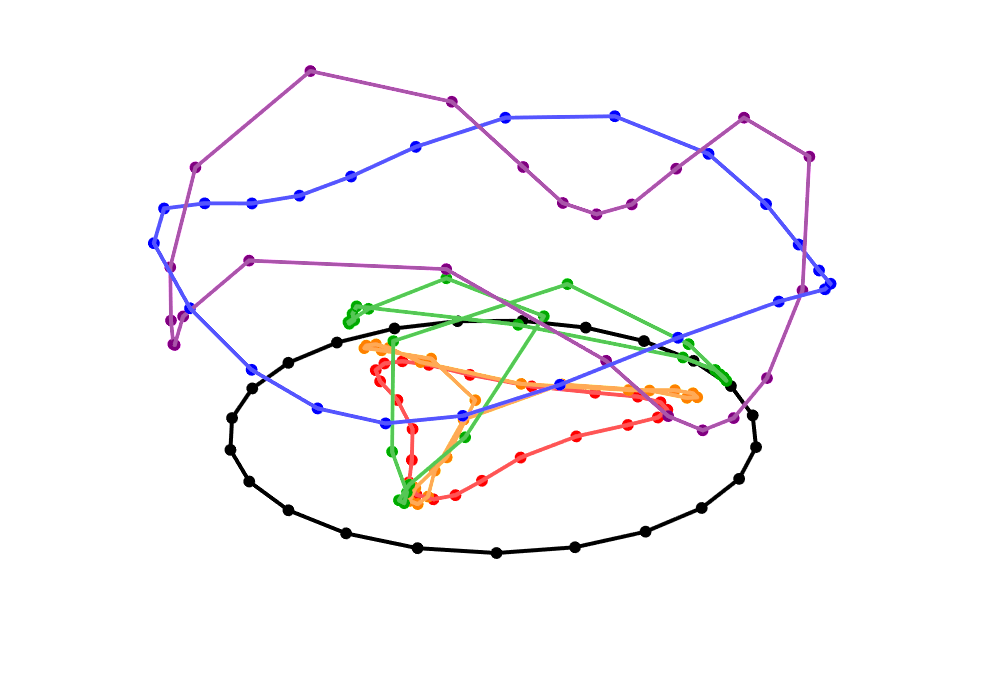}
		\end{minipage}
		\\
		\begin{minipage}{0.48\textwidth}
			\includegraphics[width=\linewidth]{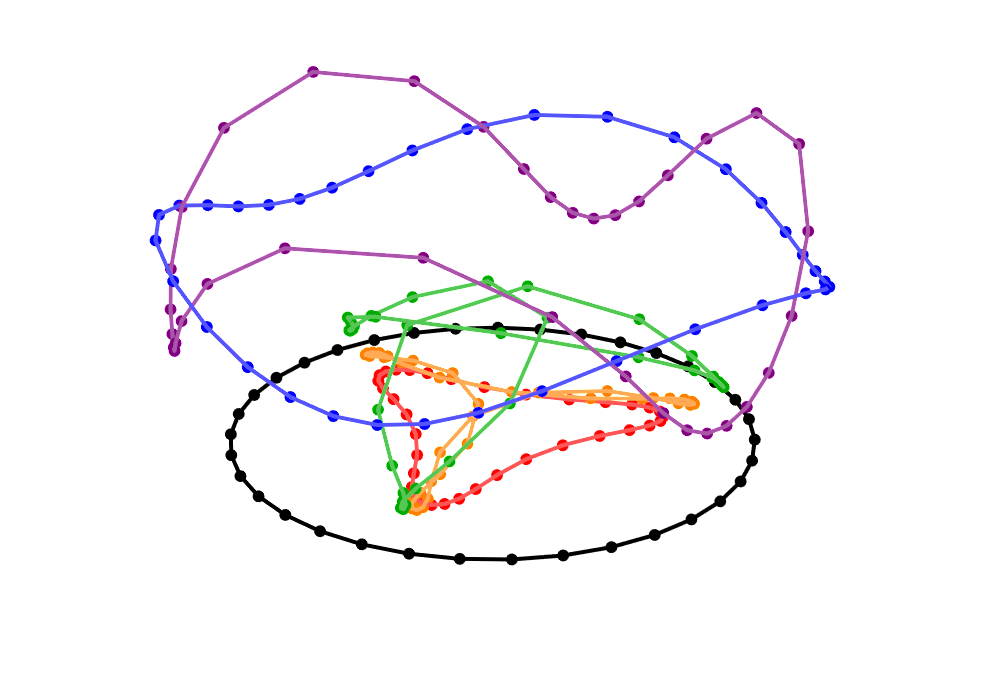}
		\end{minipage}
		\begin{minipage}{0.48\textwidth}
			\includegraphics[width=\linewidth]{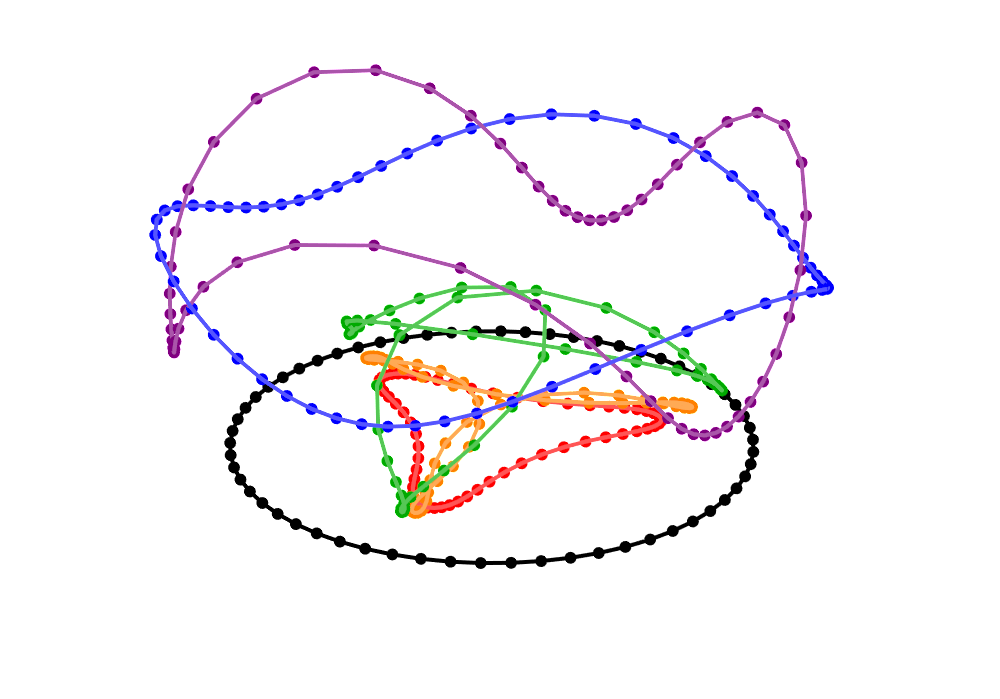}
		\end{minipage}
		\caption{Spatial Darboux transforms of the discrete circle (drawn in black) at the resonance point with $k = 3$. The figures are drawn with the same constants as in the smooth spatial curves of Figure~\ref{fig:DTsmooth} ($c_0^+ = 0.5\ii, c_0^- = 0, c_1^+ = 1, c_1^- = -4, -2, -1.2, -0.1, 0.25$), but with varying number of points on the initial circle ($M = 12$, $23$, $35$, $60$).}
		\label{fig:spatDar}
	\end{figure}
	\begin{figure}
		\centering
		\begin{minipage}{0.35\textwidth}
			\includegraphics[width=\linewidth]{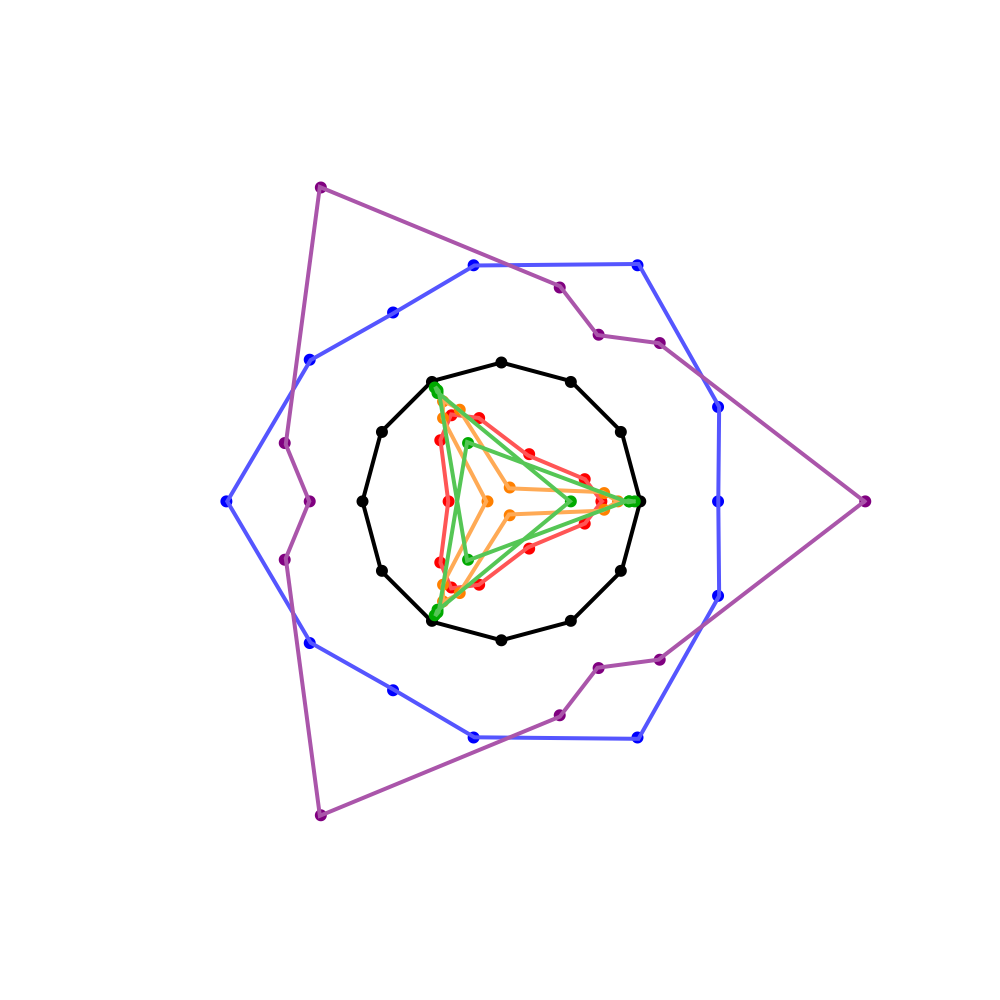}
		\end{minipage}
		\begin{minipage}{0.35\textwidth}
			\includegraphics[width=\linewidth]{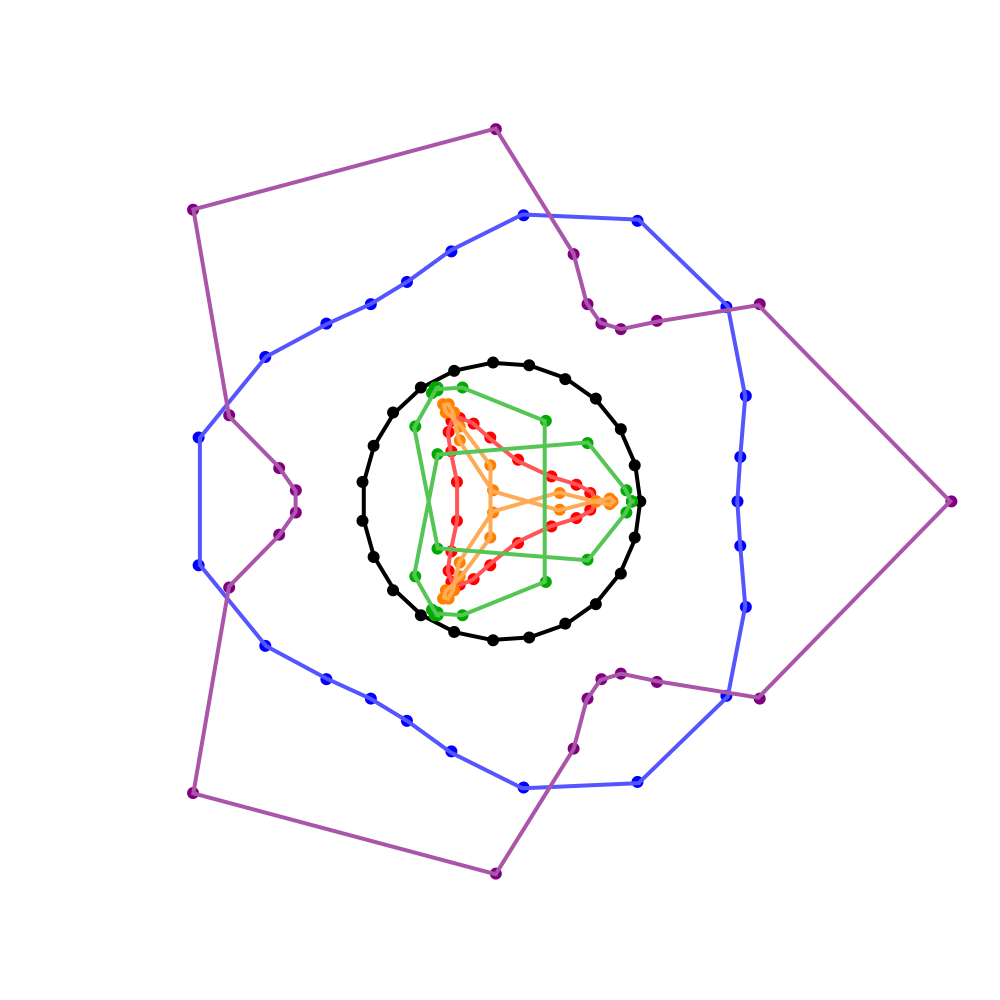}
		\end{minipage}
		\\
		\begin{minipage}{0.35\textwidth}
			\includegraphics[width=\linewidth]{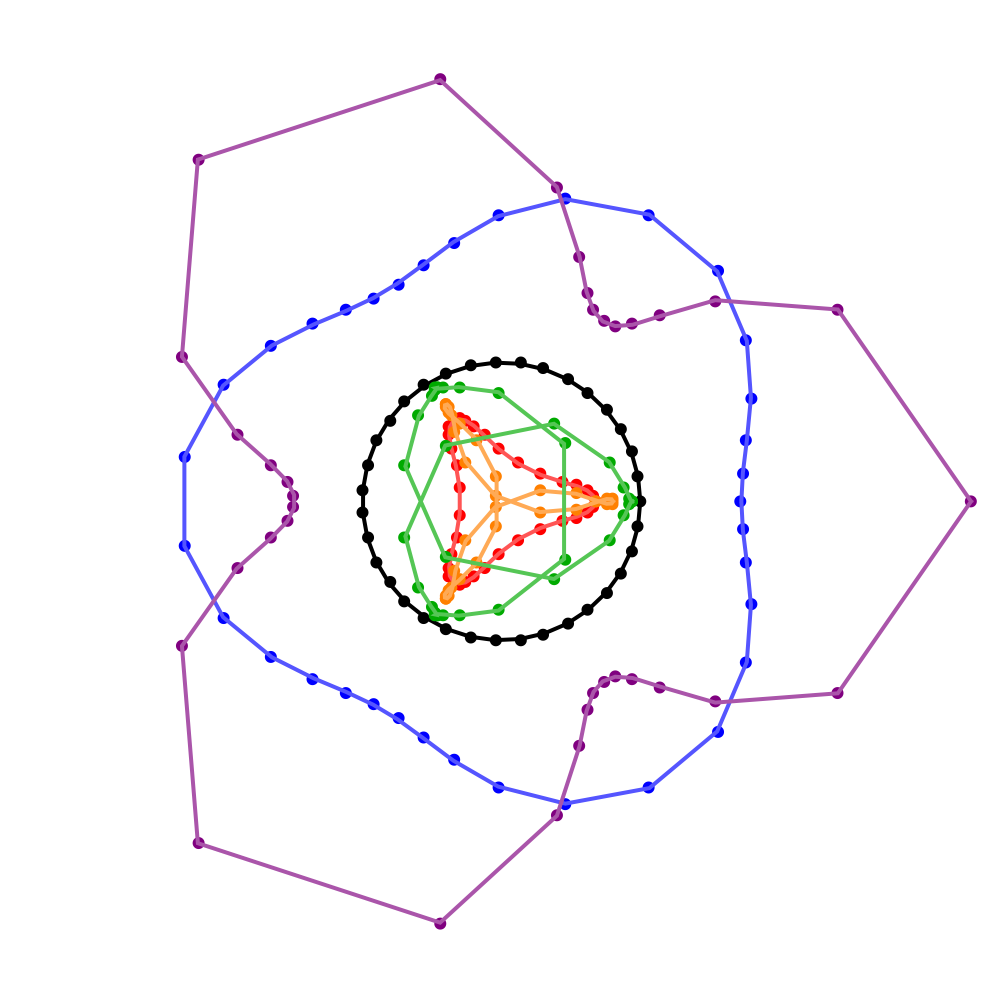}
		\end{minipage}
		\begin{minipage}{0.35\textwidth}
			\includegraphics[width=\linewidth]{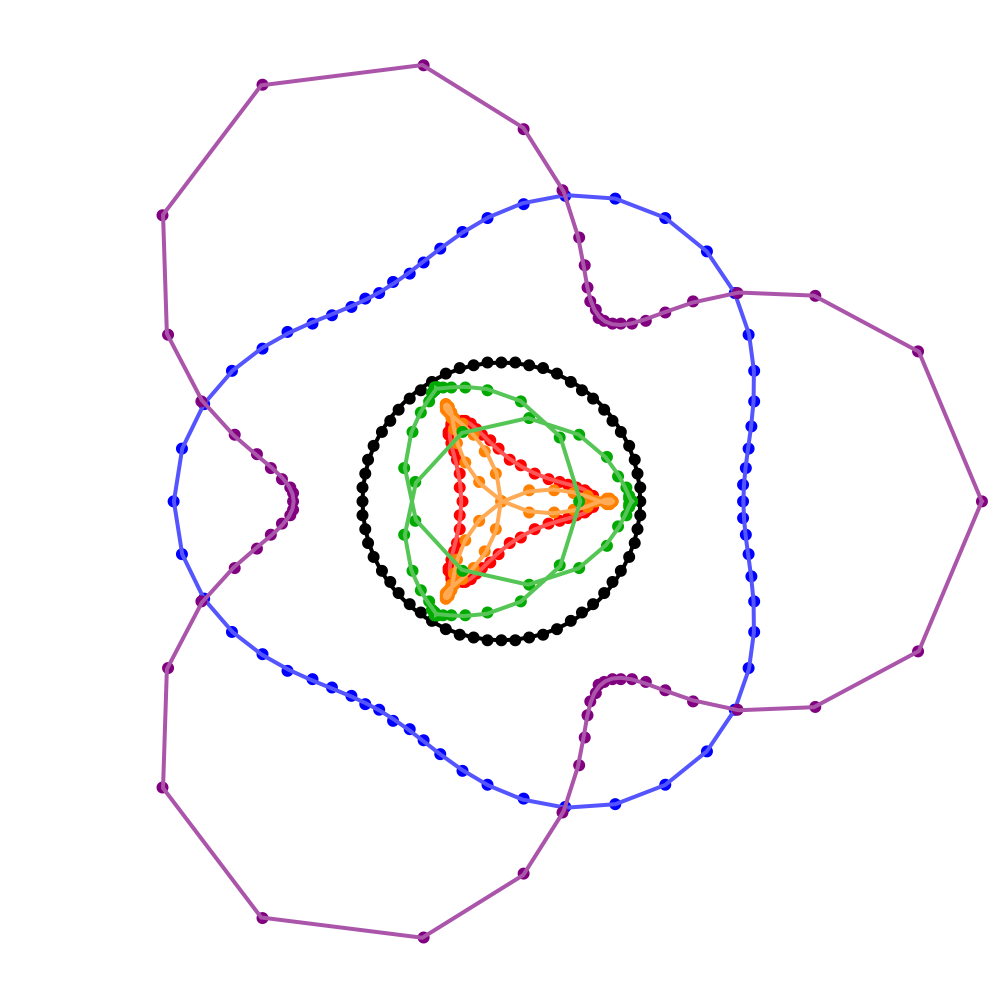}
		\end{minipage}
		\caption{Planar Darboux transforms of the discrete circle (drawn in black) at the resonance point with $k = 3$. The figures are drawn with the same constants as in the smooth planar curves of Figure~\ref{fig:DTsmooth} ($c_0^+ = 0, c_0^- = 0, c_1^+ = 1, c_1^- = -4, -2, -1.2, -0.1, 0.25$), but with varying number of points on the initial circle ($M = 12$, $23$, $35$, $60$).}
		\label{fig:planDar}
	\end{figure}	
\end{example}

\subsection{Discrete bicycle correspondence and bicycle monodromy}
Analogous to the smooth case, the discrete case also allows for an integrable reduction to the case when both of the curves forming the Darboux pair are arc-length polarised.
To see this, we first recall the definition of a discrete arc-length polarisation \cite{cho_discrete_2021-1}.
\begin{definition}
	A discrete polarised curve $x : (I, \frac{1}{m}) \to \mathbb{H}$ is \emph{arc-length polarised} if on any edge $(ij)$,
		\[
			|{\dif{x}_{ij}}|^2 = \frac{1}{m_{ij}}.
		\]
\end{definition}
\begin{figure}
		\centering
		\begin{minipage}{0.47\textwidth}
			\includegraphics[width=0.9\linewidth]{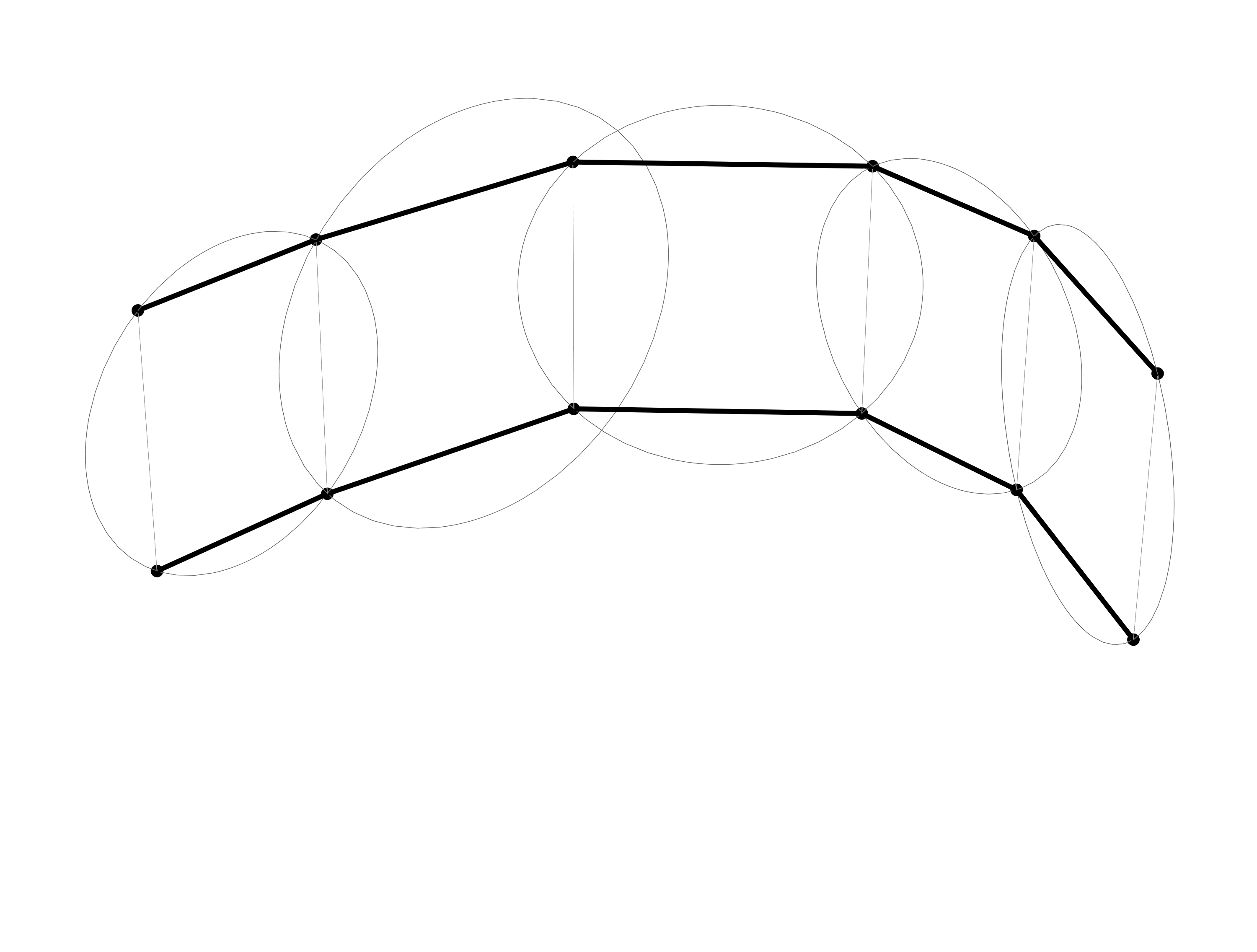}
		\end{minipage}
		\begin{minipage}{0.47\textwidth}
			\includegraphics[width=0.9\linewidth]{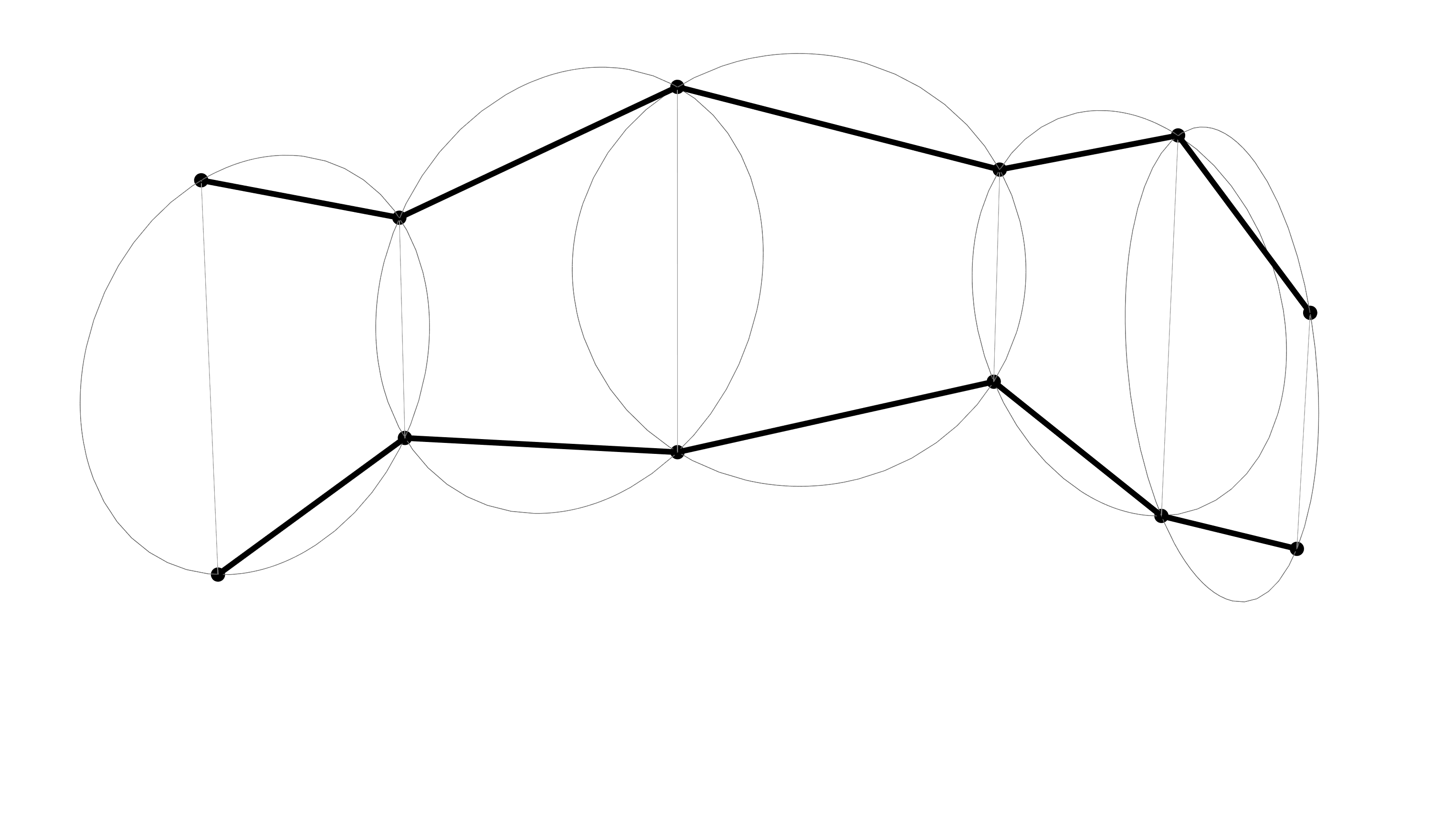}
		\end{minipage}
	\caption{Two examples of Darboux pairs of discrete arc-length polarised space curves that are reflected across a plane through centers of the discrete circle congruence.}
	\label{fig:discCount}
\end{figure}
As in the smooth case, reflections of certain discrete arc-length polarised curves are Darboux transforms (see Figure~\ref{fig:discCount}); treating these cases as trival cases, we obtain a characterisation of non-trivial Darboux transformations keeping the arc-length polarisation:
\begin{theorem}\label{thm:discB}
	Let  $x : (I, \frac{1}{m}) \to \mathbb{H}$ be arc-length polarised.
	Then a (non-trivial) Darboux transform $\hat{x}: (I, \frac{1}{m}) \to \mathbb{H}$ with spectral parameter $\mu$ is also arc-length polarised if and only if $|\hat{x}_i - x_i|^2 = \frac{1}{\mu}$ on some vertex $i \in I$.
\end{theorem}

\begin{proof}
	As in the statement of the proof, let $x$ be a discrete arc-length polarised curve.
	To show the necessary direction, assume that $|T_i|^2 = |\hat{x}_i - x_i|^2 = \frac{1}{\mu}$ for some $i \in I$.
	From the discrete Riccati equation \eqref{eqn:riccati1}, we can calculate that on the edge $(ij)$
		\[
			\hat{x}_j = (\mu x_j \overline{\dif{x}_{ij}}T_i + \hat{x}_i)(1 + \mu\overline{\dif{x}_{ij}}T_i)^{-1},
		\]
	so that
		\[
			\dif{\hat{x}}_{ij} = \mu(\hat{x}_i - x_j)\overline{\dif{x}_{ij}}T_i (1 + \mu\overline{\dif{x}_{ij}}T_i)^{-1}.
		\]
	On the other hand, since $x$ is arc-length polarised,
		\begin{align*}
			|1 + \mu\overline{\dif{x}_{ij}}T_i|^2 &= 1 + \mu(\overline{T}_i \dif{x}_{ij} + \overline{\dif{x}_{ij}}T_i) + \frac{\mu}{m_{ij}} \\
				&= \mu \left(\overline{T}_i T_i + \overline{T}_i \dif{x}_{ij} + \overline{\dif{x}_{ij}}T_i +  \overline{\dif{x}_{ij}} \dif{x}_{ij} \right) \\
				&= \mu |T_i + \dif{x}_{ij}|^2 = \mu |\hat{x}_i - x_j|^2.
		\end{align*}
	Therefore, we have that
		\[
			|{\dif{\hat{x}}_{ij}}|^2 = \frac{|\mu(\hat{x}_i - x_j)\overline{\dif{x}_{ij}}T_i|^2}{|1 + \mu\overline{\dif{x}_{ij}}T_i|^{-2}}
				= \frac{\mu|\hat{x}_i - x_j|^2 |{\dif{x}_{ij}}|^2}{\mu|\hat{x}_i - x_j|^2} = \frac{1}{m_{ij}}.
		\]
	Then, the discrete Riccati equation \eqref{eqn:riccati1} allows us to see that $|T_j|^2 = \frac{1}{\mu}$.
	Propagating the above proof for any edge $(ij)$, we have that $\hat{x}$ is discrete arc-length polarised.
	
	To see the sufficient direction, assume that $|{\dif{x}_{ij}}|^2 = |{\dif{\hat{x}}_{ij}}|^2 = \frac{1}{m_{ij}}$; hence, the cross-ratios condition \eqref{eqn:cr} implies 
		\begin{equation}\label{eqn:titj}
			|T_i|^2 |T_j|^2 = \frac{1}{\mu^2}.
		\end{equation}
	On the other hand, by Remark~\ref{rema:cratio}, we have that the circular quadrilateral formed by $x_i, x_j, \hat{x}_j, \hat{x}_i$ is an isosceles trapezoid and hence symmetric; therefore, we must have either
		\[
			|T_i|^2 = |T_j|^2 \quad\text{or}\quad |\hat{x}_i - x_j|^2 = |\hat{x}_j - x_i|^2.
		\]
	If $|T_i|^2 \neq |T_j|^2$ on any one edge $(ij)$, then by \eqref{eqn:titj} we must have that $|T_i|^2 \neq |T_j|^2$ on every edge $(ij)$.
	Therefore, on every edge, we have $|\hat{x}_i - x_j|^2 = |\hat{x}_j - x_i|^2$ so that $T_i \parallel T_j$, and the symmetry of quadrilaterals then forces the discrete curve $\hat{x}$ to be a reflection of $x$, allowing us to exclude this case.
	
	Hence, we see that $T_i$ and $T_j$ are not parallel on every edge $(ij)$ so that the circular quadrilateral is non-embedded.
	Then the cross-ratios condition \eqref{eqn:cr} reads
		\[
			\cratio(x_i, x_j, \hat{x}_j, \hat{x}_i) = \frac{\mu}{m_{ij}} > 0.
		\]
	Now since $\frac{1}{m}$ is an arc-length polarisation, we have that $m_{ij} > 0$ on every edge $(ij)$ so that $\mu > 0$.
	
	Thus, we conclude that $|T_i|^2 = |T_j|^2 = \frac{1}{\mu}$ on every vertex $i \in I$.
\end{proof}

\begin{figure}
	\centering
	\begin{minipage}{0.32\textwidth}
		\includegraphics[width=\linewidth]{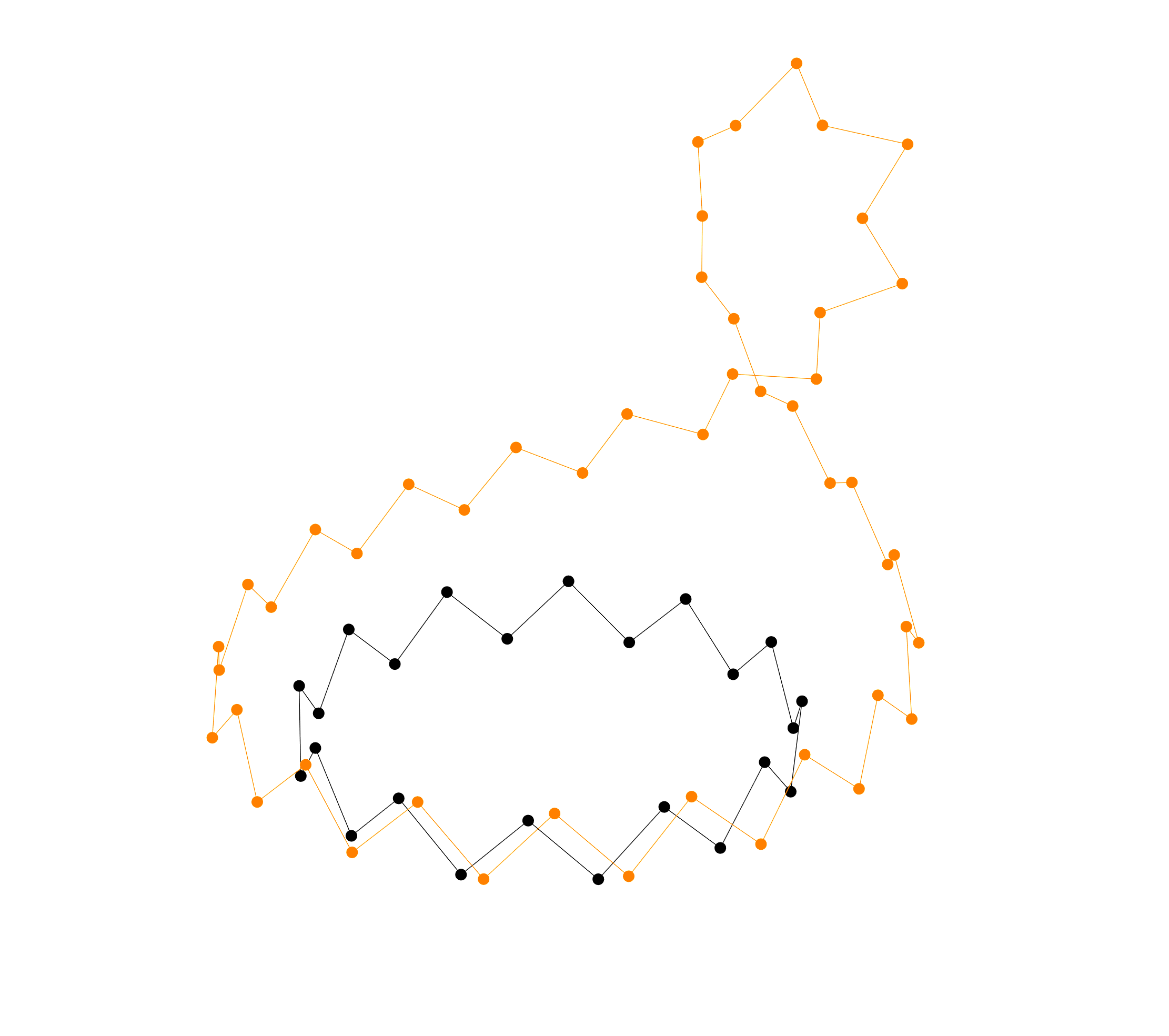}
	\end{minipage}
	\begin{minipage}{0.32\textwidth}
		\includegraphics[width=\linewidth]{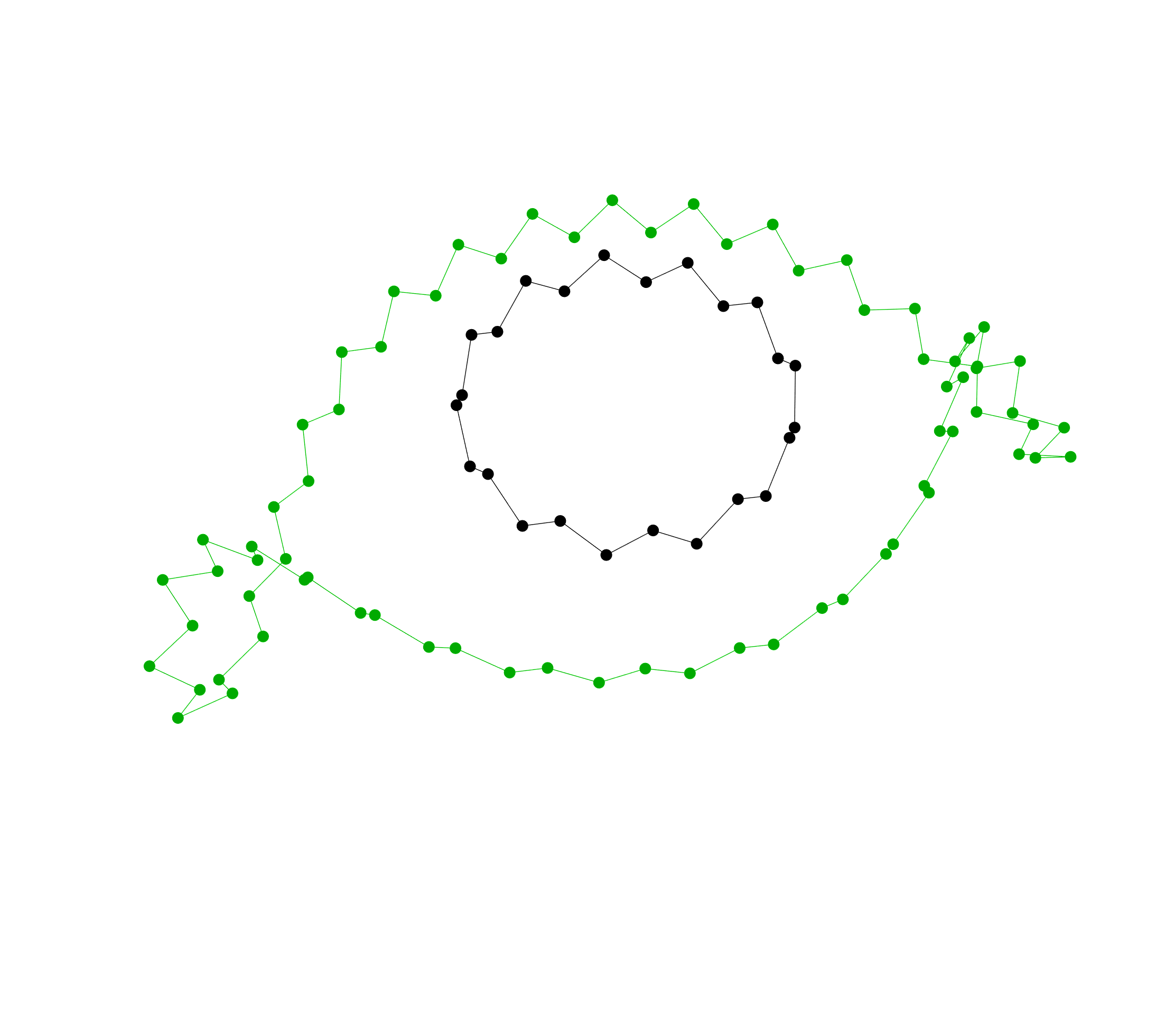}
	\end{minipage}
	\begin{minipage}{0.32\textwidth}
		\includegraphics[width=\linewidth]{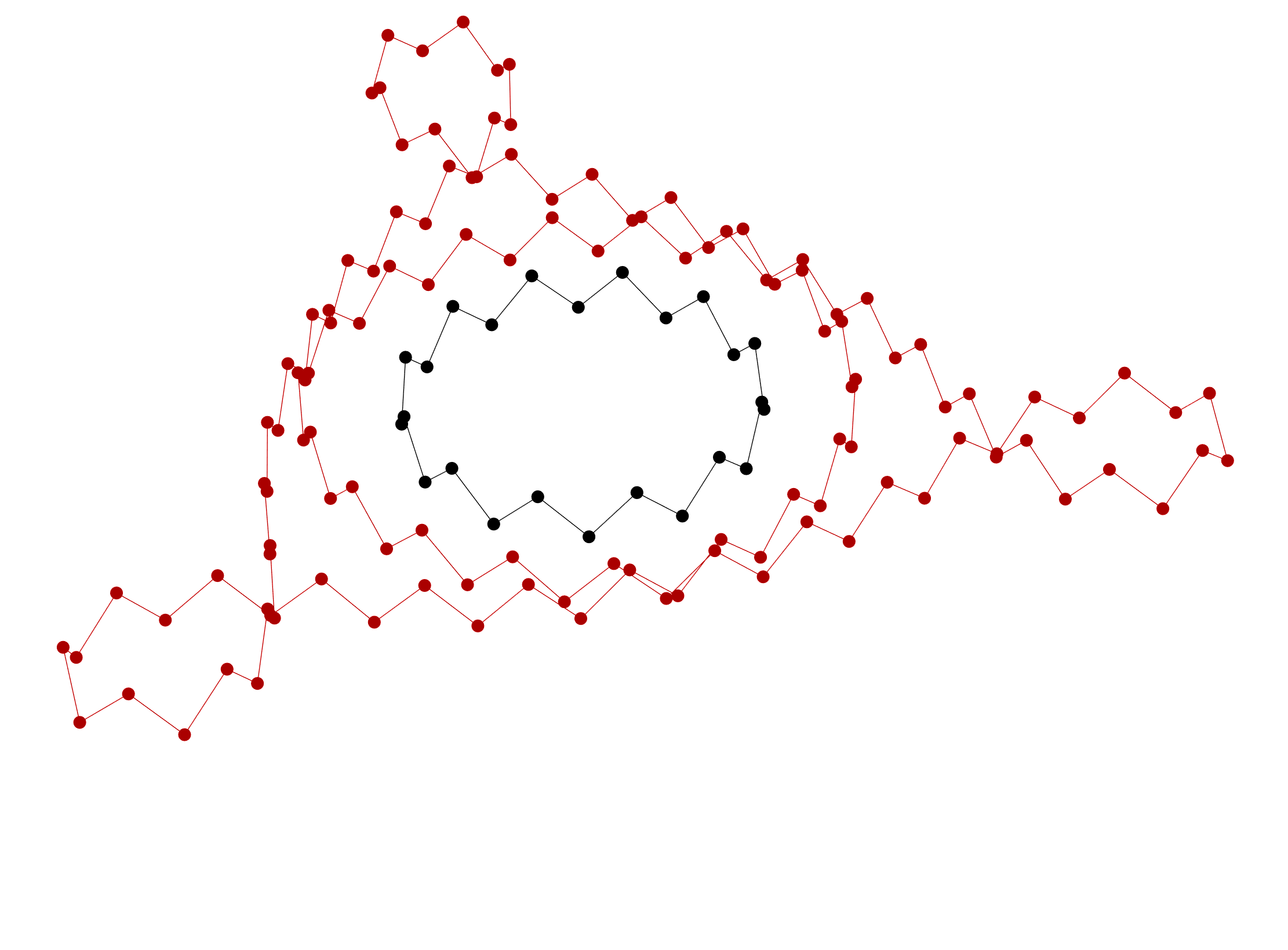}
	\end{minipage}
	\caption{Closed bicycle correspondences of a non-planar discrete curve.}
	\label{fig:discB}
\end{figure}
As in the smooth case, the discrete Darboux transformation keeping the arc-length polarisation is known as a \emph{discrete bicycle correspondence} \cite{tabachnikov_discrete_2013}.
Applying the monodromy problem, we can obtain the discrete bicycle monodromy with examples given in Figure~\ref{fig:discB}.
	
\begin{figure}
	\centering
	\begin{minipage}{0.25\textwidth}
		\includegraphics[width=\linewidth]{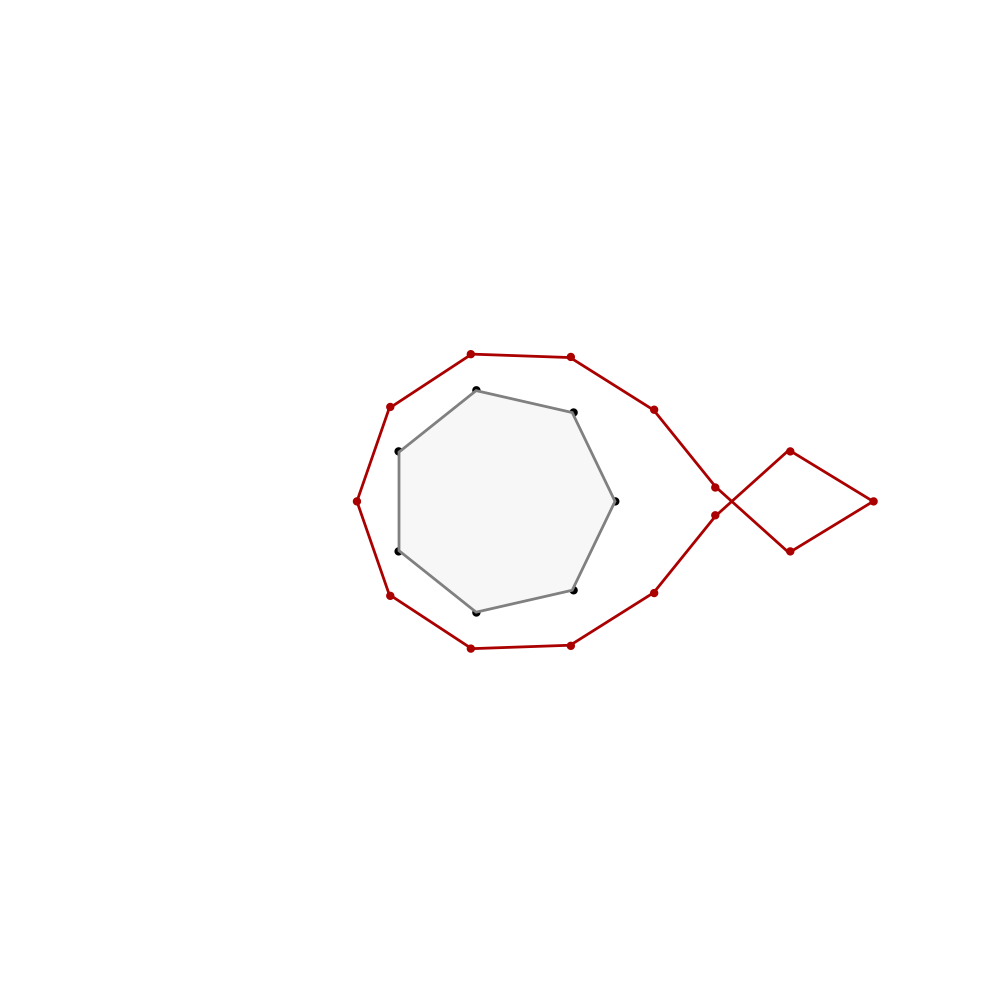}
	\end{minipage}
	\quad
	\begin{minipage}{0.25\textwidth}
		\includegraphics[width=\linewidth]{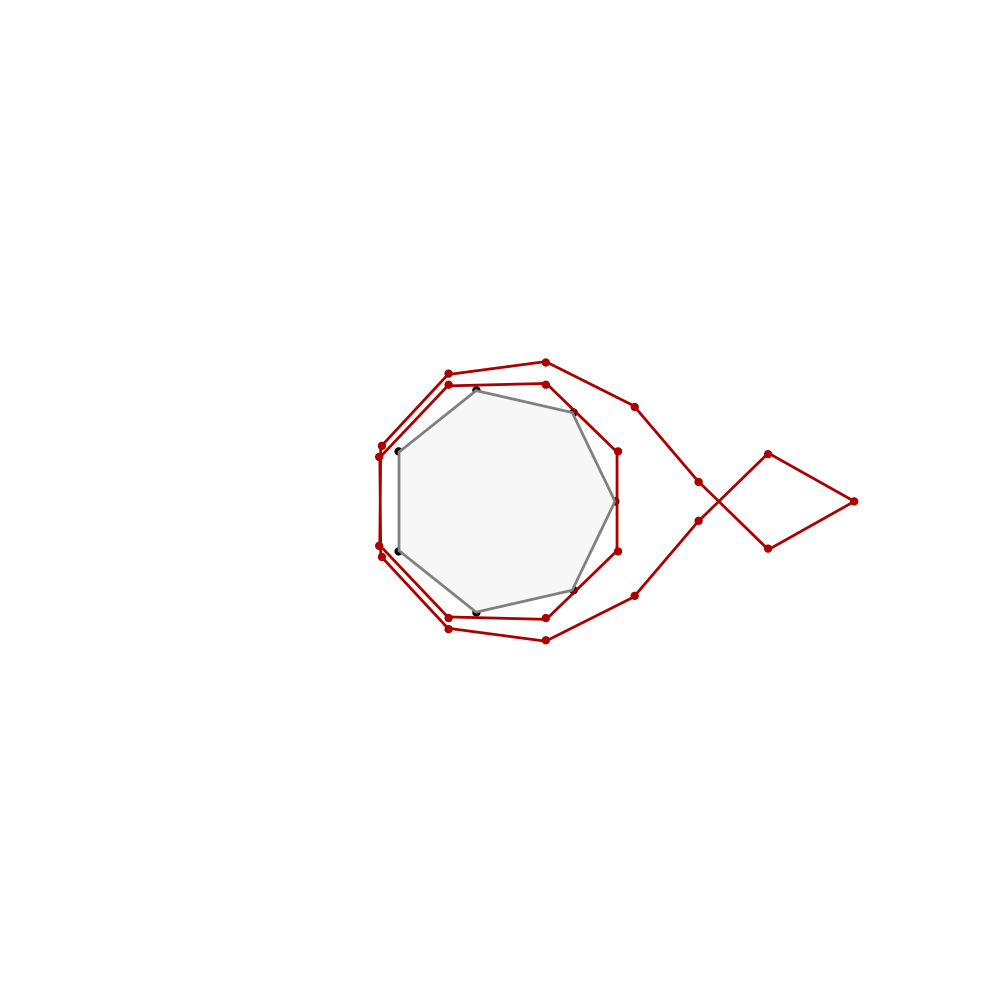}
	\end{minipage}
	\quad
	\begin{minipage}{0.25\textwidth}
		\includegraphics[width=\linewidth]{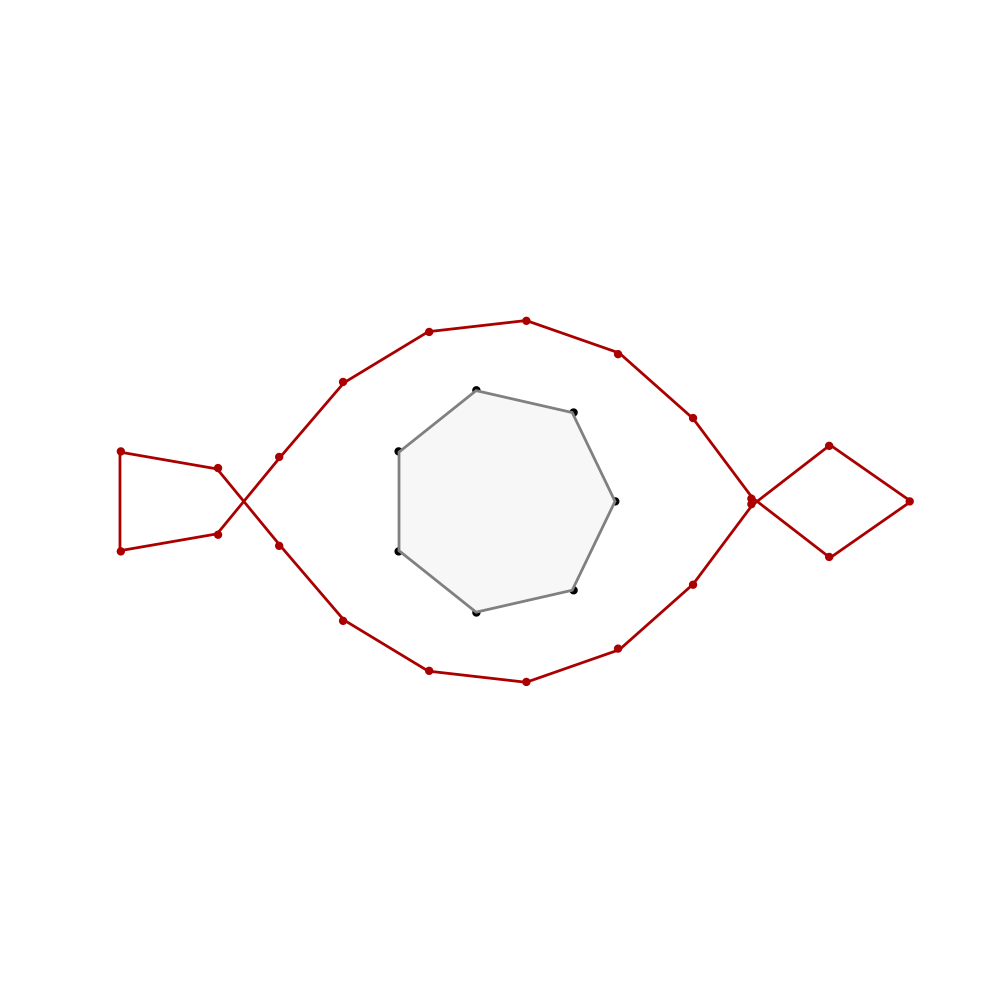}
	\end{minipage}
	\\
	\begin{minipage}{0.25\textwidth}
		\includegraphics[width=\linewidth]{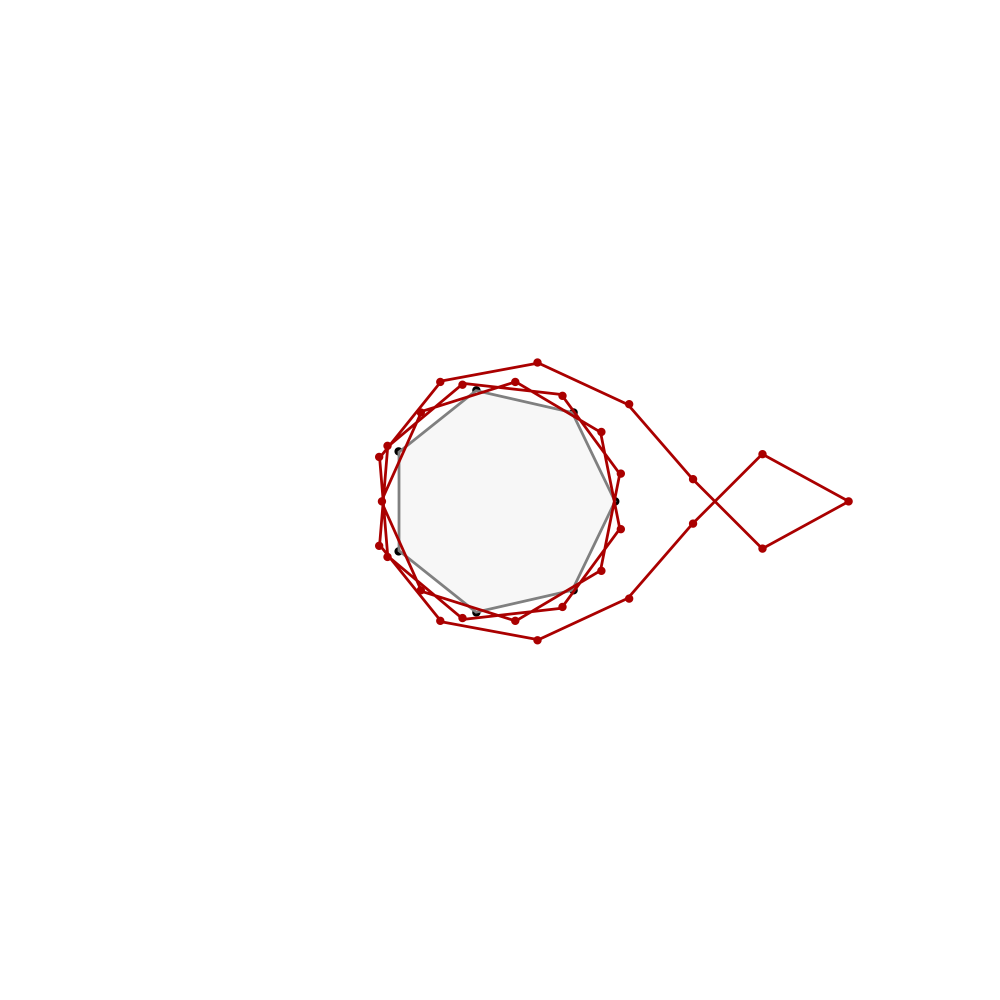}
	\end{minipage}
	\quad
	\begin{minipage}{0.25\textwidth}
		\includegraphics[width=\linewidth]{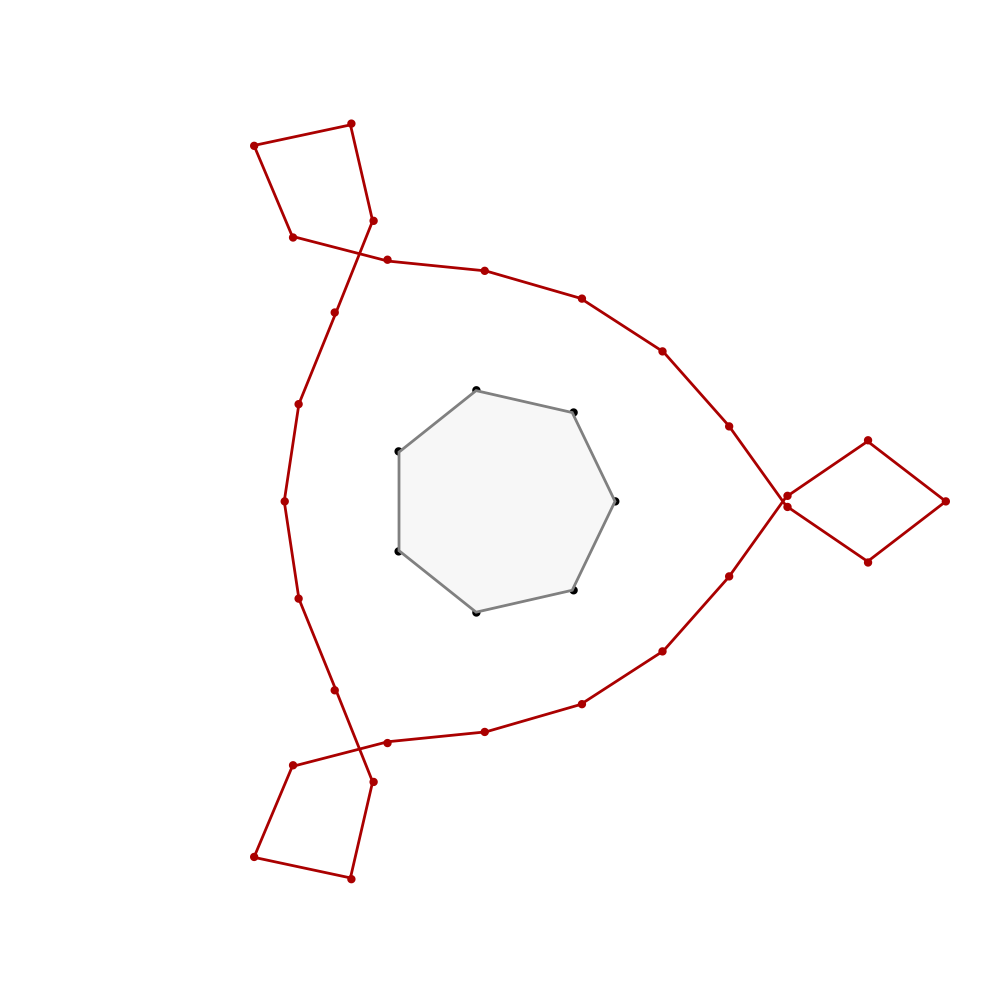}
	\end{minipage}
	\quad
	\begin{minipage}{0.25\textwidth}
		\includegraphics[width=\linewidth]{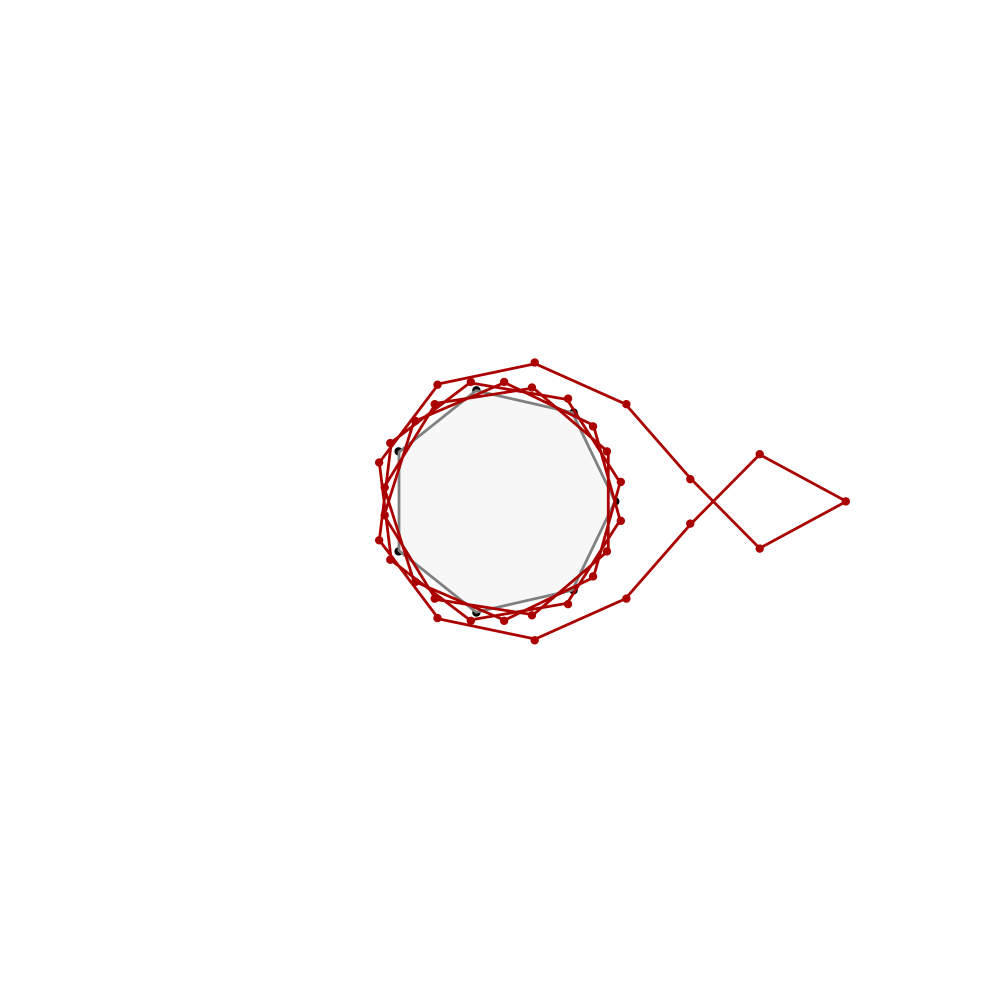}
	\end{minipage}
	\\
	\begin{minipage}{0.25\textwidth}
		\includegraphics[width=\linewidth]{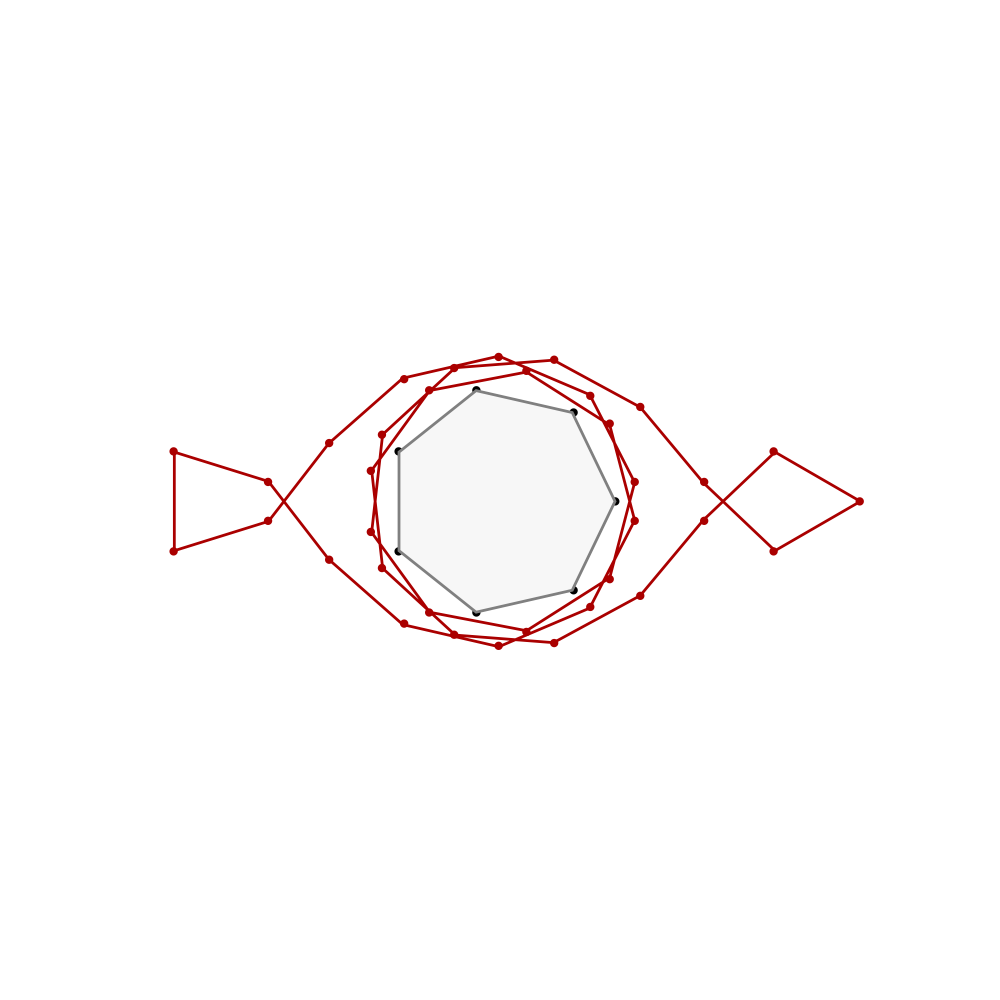}
	\end{minipage}
	\quad
	\begin{minipage}{0.25\textwidth}
		\includegraphics[width=\linewidth]{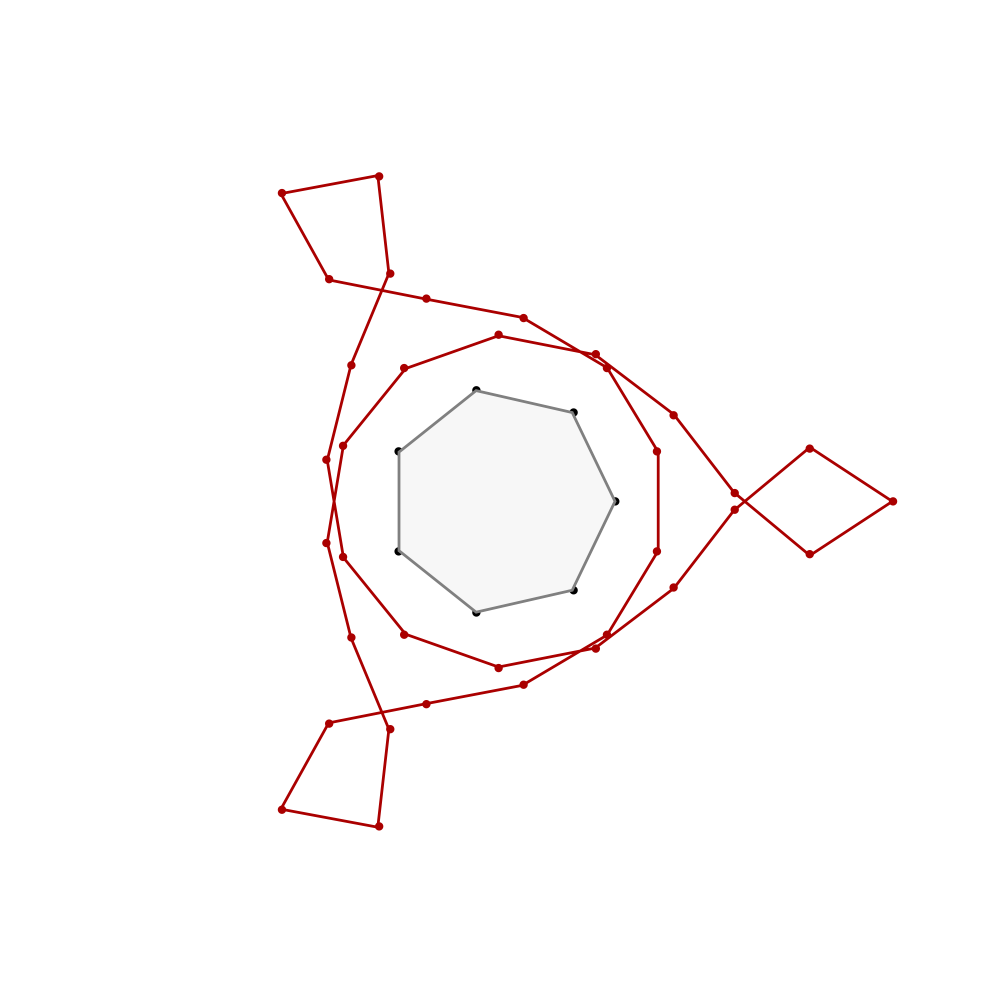}
	\end{minipage}
	\quad
	\begin{minipage}{0.25\textwidth}
		\includegraphics[width=\linewidth]{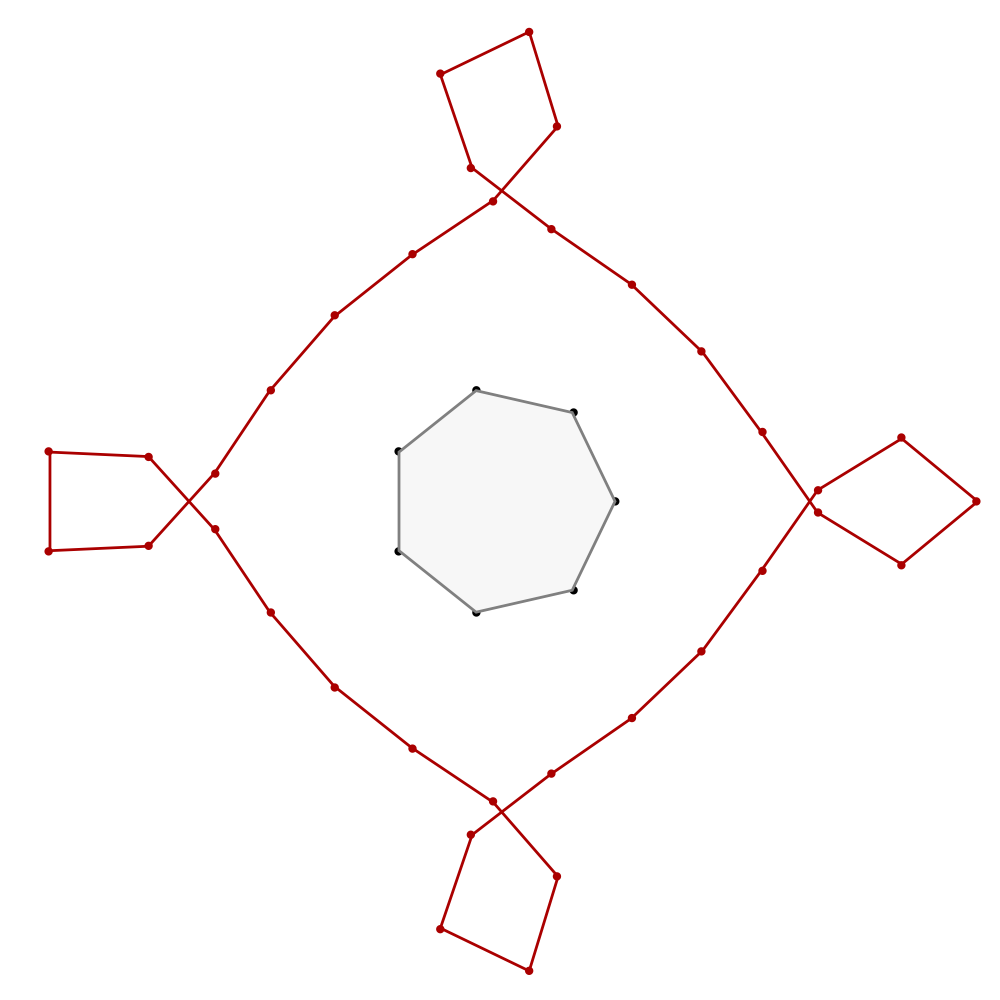}
	\end{minipage}
	\caption{Discrete circletons for $M = 7$, drawn with $\tau = \pi$ and $(k, \ell) = (1,2)$, $(1,3)$,
	 $(2,3)$, $(1,4)$, $(3,4)$, $(1,5)$, $(2,5)$, $(3,5)$, $(4,5)$ over $\ell$--fold cover of the circle.}
	\label{fig:discB1}
\end{figure}
\begin{example}\label{exam:disc2}
	In this example, we recover the discrete analogue of the smooth case in Example~\ref{exam:smoothB}: Consider the planar bicycle correspondences of the arc-length polarised discrete circle as in Example~\ref{exam:disc}, that is $x = e^{\frac{2 \pi \ii}{M} n} \in \mathbb{C}$ with $m_{ij} = | 1 - e^\frac{2\pi \ii}{M}|^{-2}$.
	Then the recurrence relation on the complex--valued function $a$ \eqref{eqn:diffPar2} becomes
		\[
			e^{-\frac{2\pi \ii}{M}} a_k - (1 + e^{-\frac{2\pi \ii}{M}}) a_j + (1 - \hat{\mu} | 1 - e^\frac{2\pi \ii}{M}|^2) a_i = 0
		\]
	on any three consecutive vertices $(ijk)$.
	Therefore, $a = c^- a^- + c^+ a^+$ where
		\[
			a_n^\pm = \left(\frac{1}{2}\left(e^\frac{2\pi \ii}{M} (1 \pm s) + (1 \mp s)\right)\right)^n
		\]
	so that $b_i = -(\dif{x}_{ij})^{-1} \dif{a}_{ij}$ implies
		\[
			b_n = -\frac{1}{2} e^{-\frac{2\pi \ii}{M}n}\left(c^- (1- s) a^- + c^+ (1 + s) a^+\right)
		\]
	for $s = \sqrt{1 - 4\mu}$.

\begin{figure}
	\centering
	\begin{minipage}{0.25\textwidth}
		\includegraphics[width=\linewidth]{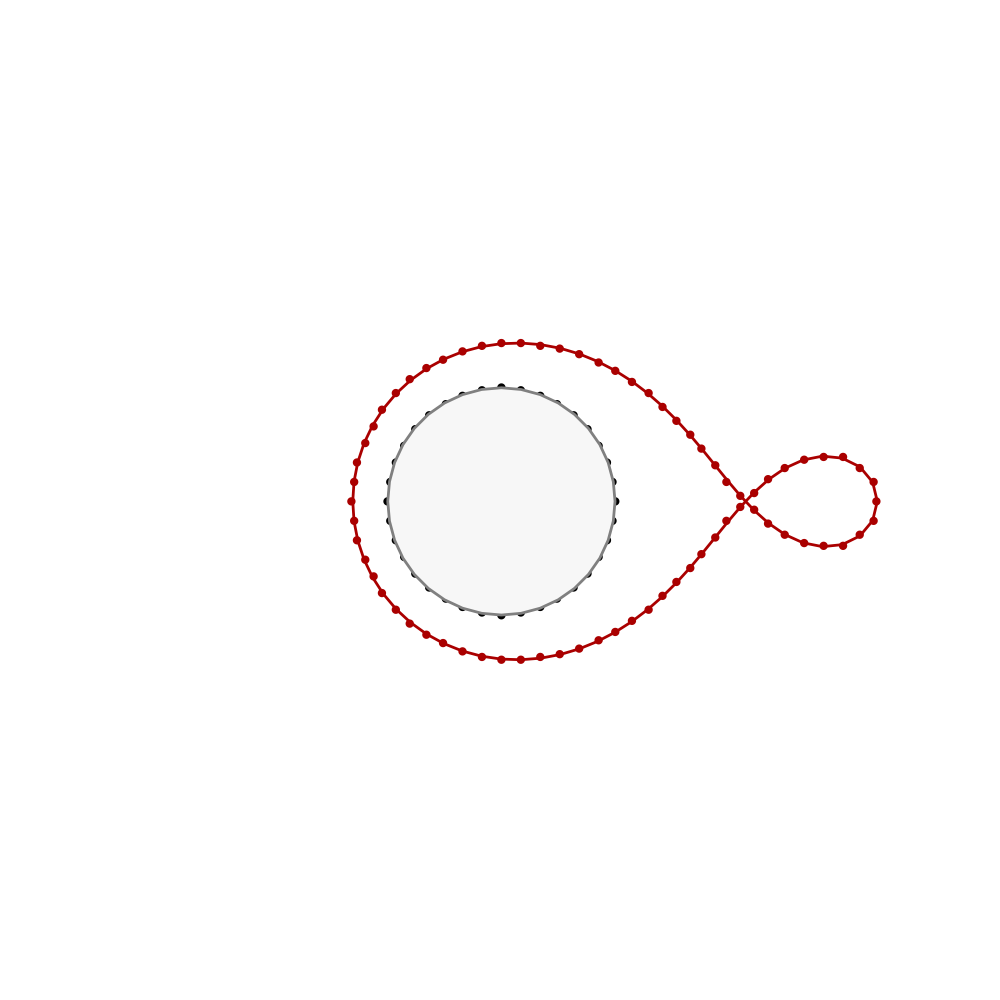}
	\end{minipage}
	\quad
	\begin{minipage}{0.25\textwidth}
		\includegraphics[width=\linewidth]{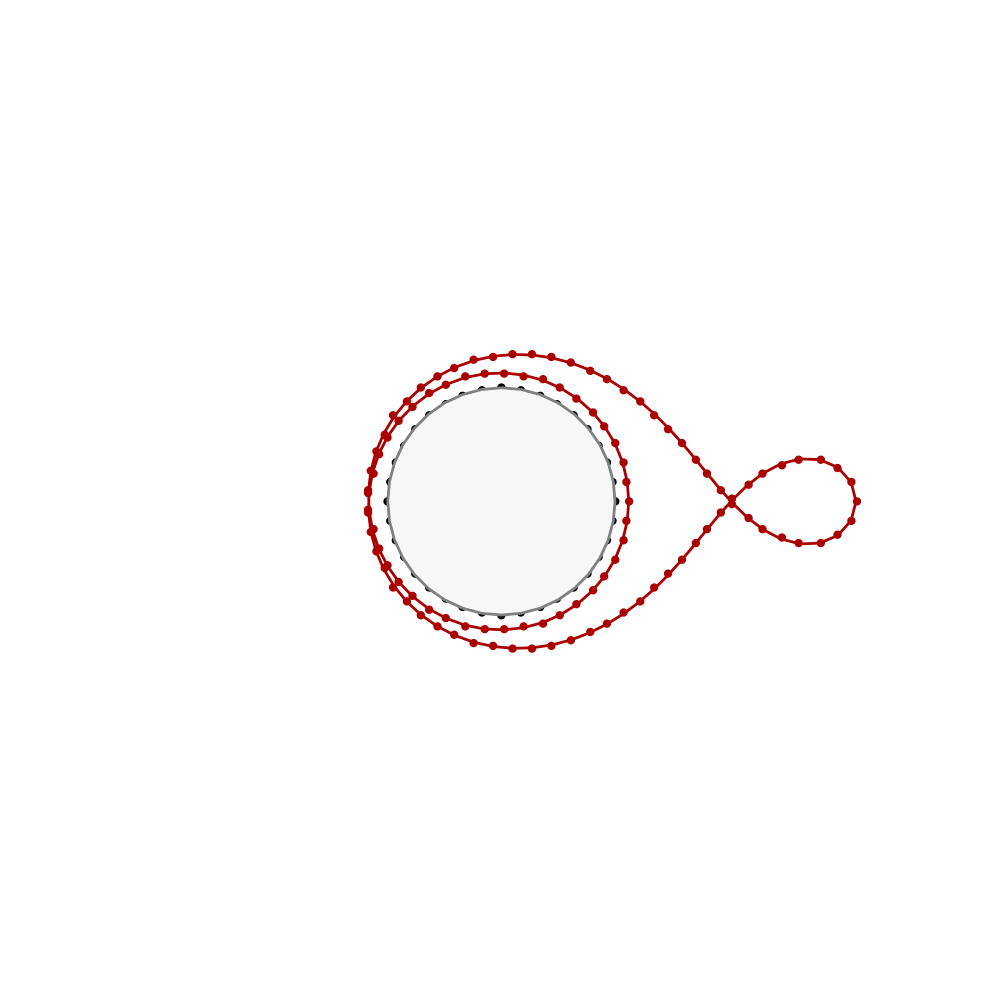}
	\end{minipage}
	\quad
	\begin{minipage}{0.25\textwidth}
		\includegraphics[width=\linewidth]{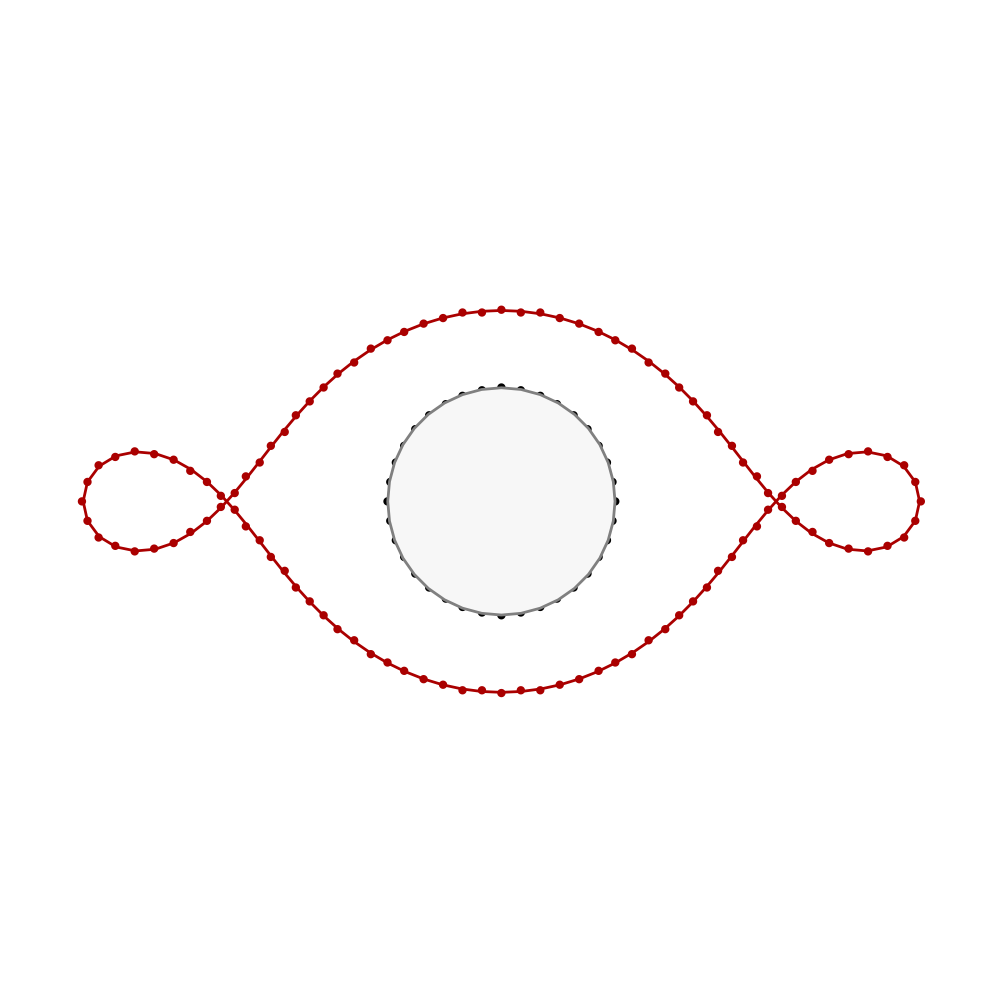}
	\end{minipage}
	\\
	\begin{minipage}{0.25\textwidth}
		\includegraphics[width=\linewidth]{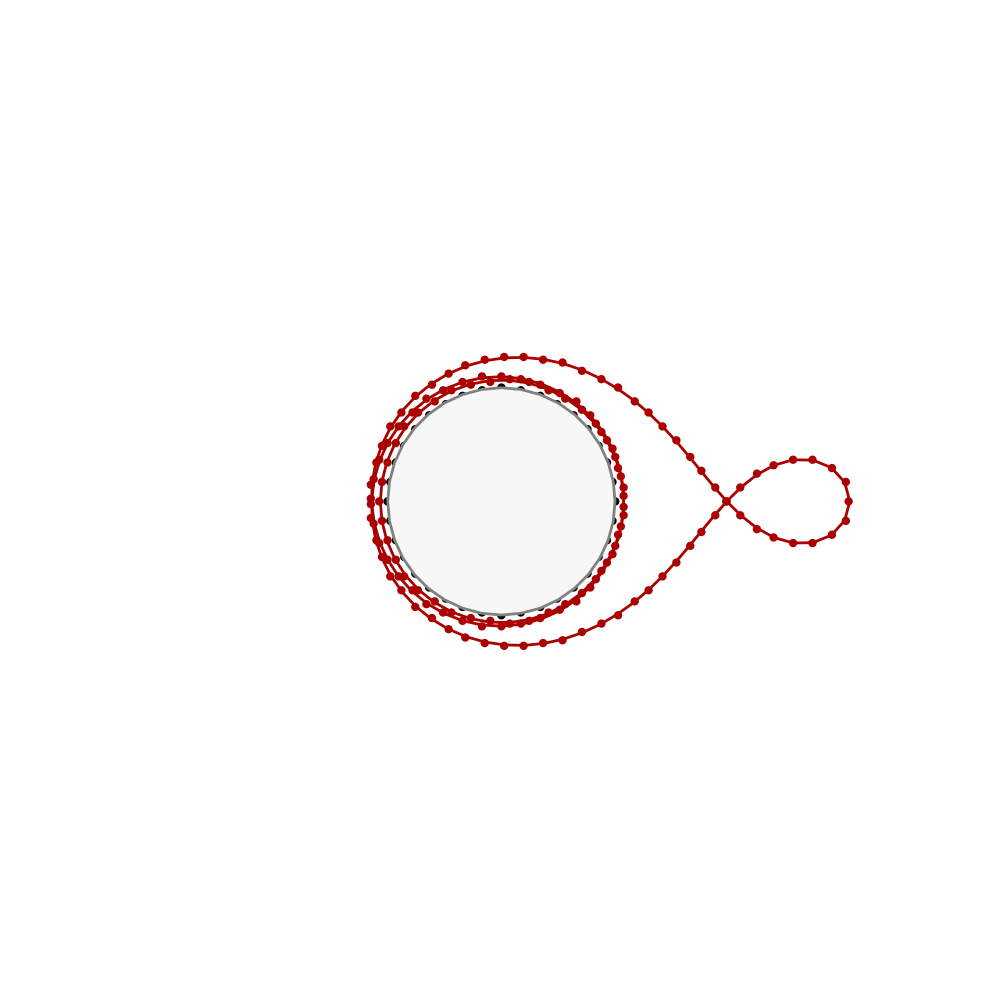}
	\end{minipage}
	\quad
	\begin{minipage}{0.25\textwidth}
		\includegraphics[width=\linewidth]{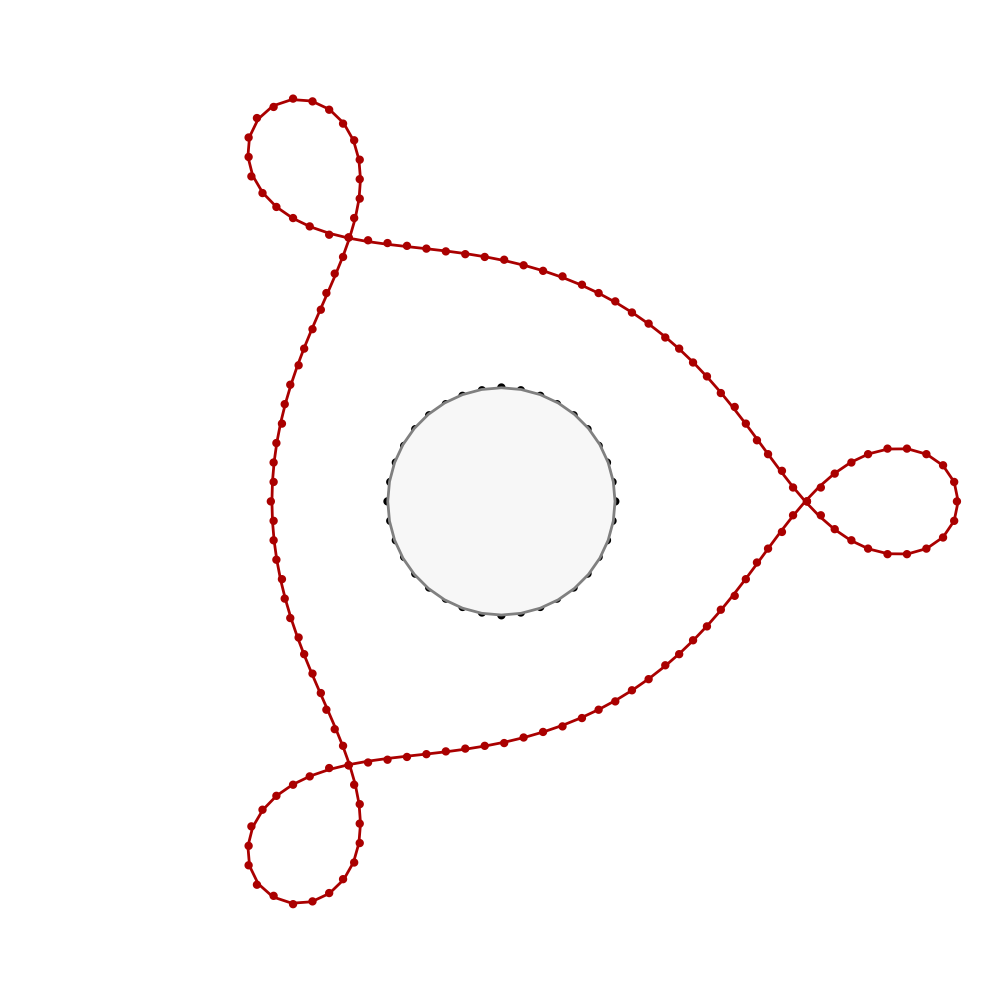}
	\end{minipage}
	\quad
	\begin{minipage}{0.25\textwidth}
		\includegraphics[width=\linewidth]{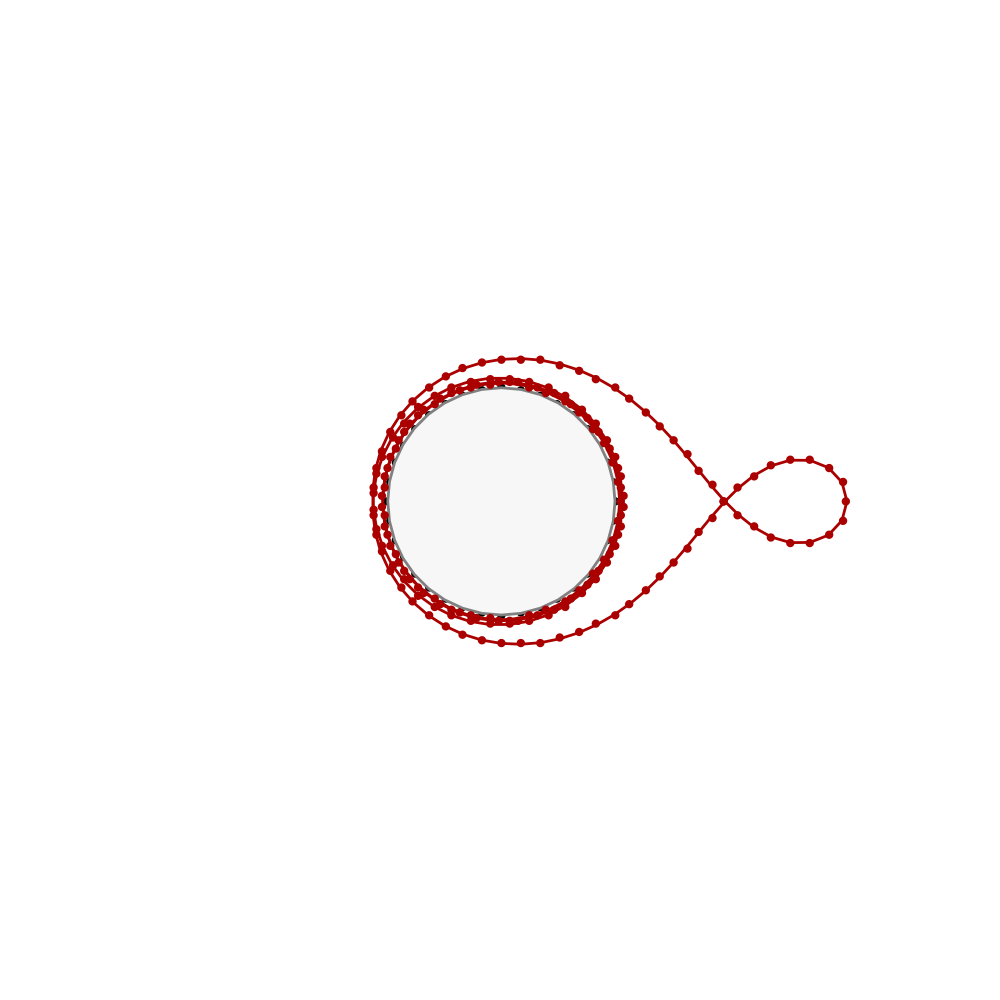}
	\end{minipage}
	\\
	\begin{minipage}{0.25\textwidth}
		\includegraphics[width=\linewidth]{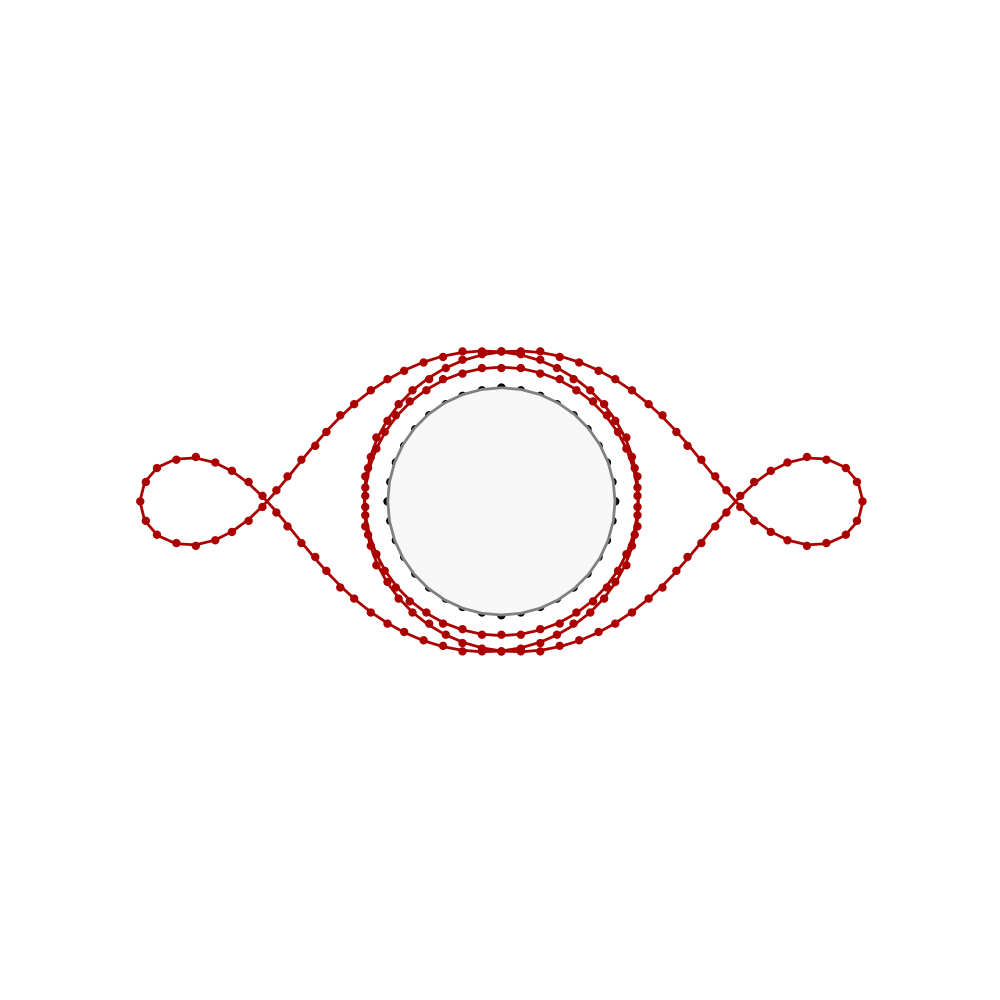}
	\end{minipage}
	\quad
	\begin{minipage}{0.25\textwidth}
		\includegraphics[width=\linewidth]{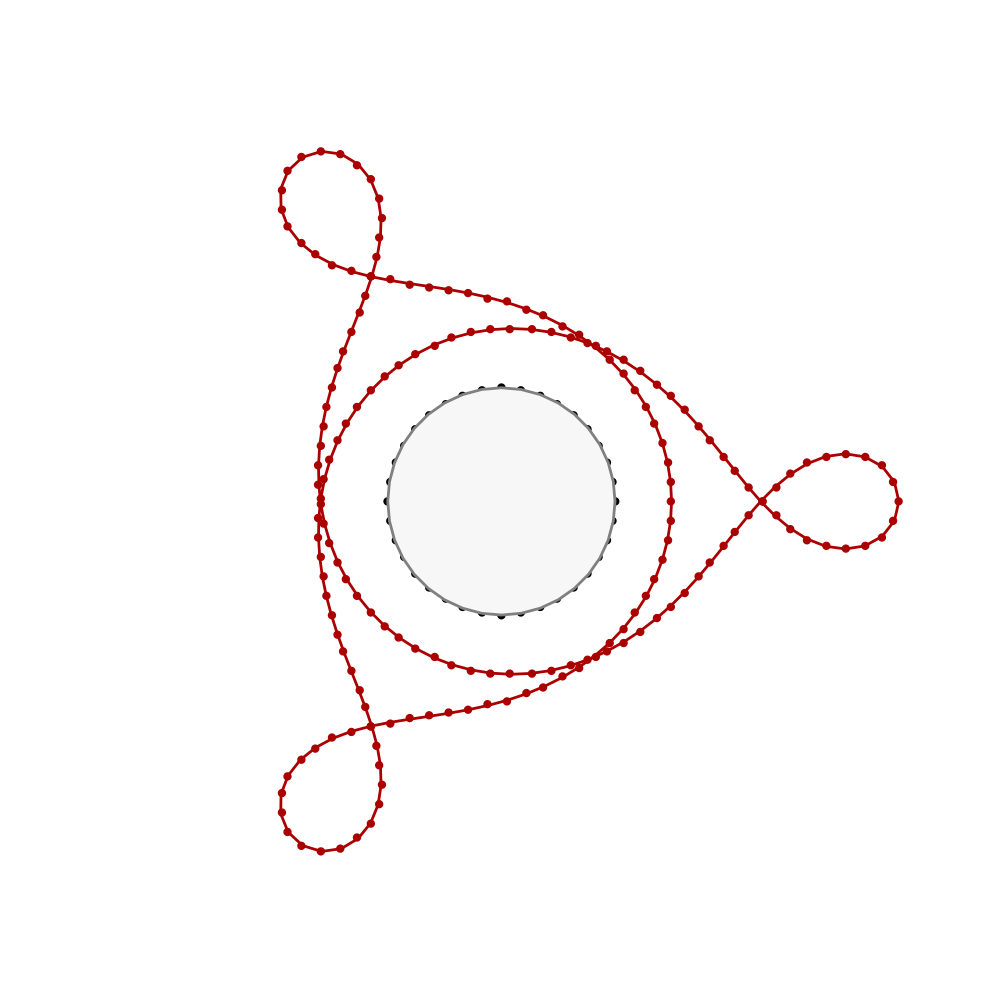}
	\end{minipage}
	\quad
	\begin{minipage}{0.25\textwidth}
		\includegraphics[width=\linewidth]{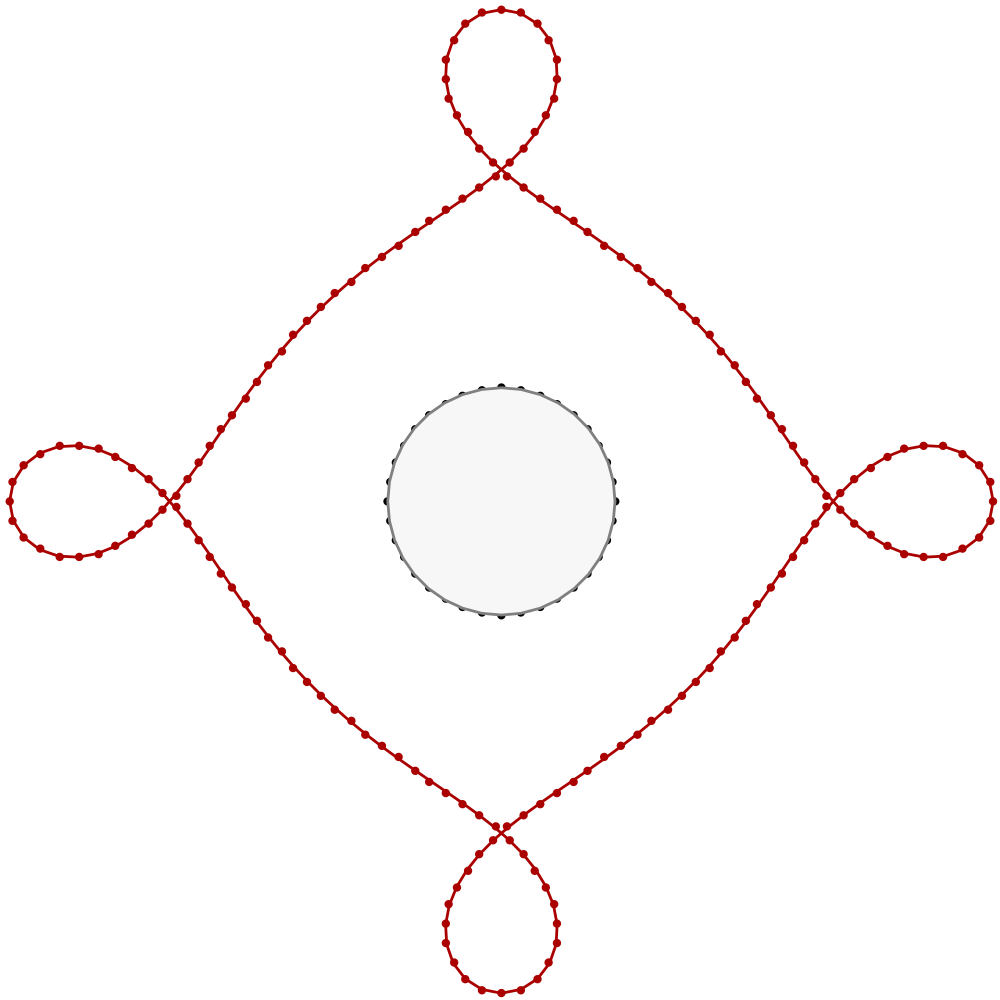}
	\end{minipage}
	\caption{Discrete circletons for $M = 36$, drawn with $\tau = \pi$ and $(k, \ell) = (1,2)$, $(1,3)$,
	 $(2,3)$, $(1,4)$, $(3,4)$, $(1,5)$, $(2,5)$, $(3,5)$, $(4,5)$ over $\ell$--fold cover of the circle.}
	\label{fig:discB2}
\end{figure}
	
	To find the Darboux transforms $\hat{x}$ that are also discrete arc-length polarised, we use Theorem~\ref{thm:discB} and require that $|ab^{-1}|^2 = \frac{1}{\mu}$ at $n = 0$, that is,
		\[
			|c^- + c^+|^2 = \frac{1}{4 \mu} |c^- (1-s) + c^+ (1+s)|^2
		\]
	implying that
		\[
			c^- + c^+ = \frac{e^{\ii \tau}}{2 \sqrt{\mu}} (c^- (1-s) + c^+ (1+s))
		\]
	for some $\tau \in \mathbb{R}$.
	Therefore, if $\mu \neq \frac{1}{4 \cos^2 \tau}$, we have $c^+ = \chi c^-$ with
		\[
			\chi = \frac{-2 \sqrt {\mu} + e^{\ii \tau}(1 - s)}{2\sqrt{\mu} - e^{\ii \tau}(1 + s)};
		\]
	otherwise, we have $c^- = 0$.
	Thus, all the planar bicycle transformations of the discrete circle are given by
		\[
			\hat{x}_n =\scalebox{0.85}{$\dfrac{-e^{\frac{2 \pi \ii}{M} n}\left(\chi (1 - s)(e^{\frac{2 \pi \ii}{M}}(1 + s) + (1 - s))^n +  (1+s)(e^{\frac{2 \pi \ii}{M}}(1 - s) + (1 + s))^n \right)}
				{ \chi  (1 + s)(e^{\frac{2 \pi \ii}{M}}(1 + s) + (1 - s))^n +(1-s)(e^{\frac{2 \pi \ii}{M}}(1 - s) + (1 + s))^n }$}.
		\]
	
	To find the \emph{discrete circletons}, i.e.\ closed bicycle correspondences of the discrete circle, we calculate the monodromy as before and require that $a$ is periodic over $\ell$--fold cover of $M$.
	Noting that $a^\pm_{n+\ell M} = a^\pm_n h^\pm$ where $h^\pm = a^\pm_{\ell M}$, we see that we have resonance points when $h^+ = h^-$, i.e.\
		\[
			\left(\frac{e^\frac{2\pi \ii}{M} (1 + s) + (1 - s)}{e^\frac{2\pi \ii}{M} (1 - s) + (1 + s)}\right)^{\ell M} = e^{2\pi \ii k}
		\]
	for some $k \in \mathbb{Z}$.
	Thus, for
		\[
			\mu =\frac{1}{4}\left(1 - \cot^2\frac{\pi}{M}\tan^2\frac{k \pi}{\ell M}\right).
		\]
	we obtain discrete circletons.
	Examples of discrete circletons with $M = 15$ was given in Figure~\ref{fig:discB3}; for other values of $M$, see Figures~\ref{fig:discB1} and \ref{fig:discB2}.
\end{example}

\section{Summary}

For a smooth integrable system, the introduction of the spectral parameter
allows one to consider an infinite system of linear partial differential
equations to solve partial differential equation of higher order. It
is this point of view which yields efficient discrete
models for higher order partial differential equations which form an
integrable system: rather than solving differential equations by typical numerical
methods, e.g., the Runge--Kutta method, discretising the integrable
system structure gives a solution
to the equation by recurrence equations, which can be easily
implemented and avoid singularities. 

In this paper we were concerned with  \emph{periodic} discrete
solutions of an integrable system.  Writing the system in terms of an
associated family of connections, new solutions to the underlying
partial differential equations are obtained from parallel sections of
the family of connections.  The question of finding periodic
smooth solutions can then be approached by finding parallel sections
with multipliers, Section~\ref{sect:two}. As an example of this
strategy, we provided a new
computational model for periodic, discrete (and smooth) solutions given by
the discrete Darboux
transformation for the special case of polarised curves, Section~\ref{sect:three}. The
integrable reductions of this system include the bicycle
correspondences, linking our results to the smoke ring flow, the
filament equations and the modified Korteweg-de Vries equations,
modeling shallow water waves.

In particular, we provided in Section~\ref{sect:three} all discrete, periodic planar and spatial
curves which are iso--spectral to the circle as well as all circletons
which preserve arc-length polarisations. The implementation with
quaternions is particularly relevant for obtaining explicit solutions to the recurrence equations in the case of discrete circles.

Our results serve as templates for generalisations to other integrable
systems, to provide computational models for periodic
solutions of problems in physics, chemistry and biology, for example
 in modelling shallow water waves, fiber optics applications, and
low-frequency collective motion in proteins and DNA.

\vspace{15pt}
\textbf{Acknowledgements.}
We are thankful to the referee for many indispensable comments.
We gratefully acknowledge the support from the Leverhulme Trust Network Grant IN-2016-019 and the JSPS Research Fellowships for Young Scientist 21K13799.



\end{document}